\documentclass{amsart}

\usepackage{amssymb}
\usepackage{mathbbol}
\usepackage{color}
\newcommand{\abs}[1]{\vert#1\vert}
\newcommand{\babs}[1]{\big\vert#1\big\vert}
\newcommand{\Babs}[1]{\Big\vert#1\Big\vert}
\newcommand{\Abs}[1]{\Vert#1\Vert}
\newcommand{\bAbs}[1]{\big\Vert#1\big\Vert}
\newcommand{\dsp}{\displaystyle}
\newcommand{\bU}{{\mathbf U}}
\newcommand{\tbU}{\widetilde{\mathbf U}}
\newcommand{\tU}{\widetilde{U}}
\newcommand{\bC}{{\mathbf C}}
\newcommand{\bA}{{\mathbf A}}

\newcommand{\ovV}{\overline{V}}

\newcommand{\uU}{\underline{U}}

\newcommand{\bom}{{\boldsymbol \omega}}

\newcommand{\uom}{\underline{ \omega}}
\newcommand{\bn}{{\mathbf n}}

\newcommand{\uV}{{\underline V}}
\newcommand{\uA}{{\underline A}}

\newcommand{\uw}{{\underline w}}

\newcommand{\Vp}{U_\parallel}
\newcommand{\Ap}{A_\parallel}

\newcommand{\Vpb}{U_{(\beta)\parallel}}

\newcommand{\R}{{\mathbb R}}
\newcommand{\N}{{\mathbb N}}

\newcommand{\dt}{\partial_t}
\newcommand{\dz}{\partial_z}
\newcommand{\cS}{{\mathcal S}}
\newcommand{\mfa}{{\mathfrak a}}
\newcommand{\mfN}{{\mathfrak m}^N}
\newcommand{\proj}{{\mathfrak P}}
\newcommand{\cE}{{\mathcal E}}
\newcommand{\curl}{\mbox{\textnormal{curl} }}
\newcommand{\curlm}{\mbox{\textnormal{curl}}^\mu}

\newcommand{\dive}{\mbox{\textnormal{div} }}
\newcommand{\curls}{\mbox{\textnormal{curl}}^\sigma}

\newcommand{\dives}{\mbox{\textnormal{div}}^\sigma}
\newcommand{\divem}{\mbox{\textnormal{div}}^\mu}
\newcommand{\surf}{{\vert_{\textnormal{surf}}}}
\newcommand{\bott}{{\vert_{\textnormal{bott}}}}
\newcommand{\surff}{{\vert_{z=0}}}
\newcommand{\bottf}{{\vert_{z=-H_0}}}
\newcommand{\bottff}{{\vert_{z=-1}}}
\newcommand{\eps}{\varepsilon}
\newcommand{\tpsi}{\widetilde{\psi}}
\newcommand{\cG}{{\mathcal G}}
\newcommand{\cGg}{{\mathcal G}_{\textnormal{gen}}}

\newcommand{\dx}{\partial_x}

\newcommand{\ep}{\varepsilon}

\newcommand{\pa}{\partial}

\newcommand{\Om}{\Omega}

\theoremstyle{plain}
\newtheorem{thm}{Theorem}[section]
\newtheorem{proposition}[thm]{Proposition}
\newtheorem{lemma}[thm]{Lemma}
\newtheorem{cor}[thm]{Corollary}

\theoremstyle{definition}
\newtheorem{defi}[thm]{Definition}

\theoremstyle{remark}
\newtheorem{remark}[thm]{Remark}
\newtheorem{nota}[thm]{Notation}

\newtheorem{rem}{Remark}

\title [Water waves with vorticity]{Well-posedness and shallow-water stability for a new
  Hamiltonian formulation of the water waves equations with
  vorticity}

\author{Angel Castro}
\address{Angel Castro\\Departamento de Matem\'aticas de la UAM, Instituto de Ciencias Matem\'aticas CSIC\\
Campus de Cantoblanco, 28049 Madrid, Spain }
\email{angel\_castro@icmat.es}

\author{David Lannes}
 \address{David Lannes \\ DMA, Ecole Normale Sup\'erieure et CNRS UMR 8553\\ 45 rue
  d'Ulm \\ 75005 Paris, France}
\email{David.Lannes@ens.fr}

\begin{document}

\begin{abstract}
In this paper we derive a new formulation of the water waves equations
with vorticity that generalizes the well-known Zalkarov-Craig-Sulem
formulation used in the irrotational case. We prove the local
well-posedness of this formulation, and show that it is formally
Hamiltonian. This new formulation is cast in Eulerian variable, and in
finite depth; we show that it can be used to provide uniform bounds on the lifespan
and on the norms of the solutions in the singular shallow water regime. As an
application to these results, we derive and provide the first rigorous
justification of a shallow water model for water waves in presence of
vorticity; we show in particular that a third equation must be added to the standard
model to recover the velocity at the surface from the averaged velocity.
\end{abstract}

\maketitle

\section{Introduction}

\subsection{General setting}

The equations governing the motion of the surface of a homogeneous,
inviscid fluid of density $\rho$ under the
influence of gravity (assumed to be constant and vertical, ${\bf
  g}=-g{\bf e}_z$, $g>0$) are known as the water waves equations, or free
surface Euler equations. In the case where the surface of the fluid is
delimited above by the graph of a function $\zeta(t,X)$ over its rest
state $z=0$ (with $t$ the time variable, $X\in \R^d$ the horizontal space
variables, and $z$ the vertical variable), and below by a flat bottom
$z=-H_0$, and denoting by $\bU$ and $P$ the velocity and pressure
fields, these equations can be written
\begin{eqnarray}\label{Eulereq}
\dt \bU+\bU\cdot \nabla_{X,z}\bU&=&-\frac{1}{\rho}\nabla_{X,z}P -g {\bf e}_z,\\
\label{incomp}
\nabla_{X,z}\cdot \bU&=&0,
\end{eqnarray}
in the fluid domain $\Omega_t=\{(X,z)\in \R^{d+1},
-H_0<z<\zeta(t,X)\}$; they are complemented with the boundary conditions
\begin{eqnarray}
\label{kineticeq}
\dt \zeta -\bU_{\surf}\cdot N&=&0 \quad(\mbox{with }N=(-\nabla\zeta^T,1)^T),\\
P_\surf&=&\mbox{constant}
\end{eqnarray}
at the surface, and
\begin{equation}
\label{condbott}
\bU_\bott\cdot N_b=0\quad (\mbox{with }N_b=\bf e_z)
\end{equation}
at the bottom. In many physical situations, the motion of the fluid is
in addition irrotational, and another equation can be added to
\eqref{Eulereq}-\eqref{condbott}, namely
$$
\curl \bU=0 \quad\mbox{ in }\quad \Omega.
$$
This additional assumption yields considerable simplifications since
all the relevant information to describe the fluid motion is then
concentrated on the interface. This can be exploited in many ways. See
for instance \cite{Lindblad,CoutandShkoller,AmbroseMasmoudi,ShatahZeng,CCF} for local well-posedness results and \cite{CCF,CCFGGpenas,CCFGG,CSS} for the existence of turning waves and splash singularities. Let us describe
briefly here the approach initiated by Zakharov \cite{Zakharov} and
Craig-Sulem \cite{CS} which is one of the most seminal and the starting
point for the present paper.\\
From the irrotationality assumption, one can infer the existence of a
velocity potential $\Phi$ such that $\bU=\nabla_{X,z}\Phi$; the
incompressibility conditions implies that $\Phi$ be harmonic, and the
bottom boundary condition that its normal derivative vanishes at the
bottom. It follows that $\Phi$ is fully determined by its trace at the
surface $\Phi_\surf=\psi$. Zakharov noticed that the irrotational
equation could be reduced to the Hamiltonian equation
$$
\dt \left(\begin{array}{c}\zeta \\ \psi\end{array}\right)=J \mbox{\rm
  grad}_{\zeta,\psi}H,
\quad\mbox{ with }\quad J=\left( \begin{array}{cc} 0 & 1 \\ -1 & 0\end{array}\right),
$$
and where $H$ is the total energy $H=\frac{1}{2}\int_{\R^d}
g\zeta^2+\frac{1}{2}\int_\Omega \abs{\bU}^2$. Introducing the
Dirichlet-Neumann operator
$\cG[\zeta]\psi=\sqrt{1+\abs{\nabla\zeta}}\partial_n \Phi_\surf$,
Craig and Sulem rewrote these equations as a closed set of two
evolution equations on $\zeta$ and $\psi$,
$$
\left\lbrace
\begin{array}{l}
\dsp \dt \zeta-\cG[\zeta]\psi=0,\\
\dsp \dt \psi +g \zeta+\frac{1}{2}
 \babs{\nabla\psi}^2-\frac{\big(\cG[\zeta]\psi+\nabla\zeta\cdot\nabla\psi\big)^2}{2
(1+\abs{\nabla\zeta}^2)}=0.
\end{array}\right.
$$
This is a convenient formulation for well-posedness issues
(see for instance \cite{LannesJAMS} for local well-posedness,
\cite{ABZ_Duke,ABZ_ENS,ABZ_Inv,MingZhang} for low regularity solutions,
\cite{GMS,IonPus,AD} for global existence, etc.). It has also been
used for numerical computations \cite{CS,GN}, weak turbulence modeling
\cite{ZDP}, analysis of periodic wave patterns or solitary waves
\cite{IoossPlotnikov,AlazardMetivier,RoussetTzvetkov}, etc. More relevant to our
present motivations, it is probably the most commonly used approach to
derive and justify asymptotic models describing the solutions
to the water waves equations in various physical regimes (e.g. shallow
water or deep water). The derivation of such models follows directly
from an asymptotic expansion of the Dirichlet-Neumann operators
(e.g. \cite{CSS1,CG,BCL,Chazel,Iguchi2,LS,MingZZ}), while the key point for
their justification is a local well posedness theorem for a
dimensionless version of the equations, and over a time scale whose
dependance on the various dimensionless parameters is controlled; in the shallow
water regime --  of great importance for applications in oceanography --
this induces an extra difficulty because the shallow water limit is
singular; the relevant existence results have been shown in
\cite{AL,Iguchi1} (see also \cite{Craig,SW,Iguchi0,Li,TW} for the
justification of various asymptotic models using other approaches). We
refer to \cite{Lannes_book} for a more comprehensive description of these aspects.

However, a limitation of the Zakharov-Craig-Sulem approach is that it
is restricted to irrotational flows. This is a relevant framework for
most applications in oceanography, but several important phenomena
such as rip currents for instance can only be understood by taking
into account vorticity effects. Rip currents are only one particular
example of wave-currents interactions; the understanding of the energy
exchanges at stake in such interactions is an important challenge
in oceanography, and vorticity is one of the key mechanisms
involved. Several asymptotic models have been derived in the physics
literature to take into account vorticity effects in shallow water
models; however, these derivations rely on assumptions on the
structure of the flow (e.g. columnar motion)
that are in general not satisfied, or only at a very low order of
precision. In particular, there does not exist to our knowledge any good
description of the nonlinear dynamics of the vorticity in the shallow water
regime; it is for instance not known whether horizontal vorticity may
be created from vertical vorticity.\\
From a mathematical viewpoint, various authors considered the local
well-posedness theory for the water waves problem in presence of
vorticity
\cite{Lindblad,CoutandShkoller,ShatahZeng,ZhZh,MasmoudiRousset}; these
results are however not adapted to answer the above preoccupations
because they consider different physical configurations (drop of
fluid, infinite depth, etc.) and use mathematical techniques that make
the singular shallow water limit very delicate to handle (see for
instance the comments of \cite{Lannes_criterion} on the incompatibility of standard symbolic
analysis and of the shallow water limit); moreover, the influence of
vorticity on the flow is generally treated implicitly. The rigorous
qualitative analysis of water waves seems to have essentially
been restricted, when vorticity is present, to one dimensional surfaces,
and to periodic or standing waves; we refer to the recent book
\cite{Constantin} for an extensive review and more references on these aspects.
The motivation
for the present work is therefore twofolds:
\begin{enumerate}
\item Find a formulation of the water waves equations allowing for
  the presence of vorticity, adapted to the physical configurations we
  have in mind for applications to oceanography, and making the
  influence of the vorticity on the flow as explicit as possible.
\item Show the well-posedness of this formulation, and control the
  life span and the size of the solution in the so called shallow
  water limit, in order  to pave the way for the derivation of shallow water
  models in presence of vorticity.
\end{enumerate}
These goals are achieved by deriving first a generalization of the above classical
Zakharov-Craig-Sulem formulation, as a set of three evolution
equations on $(\zeta,\psi,\bom)$, where $\bom=\curl \bU$ is the
vorticity. Of course, $\psi$ cannot be defined as in the irrotational
case as the trace at the surface of the velocity potential $\Phi$;
instead, we define $\nabla\psi$ as the projection onto gradient vector
fields of the horizontal component of the tangential velocity $\Vp$ at
the surface. The equations then read
$$
\left\lbrace
\begin{array}{l}
\dsp \dt \zeta+\uV\cdot \nabla\zeta-\uw=0,\\
\dsp \dt \psi +g \zeta+\frac{1}{2}
 \babs{\Vp}^2-\frac{1}{2}(1+\abs{\nabla\zeta}^2)\uw^2-\frac{\nabla^\perp}{\Delta}\cdot\big(
\uom\cdot N \uV\big)=0,\\
\dsp \dt \bom+\bU\cdot\nabla_{X,z}
\bom=\bom\cdot\nabla_{X,z} \bU
\end{array}\right.
$$
where $\uV$ and $\uw$ denote the horizontal and vertical components of
the velocity at the interface, and $\uom=\bom_\surf$. In the
irrotational case ($\bom=0$), these equations coincide exactly with
the ones derived by Zakharov-Craig-Sulem.
We show that it is a closed set
of equations and establish local well-posedness. We also show that
this new formulation is formally Hamiltonian (with a non-canonical
Poisson bracket).\\
As the classical
irrotational Zakharov-Craig-Sulem formulation, but contrary to the aforementioned works in
the rotational case, our equations are cast in Eulerian variables, in
a configuration which is relevant to applications in oceanography
(finite depth), and can easily be used to derive asymptotic models.
We consider in particular the shallow water regime, which is of great
importance in oceanography. We first rewrite the equation in
dimensionless form, and prove the existence time is of order
$\frac{T}{\eps}$, with $T$ independent of $\eps,\mu\in (0,1)$, where
$$
\eps=\frac{\mbox{typical amplitude of the wave}}{\mbox{depth}}
\quad\mbox{ and }\quad
\mu=\Big(\frac{\mbox{depth}}{\mbox{typical horizontal length}}\Big)^2;
$$
this estimate on the lifespan of solutions goes with uniform bounds on
the solutions, which make possible the derivation of asymptotic
models. As an illustration, we provide the first rigorous derivation
and justification of the nonlinear shallow water
equations with vorticity, which is an approximation of order $O(\mu)$
of the water waves equations when $\eps=O(1)$. The more technical
derivation of a $O(\mu^2)$ model of Green-Naghdi type is
left for the companion paper \cite{CastroLannes}; let us just mention
that the shallow water nonlinear dynamics of the vorticity studied
here allow us to exhibit a so far unknown mechanism of generation of
horizontal vorticity from purely vertical rotational effects. More
generally, we expect that the asymptotic bounds derived in this paper
will be of great use to assess the validity of the physical
assumptions made in the physics literature to formally derive
asymptotic shallow water models in presence of vorticity.

\medbreak

The paper is organized as follows. Section \ref{sect2} is devoted to
the derivation of our new formulation of the water waves with
vorticity (see above). The fact that this formulation is a closed set
of equations follows from the resolution of a div-curl problem
allowing to reconstruct the velocity $\bU$ in the fluid domain from
$\zeta$, $\psi$ and $\bom$. This div-curl problem is studied in full
details in
Section \ref{sect3}. The local well-posedness is then addressed in
Section \ref{sect4}: the equations are ``quasilinearized'' and {\it a
  priori} estimates are derived. Using these estimates, a solution is
then constructed by an iterative scheme (which is non trivial due to
the fact the the vorticity is defined on a domain that depends on the
surface elevation). The main result is then given in Theorem
\ref{theo1}.\\
The proof of Theorem \ref{theo1} has been
tailored to allow its implementation in the shallow water
setting. However, handling the shallow water limit induces some
difficulties that are not relevant for a standard local
well-posedness result such as Theorem \ref{theo1} (the control of the
bottom vorticity for instance). For the readers who are not interested
in the shallow water analysis, we have therefore opted to treat this
aspect separately. This is done in Section \ref{sect5}. The
first step is to write a dimensionless version of the equations; the
associated well-posedness result is then given in Theorem
\ref{theomainND}.  Note that the
non dimensionalization of the vorticity is not obvious, but that this theorem
justifies {\it a posteriori} the choice we have made. As an
application of this result, we derive and justify in \S \ref{sectSWvort} a first
order nonlinear shallow water model in presence of vorticity.\\
Finally, we investigate in Section \ref{sect6} the Hamiltonian
structure of our new formulation of the water waves equations with vorticity.

\subsection{Notations}

- $X=(x,y)\in \R^2$ denotes the horizontal variables. We also denote  by $z$ the vertical variable.\\
- $\nabla$ is the gradient with respect to the horizontal variables;
$\nabla_{X,z}$ is the full three dimensional gradient operator. The
curl and divergence operators are defined as
$$
\curl {\bf A}=\nabla_{X,z}\times {\bf A}\quad\mbox{ and }\quad \dive
{\bf A}=\nabla_{X,z}\cdot {\bf A}.
$$
- We denote by $d=1,2$ the horizontal dimension. When $d=1$, we often
identify functions on $\R$ as functions on $\R^2$ independent of the
$y$ variable. In particular, when $d=1$, the gradient, divergence and
curl operators take the form
$$
\nabla_{X,z}f=\left(\begin{array}{c}\partial_x f \\ 0 \\ \partial_z
    f \end{array}\right),
\quad
\curl {\bf A}=\left(\begin{array}{c}-\partial_z {\bf A}_2 \\ \partial_z
    {\bf A}_1-\dx {\bf A}_3\\ \partial_x {\bf A}_2
     \end{array}\right),\quad
\dive {\bf A}=\partial_x {\bf A}_1+\dz {\bf A}_3.
$$
- $\cS$ is the flat strip $\R^d\times (-H_0,0)$.\\
- We denote by $(X,\zeta(t,X))$ a parametrization of the free surface at
time $t$ and by $\Omega_t$  the fluid domain delimited at time $t$ by
this free surface and a flat bottom at depth $z=-H_0$,
$$
\Omega_t=\{(X,z)\in \R^{3}, -H_0< z <\zeta(t,X)\};
$$
when the dependence on time is not important, we just write $\Omega$
instead of $\Omega_t$.\\
- When $d=1$, $\Omega$ is invariant along the $y$ axis, and we
identify it with a two-dimensional domain; in particular,
\begin{align*}
\Abs{f}_{L^2(\Omega)}&=\int_{\R}\int_{-H_0}^{\zeta(x)} \abs{f(x,z)}^2dzdx
&\mbox{ if }d=1, \\
\Abs{f}_{L^2(\Omega)}&=\int_{\R^2}\int_{-H_0}^{\zeta(x,y) }\abs{f(x,,yz)}^2dzdxdy &\mbox{ if }d=2.
\end{align*}
- We write $\bU$ the velocity field; its horizontal component
 is written ${\mathbf V}$, and its vertical component
${\mathbf w}$.\\
- For a vector $\bA\in \R^{3}$ we often denote by $\bA_h$ its
horizontal component and by $\bA_v$ its vertical component.\\
- If ${\bf A}$ is a vector field defined on $\Omega_t$, we write
$\underline{A}$ the function
$$
\underline{A}(t,X)={\bf A}_\surf(t,X)={\bf A}(t,X,\zeta(t,X));
$$
consistently, if ${\bf A}$ is defined on the flat strip $\cS$ then $\uA(t,X)={\bf
  A}(t,X,0)$. We also denote by $A_b$ its trace at the bottom
$$
{A_b}(t,X)={\bf A}_\bott(t,X)={\bf A}(t,X,-H_0).
$$
- $\bn$ is the unit upward normal vector at the surface,
$\bn=N/\abs{N}$, with $N=(-\nabla \zeta,1)^T$.\\
- $\bn_b$ is the {\it upward} normal vector at the (flat)
bottom, $\bn_b=N_b={\bf e}_z$.\\
- If ${\bf V}\in \R^d$, we write ${\bf V}^\perp=(-{\bf V}_2,{\bf V}_1)^\perp$.\\
- For all vector field ${\bf A}$ defined on $\Omega$ and with values
in $\R^3$,  let us define $\Ap\in \R^2$ as the horizontal component of
the tangential part of ${\bf A}$ at the surface,
\begin{equation}\label{notapar}
A_\parallel=\uA_h+\uA_v\nabla\zeta,
\end{equation}
so that $\underline{A}\times
N=\left(\begin{array}{l}-\Ap^\perp\\-\Ap^\perp\cdot
    \nabla\zeta\end{array}\right)$.\\
- We always use simple bars to denote functional norms on $\R^d$ and
double bars to denote functional norms on the $d+1$ dimensional
domains $\Omega$ and $\cS$; for instance
$$
\abs{f}_p=\abs{f}_{L^p(\R^d)},\quad
\abs{f}_{H^s}=\abs{f}_{H^s(\R^d)},\quad
\Abs{f}_p=\Abs{f}_{L^p(\Omega)} \quad(\mbox{or }\Abs{f}_{L^p(\cS)}), \mbox{ etc.}
$$
- We use the Fourier multiplier notation
$$
f(D)u={\mathcal F}^{-1}(\xi\mapsto f(\xi)\widehat{u}(\xi))
$$
and denote by $\Lambda=(1-\Delta)^{1/2}=(1+\abs{D})^{1/2}$ the
fractional derivative operator.\\
-We define, for
all $s\in \R$, $k\in \N$ the space $H^{s,k}=H^{s,k}(\cS)$ by
\begin{equation}\label{defHsk}
H^{s,k}=\bigcap_{j=0}^k H^j((-H_0,0);H^{s-j}(\R^d)),\quad \mbox{ with }\quad
\Abs{u}_{H^{s,k}}=\sum_{j=0}^k \Abs{\Lambda^{s-j}\dz^j u}_2.
\end{equation}
- We shall have to handle functions whose gradient are in some Sobolev
space, but which are not in $L^2(\R^d)$. We therefore introduce the
Beppo-Levi spaces \cite{DenyLions}
$$
\forall s\geq 0,\qquad \dot{H}^s(\R^d)=\{f\in L^2_{{\rm
    loc}}(\R^d),\nabla f\in H^{s-1}(\R^d)^2\};
$$
similarly, for functions defined on the fluid domain $\Omega$, we write
$$
\forall k\in \N^*,\qquad \dot H^k(\Omega)=\{f\in L^2_{\rm
  loc}(\Omega),\nabla_{X,z} f\in H^{k-1}(\Omega)^{3}\}.
$$
- The ``dual spaces'' are $H_0^s(\R^d)$ defined as
\begin{equation}\label{defH0m}
 H^{s}_0(\R^d)=\{u\in H^{s}(\R^d)\,:\, \exists
v\in H^{s+1}(\R^d),\, u=|D|v\},
\end{equation}
and we write $\abs{u}_{H_0^{s}}=\abs{\frac{1}{\abs{D}}u}_{H^{s+1}}.$\\
- We write $s\vee t=\max\{s,t\}$.\\
- We generically denote by $C(\cdot)$ some positive function that has
a nondecreasing dependance on its arguments.\\
- We write $[\partial^\alpha,f,g]$ for the symmetric commutator
 $[\partial^\alpha,f,g]=\partial^\alpha(fg)-\partial^\alpha fg
 -f\partial^\alpha g$.\\

\section{A new formulation for the equations}\label{sect2}

\subsection{A first reduction}
Taking the trace of \eqref{Eulereq} at the free surface and then
taking the vectorial product of the resulting equation with $N$, one
obtains, with the notation \eqref{notapar},
$$
-\dt \Vp^\perp -g\nabla^\perp\zeta-\frac{1}{2}\nabla^\perp \abs{\Vp}^2+\frac{1}{2}\nabla^\perp\big(
(1+\abs{\nabla\zeta}^2)\uw^2\big)+\uom\cdot N\uV=0,
$$
where we also  used the notations
$$
\bom=\curl \bU,\qquad \uom=\bom_\surf;
$$
one gets therefore
\begin{equation}\label{eqVp}
\dsp \dt \Vp +g\nabla\zeta+\frac{1}{2}\nabla \abs{\Vp}^2-\frac{1}{2}\nabla\big(
(1+\abs{\nabla\zeta}^2)\uw^2\big)+\uom\cdot N\uV^\perp=0.
\end{equation}
Denoting by $\Pi$ the projector onto gradient vector fields, and
$\Pi_\perp$ the projector onto orthogonal gradient vector fields,
$$
\Pi=\frac{\nabla\nabla^T}{\Delta},\qquad \Pi_\perp=\frac{\nabla^\perp (\nabla^\perp)^T}{\Delta},
$$
we can decompose $\Vp$ under the form
\begin{eqnarray*}
\Vp&=&\Pi \Vp+\Pi_\perp \Vp\\
&=& \nabla\psi+\nabla^\perp\tpsi,
\end{eqnarray*}
for some scalar functions $\psi$ and $\tpsi$,
and similarly
\begin{eqnarray*}
\uom\cdot N\uV^\perp&=&\Pi \big( \uom\cdot N\uV^\perp\big)+\Pi_\perp \big(\uom\cdot N\uV^\perp\big)\\
&=& \nabla \big[\frac{\nabla}{\Delta}\cdot\big( \uom\cdot N\uV^\perp\big)\big]+\Pi_\perp \big(\uom\cdot N\uV^\perp\big).
\end{eqnarray*}
Applying $\Pi$ to the equation on $\Vp$ one therefore finds the following
equation on $\psi$,
\begin{equation}\label{4eq1}
 \dt \psi +g \zeta+\frac{1}{2} \abs{\Vp}^2-\frac{1}{2}\big(
(1+\abs{\nabla\zeta}^2)\uw^2\big)+\frac{\nabla}{\Delta}\cdot\big( \uom\cdot N\uV^\perp\big)=0.
\end{equation}
There is no need to derive an equation on $\Pi_\perp \Vp=\nabla^\perp\tpsi$ since this
component of $\Vp$ is fully determined
by the knowledge of $\bom$
and $\zeta$; indeed, using the differential identity
\begin{align}\label{iden}
\left(\nabla\times {\bf A}\right)_\surf\cdot N=\nabla^\perp \cdot A_{\parallel},
\end{align} one computes easily that
$$
\uom\cdot N=\nabla^\perp\cdot \Vp,
$$
and therefore
$$
\Pi_\perp \Vp=\nabla^\perp \tpsi,
$$
where $\tpsi$ is the unique solution\footnote{We assume here that
  $\bom\in L^2(\Omega)$ and divergence free; see Lemma \ref{existtildepsi}
  for the existence and uniqueness of $\tpsi$.}   in the Beppo-Levi space $\dot{H}^{3/2}(\R^d)$ of
$$
\Delta \tpsi=\uom\cdot N.
$$

Taking now the curl of \eqref{Eulereq}  we classically obtain the
vorticity equation
\begin{equation}\label{4eq3}
\dt \bom+\bU\cdot\nabla_{X,z} \bom=\bom\cdot\nabla_{X,z} \bU
\quad\mbox{in }\quad \Omega_t,
\end{equation}
with $\bom=\curl \bU$.

\medbreak

Our claim is that the kinematic equation \eqref{kineticeq}, together
with \eqref{4eq1} and \eqref{4eq3}, forms a closed system of equations
on $(\zeta,\psi,\bom)$. We have therefore to prove that these
quantities fully determine the velocity field in the whole fluid
domain. This is done in the next subsection.

\begin{rem}\label{remmm}
Applying $\Pi_\perp$ to the equation on $\Vp$ does not bring any
further information; this
  leads to $$
 \dt \Pi_\perp \Vp +\Pi_\perp\big(\uom\cdot N\uV^\perp)=0,
$$
and therefore
$$
 \dt (\nabla^\perp\cdot \Vp) +\nabla\cdot\big(\uom\cdot N\uV)=0.
$$
Evaluating the vorticity equation at the surface, we also get
$$
\dt \uom+\uV\cdot \nabla\uom=\uom\cdot \nabla \uU+\uom\cdot N \dz \bU_\surf,
$$
from which one readily deduces that
$$
\dt (\uom\cdot N)+\nabla\cdot (\uom\cdot N \uV)=0.
$$
We therefore get
$$
\dt (\uom\cdot N-\nabla^\perp\cdot \Vp)=0,
$$
which is always true since, as seen above, $\uom\cdot
N=\nabla^\perp\cdot U_\parallel$.
\end{rem}

\begin{rem}\label{vorbot}We have an analogous equation to \eqref{iden}
  at the bottom, which yields the following relation for the bottom vorticity,
$$\bom_\bott\cdot N_b=\nabla^\perp \cdot V_\bott;$$
in particular, if ${\bf V}\in H^1(\Omega)^2$ then $\omega_b\cdot
N_b\in H^{-\frac{1}{2}}_0(\R^d)$ with $H^{-\frac{1}{2}}_0(\R^d)$ as
defined in \eqref{defH0m}.
\end{rem}

\subsection{A div-curl problem}\label{sectdivcurl1}

Let $\zeta\in W^{2,\infty}(\R^d)$ and denote by $\Omega$ the associated
fluid domain,
\begin{equation}\label{defdomain}
\Omega=\{(X,z)\in \R^{d+1}, -H_0< z <\zeta(X)\};
\end{equation}
we assume that the fluid domain is strictly connected in the sense that
\begin{equation}\label{hmin}
\exists h_{\min},\quad \forall X\in \R^d, \qquad H_0+\zeta(X)\geq
h_{\min}.
\end{equation}
From the discussion of the previous section, one has
$$\Pi_\perp\Vp=\nabla^\perp \tpsi=\nabla^\perp \Delta^{-1}(\uom\cdot N)
$$
and is therefore fully
determined by the knowledge of (the normal component at the surface
of) the vorticity $\bom$. The following theorem shows that it is possible to reconstruct
the whole velocity field $\bU$ in the fluid domain in terms $\bom$,
$\Pi \Vp=\nabla\psi$ and $\zeta$; more precisely, there is a unique solution $\bU\in
H^1(\Omega)^{3}$ to the boundary value problem
\begin{equation}\label{divrot}
\left\lbrace
\begin{array}{lll}
\curl \bU =& \bom&\quad \mbox{ in }\quad \Omega\\
\dive \bU = &0&\quad \mbox{ in }\quad \Omega\\
\Vp=& \nabla\psi+\nabla^\perp \Delta^{-1}(\uom\cdot N)&\quad\mbox{ at the surface}\\
U_b\cdot N_b=&0 &\quad\mbox{ at the bottom},
\end{array}\right.
\end{equation}
where we recall that $\Vp$ is
defined in \eqref{notapar}, that $\uom$ stands for the trace of
$\bom$ at the surface, and that $U_b$ is the trace of $\bU$ at the bottom.
In the statement of the
theorem, we use the following
definition to denote divergence free vector fields defined on the
fluid domain $\Omega$. The second point of the definition is motivated
by Remark \ref{vorbot} (where the space $H_0^{-1/2}$ is also introduced).
\begin{defi}\label{divfree}
Let $\zeta\in W^{2,\infty}(\R^d)$ be such that \eqref{hmin} is
satisfied and $\Omega$ be as in \eqref{defdomain}. \\
{\bf i.} We define the subspace of $L^2(\Omega)^{3}$ of divergence
  free vector fields as
$$
H(\mbox{\textnormal{div}}_0,\Omega)=\{{\bf B}\in
L^2(\Omega)^3,\dive{\bf B}=0\}.
$$
{\bf ii.} The set of such functions satisfying
$B_b\cdot N_b \in H^{-1/2}_0(\R^d) $ is
denoted
$$
H_b(\mbox{\textnormal{div}}_0,\Omega)=\{{\bf B}\in
H(\mbox{\textnormal{div}}_0,\Omega), \, B_b\cdot N_b \in H^{-1/2}_0(\R^d)\},
$$
which we equip with the norm
$$
\Abs{{\bf B}}_{2,b}=\Abs{{\bf B}}_2+\abs{B_b\cdot N_b}_{H_0^{-1/2}}.
$$
\end{defi}
For the sake of clarity, the proof of the following theorem
is postponed to \S \ref{sectproofth1}.
\begin{thm}\label{prop1}
Let $\zeta\in W^{2,\infty}(\R^d)$ be such that \eqref{hmin} is
satisfied and $\Omega$ be as in \eqref{defdomain}. Let also
$\bom\in H_b(\mbox{\textnormal{div}}_0,\Omega)^3$ and  $\psi\in
\dot{H}^{3/2}(\R^d)$.\\
There exists a unique solution $\bU\in H^1(\Omega)^{3}$ to the
boundary value problem \eqref{divrot}, and one has
$\bU=\curl \bA+\nabla_{X,z}\Phi$ with $\bA\in H^2(\Omega)^{3}$ solving
\begin{equation}\label{eqA}
\left\lbrace
\begin{array}{rll}
\dsp \curl \curl \bA&=\bom &\mbox{ \textnormal{in} }\Omega,\\
\dsp \dive \bA&=0 &\mbox{ \textnormal{in} }\Omega,\\
\dsp N_b\times A_b&=0&\\
\dsp N\cdot \uA&=0&\\
\dsp (\curl
  \bA)_\parallel&=\nabla^\perp \Delta^{-1}(\uom\cdot N),&\\
\dsp N_b\cdot(\curl \bA)_\bott & = 0,&
\end{array}\right.
\end{equation}
while $\Phi\in \dot{H}^2(\Omega)$ solves
\begin{equation}\label{eqPhi0}
\left\lbrace
\begin{array}{l}
\Delta_{X,z}\Phi=0 \quad\mbox{ \textnormal{in} }\quad \Omega,\\
\Phi_\surf=\psi,\qquad\partial_n\Phi_\bott=0.
\end{array}\right.
\end{equation}
Moreover, one has
$$
\Abs{\bU}_2+\Abs{\nabla_{X,z}\bU}_2^2\leq C(\frac{1}{h_{\min}},H_0,\abs{\zeta}_{W^{2,\infty}})\big(\Abs{\bom}_{2,b}+\abs{\nabla\psi}_{H^{1/2}}\big).
$$
\end{thm}
The theorem furnishes a Hodge-Weyl decomposition\footnote{Note that
  this decomposition is
  not orthogonal for the $L^2(\Omega)$ scalar product. The standard
  (orthogonal) Hodge-Weyl decomposition would be
$$
\bU=\bU^\sharp+\nabla_{X,z} \Phi^\sharp,
$$
with $\bU^\sharp$ divergence free and tangential to the boundaries of
$\Omega$. This decomposition does not isolate the vorticity effects in
the sense that they are present in both terms of the decomposition,
while they are absent in the potential part of the decomposition we
use here.
} of the velocity
field, $\bU=\nabla_{X,z}\Phi+\curl \bA$, the first component of which
is the irrotational part of the velocity field, and the second one its
rotational part. More precisely, the theorem allows us to give the
following definition.
\begin{defi}\label{defimappings}
Let $\zeta\in W^{2,\infty}(\R^d)$ satisfy \eqref{hmin} and $\Omega$ be
given by \eqref{defdomain}.
\item[(1)] We define the linear mapping ${\mathbb
  A}[\zeta]$ as follows
$$
{\mathbb A}[\zeta]:
\begin{array}{lcl}
H_b(\mbox{\textnormal{div}}_0,\Omega)&\to& H^2(\Omega)^{3},\\
\bom&\mapsto& {\bf A},
\end{array}
$$
where $\bA$ is the solution to \eqref{eqA} provided by Theorem
\ref{prop1}.
\item[(2)] We define the linear mappings ${\mathbb
  U}_I[\zeta]$ and ${\mathbb
  U}_{II}[\zeta]$ by
$$
{\mathbb U}_I[\zeta]:
\begin{array}{lcl}
\dot{H}^{3/2}(\R^d)&\to& H^1(\Omega)^{3}\\
\psi &\mapsto& \nabla_{X,z} \Phi
\end{array}
\quad \mbox{\textnormal{and}}\quad
 {\mathbb U}_{II}[\zeta]:
\begin{array}{lcl}
H_b(\mbox{\textnormal{div}}_0,\Omega)&\to& H^1(\Omega)^{3},\\
\bom &\mapsto& \curl ({\mathbb A}[\zeta]\bom)
\end{array},
$$
where $\Phi$ solves \eqref{eqPhi0}.
\item[(3)] The linear mapping ${\mathbb
  U}[\zeta]$ is 
$$
 {\mathbb U}[\zeta]:
\begin{array}{lcl}
\dot{H}^{3/2}(\R^d)\times H_b(\mbox{\textnormal{div}}_0,\Omega)&\to& H^1(\Omega)^{3},\\
(\psi,\bom) &\mapsto& {\mathbb U}_I[\zeta]\psi+{\mathbb U}_{II}[\zeta]\bom
\end{array}.
$$
\end{defi}

\subsection{The generalized Zakharov-Craig-Sulem formulation}\label{sectTheeq}

We use here the results of \S \ref{sectdivcurl1}
to derive a closed system of equations on $\zeta$, $\psi$ and $\bom$,
from which all the other physical quantities can be deduced. \\
According to Definition \ref{defimappings}, the kinematic
equation \eqref{kineticeq} can be written
$$
\dt \zeta-{\mathbb U}[\zeta](\psi,\bom)_\surf\cdot N=0.
$$
Proceeding similarly with the equation \eqref{4eq1} for $\psi$, and
the vorticity equation \eqref{4eq3}, and introducing the mappings\footnote{Note
  that according to \eqref{divrot}, one has
$$
{\mathbb U}_\parallel[\zeta](\psi,\bom)=\nabla\psi+\frac{\nabla^\perp}{\Delta}(\uom\cdot
N).
$$}
\begin{equation}\label{defVVp}
\begin{array}{lcl}
\dsp \underline{\mathbb
  V}[\zeta](\psi,\bom)&=&{\mathbb
  V}[\zeta](\psi,\bom)_\surf,\\
\dsp \underline{\mathbb
  w}[\zeta](\psi,\bom)&=&{\mathbb
  w}[\zeta](\psi,\bom)_\surf,\\
\dsp {\mathbb
  U}_\parallel[\zeta](\psi,\bom)&=&\underline{\mathbb
  V}[\zeta](\psi,\bom) +\underline{\mathbb
  w}[\zeta](\psi,\bom)\nabla\zeta
\end{array}
\end{equation}
we  derive the following
generalization of the Zakharov-Craig-Sulem formulation of the
water-waves equations in presence of a nonzero vorticity field,
\begin{equation}\label{ZCSgen}
\left\lbrace
\begin{array}{l}
\dsp \dt \zeta-\underline{\mathbb U}[\zeta](\psi,\bom)\cdot N=0,\\
\dsp \dt \psi +g \zeta+\frac{1}{2}
 \babs{{\mathbb
     U}_\parallel[\zeta](\psi,\bom)}^2-\frac{1}{2}(1+\abs{\nabla\zeta}^2)\underline{{\mathbb
   w}}[\zeta](\psi,\bom)^2 \\
\hspace{4.2cm}\dsp -\frac{\nabla^\perp}{\Delta}\cdot\big(
\uom\cdot N \underline{\mathbb
  V}[\zeta](\psi,\bom)\big)=0,\\
\dsp \dt \bom+{\mathbb U}[\zeta](\psi,\bom)\cdot\nabla_{X,z}
\bom=\bom\cdot\nabla_{X,z} {\mathbb U}[\zeta](\psi,\bom)
\end{array}\right.
\end{equation}
(with $\underline{\omega}=\bom_\surf$). Note that the divergence free constraint on the vorticity
\begin{equation}\label{DFvort}
\dive\bom=0 \quad\mbox{ in }\quad \Omega_t
\end{equation}
should be added to these equation; we omit it however since it is
propagated by the vorticity equation if it is initially satisfied.
\begin{remark}
For $\zeta\in W^{2,\infty}(\R^d)$ satisfying \eqref{hmin} one can
define a {\it generalized} Dirichlet-Neumann operator $\cGg[\zeta]$ as
$$
\cGg[\zeta]:\begin{array}{lcl}
\dot{H}^{3/2}(\R^d)\times H_b(\mbox{\textnormal{div}}_0,\Omega)&\to & H^{1/2}(\R^d),\\
(\psi,\bom)&\mapsto& {\mathbb U}[\zeta](\psi,\bom)_\surf\cdot N;
\end{array}
$$
the standard Dirichlet-Neumann operator used in the irrotational case
for the Zakharov-Craig-Sulem formulation of the water waves
corresponds to $\cG[\zeta]\psi=\cGg[\zeta](\psi,0)$. Remarking that
$$
\underline{\mathbb
  w}[\zeta](\psi,\bom)=\frac{\cGg[\zeta](\psi,\omega)+\nabla\zeta\cdot
{\mathbb U}_\parallel[\zeta](\psi,\omega)}{1+\abs{\nabla\zeta}^2},
$$
the equations \eqref{ZCSgen} can be written under the form
$$
\left\lbrace
\begin{array}{l}
\dsp \dt \zeta-\cGg[\zeta](\psi,\bom)=0,\\
\dsp \dt \psi +g \zeta+\frac{1}{2}
 \babs{{\mathbb U}_\parallel[\zeta](\psi,\bom)}^2-\frac{\big(\cGg[\zeta](\psi,\bom)+\nabla\zeta\cdot {\mathbb U}_\parallel[\zeta](\psi,\bom)\big)^2}{2
(1+\abs{\nabla\zeta}^2)}\\
\dsp \hspace{4.2cm}-\frac{\nabla^\perp}{\Delta}\cdot\big(
\uom\cdot N \underline{\mathbb
  V}[\zeta](\psi,\bom)\big)=0,\\
\dsp \dt \bom+{\mathbb U}[\zeta](\psi,\bom)\cdot\nabla_{X,z}
\bom=\bom\cdot\nabla_{X,z} {\mathbb U}[\zeta](\psi,\bom),
\end{array}\right.
$$
In the irrotational case, one has $\bom=0$ and ${\mathbb
  U}_\parallel[\zeta](\psi,\bom)=\nabla\psi$; denoting further
$\cG[\zeta]\psi=\cGg[\zeta](\psi,0)$,  these equations then
simplify into
\begin{equation}\label{ZCS}
\left\lbrace
\begin{array}{l}
\dsp \dt \zeta-\cG[\zeta]\psi=0,\\
\dsp \dt \psi +g \zeta+\frac{1}{2}
 \babs{\nabla\psi}^2-\frac{\big(\cG[\zeta]\psi+\nabla\zeta\cdot\nabla\psi\big)^2}{2
(1+\abs{\nabla\zeta}^2)}=0,
\end{array}\right.
\end{equation}
which is the standard Zakharov-Craig-Sulem formulation.
\end{remark}

\begin{remark}\label{straightvort}
Contrary to the irrotational case where the water waves equations
\eqref{ZCS} are cast on the fixed domain $\R^d$, our formulation
\eqref{ZCSgen} of
the water waves equations with vorticity are partly cast on the moving
-- and unknown -- fluid domain $\Omega_t$ parametrized at time $t$ by
the free surface $\zeta(t,\cdot)$ through \eqref{defdomain}. The
functional setting for the study of the vorticity $\bom$ is therefore
less straightforward than for $\zeta$ and $\psi$. A convenient way to
deal with this difficulty\footnote{Even in the
  irrotational case, this problem arises if one wants to give a
  rigorous meaning to the original water waves equations
  \eqref{Eulereq}, \eqref{incomp}, \eqref{kineticeq} (plus an
  irrotationality condition). One of the advantages of the
  Zakharov-Craig-Sulem formulation \eqref{ZCS}  is that such
  difficulties have disappeared. They must however be handled to prove
rigorously that the free surface Euler equations are indeed equivalent to
\eqref{ZCS}. This very careful analysis has been performed only
recently in \cite{ABZ_Berti}.} is to fix the domain by using a
diffeomorphism $\Sigma(t,\cdot)$ mapping at each time the flat strip
$\cS=\R^d\times (-H_0,0)$ onto the fluid domain $\Omega_t$.
The equation on the
vorticity $\bom$ in \eqref{ZCSgen} is then replaced by an equation on
the {\it straightened vorticity} $\omega=\bom\circ \Sigma$, which is
defined on the (fixed) strip $\cS$.
\end{remark}

\section{The div-curl problem} \label{sect3}

The resolution of the div-curl problem \eqref{divrot} is necessary to
prove that the formulation \eqref{ZCSgen} of the water waves equations
with vorticity forms a closed set of equations. The well-posedness of
this boundary value problem was stated in Theorem \ref{prop1}; its
proof is given below in \S \ref{sectproofth1}. A consequence of
Theorem \ref{prop1} is that the $\curl$ operator can be ``inverted'',
as explained in \S \ref{sectinvertcurl}.\\
The analysis of the evolution equations \eqref{ZCSgen} shall require
additional properties on the velocity field provided by Theorem
\ref{prop1}. After transforming the fluid domain into a flat strip in
\S \ref{sectstraight}, we provide in \S \ref{secthigher} higher order
estimates on the solution. The control of time derivatives requires a
specific treatment, and is performed in \S
\ref{sectproofpropshape}. Finally, crucial properties of the so called
``good unknowns'' are provided in \S \ref{sectAIG}.

\subsection{Proof of Theorem \ref{prop1}}\label{sectproofth1}

We prove in this section Theorem \ref{prop1}, that is, we solve the
following div-curl problem in the fluid domain $\Omega$,
\begin{equation}\label{divrotbis}
\left\lbrace
\begin{array}{lll}
\curl \bU =& \bom&\quad \mbox{ in }\quad \Omega\\
\dive \bU = &0&\quad \mbox{ in }\quad \Omega\\
U_\parallel=& \nabla\psi+\nabla^\perp \tpsi ,&\quad\mbox{ at the surface}\\
U_b\cdot N_b=&0 &\quad\mbox{ at the bottom}.
\end{array}\right.
\end{equation}
with $\tpsi\in \dot{H}^{3/2}(\R^d)$ such that $\Delta
\tpsi=\uom\cdot N$ (see Lemma \ref{existtildepsi} below for the
existence of $\tpsi$).\\
In order for the boundary conditions to make sense, some minimal
regularity is needed on $\bU$; let us recall the definitions
\begin{eqnarray}
\label{defHdiv} H(\dive,\Omega)&=&\{\bU\in
L^2(\Omega)^3,\dive\bU\in L^2(\Omega)\}\\
\label{defHrot} H(\curl,\Omega)&=&\{\bU\in
L^2(\Omega)^3,\curl\bU\in L^2(\Omega)^3\};
\end{eqnarray}
it is classical (see for instance Chapter 9 of \cite{DL}) that normal traces
(resp. tangential traces) at the boundary are well defined in
$H(\dive,\Omega) $ (resp. $H(\curl,\Omega)$). It follows that if
$\bom\in L^2(\Omega)^{3}$ and $\bU$ solves the first two equations
of \eqref{divrotbis}, one can take normal and tangential traces at the
boundaries, so that it makes sense to impose the two boundary
conditions of \eqref{divrotbis}.

The first remark to do, is that it is easy to reduce \eqref{divrotbis}
to the case $\psi=0$. Indeed, let $\Phi\in \dot{H}^2(\Omega)$ solve
the boundary value problem
\begin{equation}\label{eqPhi}
\left\lbrace
\begin{array}{l}
\Delta_{X,z}\Phi=0\quad\mbox{ in }\quad \Omega,\\
\Phi_\surf=\psi,\qquad \partial_n\Phi_\bott=0
\end{array}\right.
\end{equation}
(see for instance Chapter 2 of \cite{Lannes_book} for the existence
and uniqueness of such a $\Phi$); defining
$\tbU=\bU-\nabla_{X,z}\Phi$, one readily computes that
\begin{equation}\label{divrotter}
\left\lbrace
\begin{array}{lll}
\curl \tbU =& \bom&\quad \mbox{ in }\quad \Omega\\
\dive \tbU = &0&\quad \mbox{ in }\quad \Omega\\
\tU_\parallel=& \nabla^\perp\tpsi,&\quad\mbox{ at the surface}\\
\tU_b\cdot N_b=&0 &\quad\mbox{ at the bottom},
\end{array}\right.
\end{equation}
which is the same problem as \eqref{divrotbis} with $\psi=0$.

We look for a solution to \eqref{divrotter}  under the form $\tbU=\curl
\bA$, where the  potential vector $\bA$ satisfies the system
\begin{equation}\label{potvector}
\left\lbrace
\begin{array}{rl}
\dsp \curl \curl \bA&=\bom\\
\dsp N_b\times A_b&=0\\
\dsp N\cdot \uA&=0\\
\dsp \left(\curl \bA\right)_\parallel&=\nabla^\perp \widetilde{\psi}.\\
\end{array}\right.
\end{equation}
It is important to notice that
\begin{align*}
N_b\times A _b=0 \Longrightarrow N_b \cdot (\nabla\times A) _\bott=0,
\end{align*}
which corresponds to the last boundary condition of  \eqref{divrotter}.

Before proving the existence of such a potential
vector $\bA$ in $H^1(\Omega)$, we need a series of preliminary lemmas.
The first one is a Poincar\'e inequality for vector fields whose
normal component vanishes at the surface, and whose tangential
components vanish at the bottom.
\begin{lemma}\label{lempoincare}
Assume that $\zeta\in W^{1,\infty}(\R^d)$ satisfies the non vanishing
depth condition \eqref{hmin}.  Let also $\bA\in
H(\dive,\Omega)\cap H(\curl,\Omega)$ be such that
$$
A_b\times N_b=0,\quad\mbox{\textnormal{ and }}\quad
\uA\cdot N=0.
$$
Then one has
$$\Abs{\bA}_{2}\leq C(H_0,\abs{\zeta}_{W^{1,\infty}}) \Abs{\dz\bA}_{2}.$$
\end{lemma}
\begin{proof}
Let $\Pi_I(X): \R^{3}\to \R^{3}$ be the projection onto
$\R N(X)$ parallel to $N_b^\perp$, and $\Pi_{II}(X): \R^{3}\to \R^{3}$ be the projection onto
the horizontal plane $N_b^\perp$ parallel to $N(X)$. One can decompose $\bA$ under
the form
$$
\bA=\bA_I+\bA_{II} \quad \mbox{ with }\quad \bA_{I}=\Pi_I \bA, \quad
\bA_{II}=\Pi_{II} \bA.
$$
Since $\bA_I$ vanishes at the surface and $\bA_{II}$ vanishes at the
bottom, we can use the standard Poincar\'e inequality to get
\begin{eqnarray*}
\Abs{\bA}_2&\leq& \Abs{\bA_I}_2+\Abs{\bA_{II}}_2\\
&\leq& \abs{H_0+\zeta}_\infty (\Abs{\dz \bA_I}_2+\Abs{\dz \bA_{II}}_2).
\end{eqnarray*}
since the projectors $P_{j}$ ($j=I,II$) do not depend on $z$ and have
operator norm bounded by
$C(\abs{\zeta}_{W^{1,\infty}})$, we easily deduce
that
$$
\Abs{\bA}_2 \leq C(\abs{\zeta}_{W^{1,\infty}},H_0)
\Abs{\dz \bA}_2.
$$
\end{proof}
The second lemma shows that all the derivatives of $\bA$ can be
controlled in terms of $\curl \bA$, $\dive \bA$ and the trace of $\bA$
at the surface and bottom.
\begin{lemma}\label{leminterm}
Let $\zeta\in W^{2,\infty}(\R^d)$ satisfy \eqref{hmin} and let  $\bA\in H(\dive, \Omega)\cap H(\curl,\Omega)$ be such that
$$
A_b\times N_b=0,\quad\mbox{\textnormal{ and }}\quad
\uA\cdot N=0.
$$
Then one has
\begin{align*}
\Abs{\nabla_{X,z} \bA}_{2}^2\leq \Abs{\curl \bA}_{2} +\Abs{\dive
  \bA}_2^2 +C \abs{\zeta}_{W^{2,\infty}}\abs{\uA}_2^2,
\end{align*}
for some numerical constant $C$.
\end{lemma}
\begin{proof}
The proof is a small variant of classical estimates (see for instance
Chapter 9 in \cite{DL}). With the convention of summation of repeated indices, and with the notation
$(\partial_1,\partial_2,\partial_3)^T=\nabla_{X,z}$, one has
\begin{eqnarray*}
\int_\Omega
\abs{\partial_j\bA_i}^2&=&\int_\Omega \partial_j\bA_i(\partial_j\bA_i-\partial_i\bA_j)+\int_\Omega \partial_j\bA_i\partial_i\bA_j\\
&=&\Abs{\curl \bA}_2^2 +\int_\Omega \partial_j\bA_i\partial_i\bA_j,
\end{eqnarray*}
and one can rewrite the second term of the right-hand-side as
\begin{eqnarray*}
\int_\Omega \partial_j\bA_i\partial_i\bA_j&=&\int_{\Gamma} \overrightarrow{n}_j
\bA_i \partial_i \bA_j-\int_\Omega \bA_i\partial_i\partial_j \bA_j\\
&=&\int_\Gamma \big(\overrightarrow{n}_i
\bA_j\partial_j\bA_i-\overrightarrow{n}_j\bA_j\partial_i\bA_i\big)+\int_\Omega \partial_i\bA_i\partial_j\bA_j
\end{eqnarray*}
where $\overrightarrow{n}$ is the outward unit normal vector to the boundary
$\Gamma$ of $\Omega$ (i.e. the bottom and the surface); we have therefore obtained
\begin{equation}
\label{nablacurldiv}
\Abs{\nabla_{X,z}\bA}_2^2=\Abs{\curl \bA}_2^2+\Abs{\dive \bA}_2^2+\int_\Gamma \big(\overrightarrow{n}_i
\bA_j\partial_j\bA_i-(\overrightarrow{n}\cdot \bA)\dive\bA\big).
\end{equation}
Let us now evaluate the surface and bottom contributions of the
boundary integral in the right-hand-side:\\
- {\it Bottom contribution.} Since $\bA$ is a normal vector field at
the bottom, one has
$$
(\overrightarrow{n}\cdot \bA)(\dive \bA)_\bott=2\abs{\bA}^2
H_b+\bA\cdot \partial_{\overrightarrow{n}} \bA,
$$
where $H_b$ is the mean curvature of the bottom. Remarking that
\begin{eqnarray*}
\int_{\rm bott} \overrightarrow{n}_i
\bA_j\partial_j\bA_i
&=& \int_{\rm bott} 2(\overrightarrow{n}\times \bA)\cdot \curl
\bA+\bA\cdot\partial_{\overrightarrow{n}}\bA\\
&=&\int_{\rm bott} \bA\cdot\partial_{\overrightarrow{n}}\bA
\end{eqnarray*}
(since $\bA$ is normal to the bottom), we get the following expression
for the contribution of the bottom to the boundary integral in \eqref{nablacurldiv},
\begin{equation}\label{contribbott}
\int_{\rm bott} \big(\overrightarrow{n}_i
\bA_j\partial_j\bA_i-(\overrightarrow{n}\cdot
\bA)\dive\bA\big)=-2\int_{\rm bott} H_b \abs{A_b}^2,
\end{equation}
which vanishes since the bottom is flat.\\
- {\it Surface contribution.} Since $\uA\cdot\overrightarrow{n}=0$ at
the surface, the contribution of the surface to
the boundary integral in \eqref{nablacurldiv} is of the form
$$
\int_{\rm surf} \overrightarrow{n}_i
\bA_j\partial_j\bA_i
=
\int_{\rm surf} (\bA\cdot \nabla_{X,z})
(\bA \cdot \overrightarrow{n})-\bA\cdot (\bA\cdot \nabla_{X,z} )\overrightarrow{n},
$$
where we still denote by $\overrightarrow{n}$ a local extension of
$\overrightarrow{n}$ inside $\Omega$. Since $\bA$ is tangent to the
surface, the operator $\bA\cdot \nabla_{X,z}$ is
tangential, and the first component of the right-hand-side
vanishes. Since moreover, the extension of $\overrightarrow{n}$ can be
chosen such that $\Vert \overrightarrow{n}\Vert_{W^{1,\infty}}\leq C
\abs{\zeta}_{W^{2,\infty}}$, we deduce that
\begin{equation}\label{contribsurf}
\int_{\rm surf} \big(\overrightarrow{n}_i
\bA_j\partial_j\bA_i-(\overrightarrow{n}\cdot
\bA)\dive\bA\big)\leq C \abs{\zeta}_{W^{2,\infty}}\abs{\uA}_2^2.
\end{equation}
\end{proof}
The Lemmas \ref{lempoincare} and \ref{leminterm} imply a useful
equivalence property for the $H^1$ norm of $\bA$ with the $L^2$ norms
of $\dive\bA$ and $\curl \bA$.
\begin{lemma}\label{equivnorm}
Let $\zeta\in W^{2,\infty}(\R^d)$ satisfy \eqref{hmin} and let  $\bA\in H(\dive, \Omega)\cap H(\curl,\Omega)$ be such that
$$
A_b\times N_b=0,\quad\mbox{\textnormal{ and }}\quad
\uA\cdot N=0.
$$
Then one has
$$
\Abs{\bA}_2^2+\Abs{\nabla_{X,z}\bA}_2^2\leq
C(\abs{\zeta}_{W^{2,\infty}},H_0) \big(\Abs{\dive \bA}^2_2+\Abs{\curl \bA}_2^2\big).
$$
\end{lemma}
\begin{proof}
Recalling that $\bA_h$ and $\bA_v$ stand for the horizontal and
vertical components of $\bA$, we deduce from the assumption
that $\bA$ is tangential at the surface that $\uA_v=\nabla\zeta \cdot
\uA_h$ and therefore,
$$
\abs{\uA}_2 \leq C(\abs{\nabla\zeta}_\infty)\abs{\uA_h}_2,
$$
and we therefore turn to estimate the $L^2$-norm of $\uA_h$. Using the
fact that ${\bA_h}_\bott=0$ by assumption, we can write
\begin{align*}
&\int_{\R^d}\abs{{\uA_h}}^2=\int_{\R^d}\int_{-H_0}^{\zeta}\pa_z \abs{{\bA_h}}^2=2\int_{\Omega}\bA_h\cdot \pa_z \bA_h.
\end{align*}
Remarking that $\dz\bA_h=-(\curl \bA)_h^\perp+\nabla \bA_v$,
we have that
\begin{align*}
\int_{\R^d}\abs{{\uA_h}}^2&=-2\int_{\Omega}\bA_h\cdot
\left(\curl \bA\right)_h^\perp+2\int_\Omega \bA_h\cdot \nabla \bA_v\\
&=-2\int_{\Omega}\bA_h\cdot \left(\curl
  \bA\right)_h^\perp-2\int_\Omega(\nabla\cdot \bA_h) \bA_v-2\int_{\R^d}\nabla\zeta\cdot {\uA_h} {\uA_v}\\
&=-2\int_{\Omega}\bA_h\cdot \left(\curl
  \bA\right)_h^\perp-2\int_\Omega (\dive \bA)\bA_v+  2\int_\Omega\pa_z
\bA_v \bA_v\\
& \indent+2\int_{\R^d} -\nabla\zeta\cdot \uA_h \uA_v
\end{align*}
Remarking that the third term is equal to
$\abs{\uA_v}_2-\vert{(\bA_v)_\bott}\vert_2$,  and recalling that
$\uA_v=\nabla\zeta\cdot\uA_h$, we deduce
\begin{align*}
\int_{\R^d}\abs{{\uA_h}}^2
&=-2\int_{\Omega}\bA_h\cdot \left(\curl
  \bA\right)_h^\perp-2\int_\Omega (\dive \bA)\bA_v-\int_{\R^2}{\bA_v}_{|_\text{bott}}^2-\int_{\R^2}\uA_v^2.
\end{align*}
We can now deduce from the above that
\begin{eqnarray*}
\abs{\uA}_2^2&\leq& C(\abs{\nabla\zeta}_\infty)\Abs{\bA}_2
\big(\Abs{\dive \bA}_2+\Abs{\curl \bA}_2\big)\\
&\leq& C(\abs{\zeta}_{W^{1,\infty}},H_0)\Abs{\nabla_{X,z}\bA}_2
\big(\Abs{\dive \bA}_2+\Abs{\curl \bA}_2\big),
\end{eqnarray*}
where we used the Poincar\'e inequality provided by Lemma \ref{lempoincare} to
obtain the second inequality. Using Young's inequality $2ab\leq
\epsilon^2 a^2 +\epsilon^{-2} b^2$ with $\epsilon$ small enough and
the estimate furnished by Lemma \ref{leminterm}, we get the result.
\end{proof}

We can now proceed to construct a solution to
\eqref{potvector}; we also impose that $\bA$ be divergence free. We are
therefore concerned with the problem
\begin{equation}\label{potencial}
\left\lbrace\begin{array}{rc}
\dsp \curl\curl \bA=&\bom\\
\dsp \dive \bA=& 0,
\end{array}\right.
\mbox{ in }\Omega,
\end{equation}
with the boundary conditions
\begin{equation}\label{Bcond}
\left\lbrace
\begin{array}{rc}
\dsp \uA\cdot N=&0\\
\dsp N_b\times A_b=&0\\
\dsp  (\curl \bA)_\parallel=&\nabla^\perp\tpsi,\\
\end{array}\right.
\end{equation}
with $\tpsi\in \dot{H}^{1/2}(\R^d)$.
The first step is to construct a variational solution in the following sense.
\begin{defi}\label{weakpotencial}
\item[(1)] We denote by ${\mathfrak X}$ the closed subspace of
  $H^1(\Omega)^3$ defined as
$$
{\mathfrak X}=\{\bA\in H^1(\Omega)^{3},\quad \dive \bA=0,\quad \uA\cdot N=0,\quad
N_b\times A_b=0\}.
$$
\item[(2)] Let $\bom\in L^2(\Omega)^{3}$ and $\tpsi\in \dot{H}^{1/2}(\R^d)$. Then the vector field $\bA\in\mathfrak{X}$ is a variational solution of the system \eqref{potencial}-\eqref{Bcond} if and only if
\begin{align}\label{varform}
\forall {\bf C}\in {\mathfrak X},\qquad \int_\Omega \curl
\bA\cdot\curl {\bf C}=\int_{\Omega} \bom\cdot {\bf
  C}+\int_{\R^d} \nabla\tpsi\cdot {C}_\parallel,
\end{align}
where we recall the notation $C_\parallel=\underline{C}_h+\nabla\zeta \underline{C}_v$.
\end{defi}

The following lemma shows the existence and uniqueness of such a variational
solution.
\begin{lemma}\label{1exthm}
Let $\zeta\in W^{2,\infty}(\R^d)$ satisfy \eqref{hmin}, and let $\bom\in L^2(\Omega)^{3}$ and
$\tpsi\in \dot{H}^{1/2}(\R^d)$. Then there exists a unique variational
solution  $\bA\in \mathfrak{X}$ to \eqref{potencial}-\eqref{Bcond} in
the sense of Definition \ref{weakpotencial}. One has moreover
\begin{equation}\label{estrot}
\Abs{\curl \bA}_2 \leq C(H_0,\abs{\zeta}_{W^{2,\infty}})\big(\Abs{\bom}_2+\abs{\nabla\tpsi}_{H^{-1/2}}\big).
\end{equation}
\end{lemma}
\begin{proof}
The bilinear form
$$
a(\bA,{\bf C})=\int_\Omega \curl \bA\cdot \curl {\bf C}
$$
is obviously continuous on ${\mathfrak X}\times {\mathfrak X}$; Lemma
\ref{equivnorm} also shows that it is coercive. The linear form
\begin{align*}
L(\bA)=&\int_\Omega \bom \cdot \bA +\int_{\R^2}\nabla\tpsi\cdot A_\parallel
\end{align*}
also satisfies
\begin{align*}
L(\bA)&\leq \Abs{\bom}_2\Abs{\bA}_2+\abs{\nabla\tpsi}_{H^{-\frac{1}{2}}} \abs{A_\parallel}_{H^{\frac{1}{2}}}\\
&\leq \Abs{\bom}_2\Abs{\bA}_2+C(\abs{\zeta}_{W^{2,\infty}})\abs{\nabla\tpsi}_{H^{-\frac{1}{2}}} \Abs{\bA}_{H^1},
\end{align*}
where we used the trace lemma\footnote{By invoking the ``trace
  lemma'', we refer throughout this article to the continuity of the
  trace operators at the surface and at the bottom, as operators
  defined on $H^1(\Omega)$ with values in $H^{1/2}(\R^d)$,
$$
\forall \bA\in H^1(\Omega),\qquad \abs{\uA}\leq
C(\abs{\zeta}_{W^{1,\infty}})\Abs{\bA}_{H^1}
\quad\mbox{ and }\quad
\abs{A_b}\leq
C\Abs{\bA}_{H^1}.
$$
 } to derive the second inequality; is is
therefore continuous on ${\mathfrak X}$ and  we can
apply Lax-Milgram's theorem to obtain
the existence and uniqueness of the variational solution.\\
Taking ${\bf C}={\bf A}$ in \eqref{varform} and using the above
estimate on $L(\cdot)$, we get
$$
\Abs{\curl \bA}_2^2\leq C(\abs{\zeta}_{W^{2,\infty}})\big(\Abs{\bom}_2+\abs{\nabla\tpsi}_{H^{-1/2}}\big) \Abs{\bA}_{H^1},
$$
and the estimate of the lemma therefore follows directly from Lemma
\ref{equivnorm} since $\dive \bA=0$.
\end{proof}
Because of the divergence free condition, the space ${\mathfrak X}$
used as the space of test functions in Definition \ref{weakpotencial}
is too small to ensure that the variational solution provided by Lemma
\ref{1exthm} satisfies the first equation of \eqref{potencial} in the
sense of distributions. If we therefore want to take a larger space of test functions by
removing the divergence free condition in the definition of
${\mathfrak X}$, namely, by working with test functions in the space
$$
H^1_{\rm b.c.}(\Omega)=\{\bA\in H^1(\Omega)^{3},\quad \uA\cdot N=0,\quad
N_b\times A_b=0\},
$$
For all  ${\bf C}\in H^1_{\rm b.c.}(\Omega)$, we  define $\varphi$ by
\begin{align*}
\Delta \varphi =& \dive {\bf C}\\
\partial_n\varphi_\surf=&0\\
\varphi_\bott=&0,
\end{align*}
so that ${\bf C}-\nabla_{X,z}\varphi\in \mathfrak{X}$ and therefore
\begin{align*}
\int_{\Om} \curl {\bf A}\cdot \curl {\bf C}=\int_{\Om} \bom\cdot
\left({\bf C}-\nabla_{X,z}\varphi\right)
+\int_{\R^d}\nabla\tilde{\psi}\cdot ({C}_\parallel-(\nabla_{X,z}\varphi)_\parallel).
\end{align*} 
But, if $\dive \bom=0$,  
\begin{align*}
\int_{\Omega} \bom\cdot \nabla_{X,z} \varphi = &\int_{\R^d} \underline{\omega}\cdot N \underline{\varphi}\\
\int_{\R^d}\nabla\tilde{\psi}\cdot (\nabla_{X,z}\varphi)_\parallel=&-\int_{\R^d}\Delta \tilde{\psi}\underline{\varphi}.
\end{align*}
Thus, if $\Delta\tilde{\psi}=\underline{\omega}\cdot N$ we learn that
\begin{align}\label{varform3}
\int_{\Omega}\curl \bA\cdot \curl {\bf C} =\int_{\Omega}\bom\cdot
{\bf C}+\int_{\R^d}\nabla \tilde{\psi}\cdot C_\parallel
\end{align}
for all ${\bf C}\in H^1_{\rm b.c.}(\Omega)$.

It is now easy to deduce that the variational solution furnished by
Lemma \ref{1exthm} is a strong solution of
\eqref{potencial}-\eqref{Bcond}.
\begin{lemma}\label{3exthm}
Let $\zeta\in W^{2,\infty}(\R^d)$ satisfy \eqref{hmin}, and let $\bom\in L^2(\Omega)^{3}$ and
$\tpsi\in \dot{H}^{1/2}(\R^d)$. If moreover $\dive \bom=0$ and $\Delta
\tpsi=\uom\cdot N$, then the variational solution $\bA\in H^1(\Omega)^3$ furnished by Lemma
\ref{1exthm} solves \eqref{potencial} with boundary conditions \eqref{Bcond}.
\end{lemma}
\begin{proof}
By construction, $\dive \bA=0$. The fact that $\curl\curl \bA=\bom$
stems from \eqref{varform3} with all ${\bf C}\in {\mathcal
  D}(\Omega)\subset H^1_{\rm b.c.}(\Omega)$ as test functions. It follows
that $\curl \bA$ belongs to  $H(\dive,\Omega)\cap H(\curl,\Omega)$ and
therefore that the traces of $\curl \bA$ at the surface and bottom
make sense. The fact that these traces satisfy the last  condition
in \eqref{Bcond} is also a consequence of \eqref{varform3}. The first
two conditions of \eqref{Bcond} are automatically satisfied since
$\bA\in H^1_{\rm b.c.}(\Omega)$.
\end{proof}
Now we have all the ingredients to finish the proof of the
theorem. With $\dive\bom=0$ and $\Delta\tpsi=\uom\cdot N$, we
denote by $\bA\in {\mathfrak X}$ the variational solution furnished by
Lemma \ref{1exthm} and set $\tbU=\curl \bA$. We get directly from
Lemma \ref{3exthm} that $\tbU\in H(\dive,\Omega)\cap H(\curl,\Omega)$
solves \eqref{divrotter}. We also get from \eqref{estrot} that
\begin{eqnarray}
\nonumber
\Abs{\tbU}_2 &\leq&
\nonumber
 C(H_0,\abs{\zeta}_{W^{2,\infty}})\big(\Abs{\bom}_2+\abs{\nabla\tpsi}_{H^{-1/2}}\big)\\
&\leq&
\label{L2est}
 C(H_0,\frac{1}{h_{\rm min}},\abs{\zeta}_{W^{2,\infty}}) \Abs{\bom}_{2,b}.
\end{eqnarray}
where we used the following lemma to derive the second inequality (we
recall that $H_0^{-1/2}$ is defined in Remark \ref{vorbot} and that
$ H_b(\dive\!\!_0, \Omega)$ and $\Vert \bom\Vert_{2,b}$ are defined in Definition \ref{divfree}).
\begin{lemma}\label{existtildepsi}
Let $\zeta\in W^{1,\infty}(\R^d)$ satisfy \eqref{hmin} and $\bom\in H_b(\dive\!\!_0, \Omega)$. Then there exists a unique solution $\tpsi\in
\dot{H}^{3/2}(\R^d)$ to the equation $\Delta\tpsi=\uom\cdot N$, and
one has
$$
\abs{\nabla\tpsi}_{H^{1/2}}\leq
C(\frac{1}{h_{\min}},\abs{\zeta}_{W^{1,\infty}})\Abs{\bom}_{2,b}.
$$
\end{lemma}
\begin{proof}
The bilinear form $a(u,v)=\nabla u\cdot \nabla v$ is continuous and
coercive on the Hilbert space $\dot{H}^1(\R^d)$. Let us now define the
linear $l(\cdot)$ form on $\dot{H}^1(\R^d)$ by
$$
\forall v\in \dot{H}^1(\R^d),\qquad  l(v):=\int_{\R^d} \uom\cdot N v.
$$
If we can prove that $l$ is continuous, then the existence and
uniqueness of a variational solution $\tpsi\in \dot{H}^1(\R^d)$
to $\Delta\tpsi=\uom\cdot N$
will be a direct consequence of Lax-Milgram's theorem. We therefore
show this continuity property.\\
Note now that $\Sigma(X,z+\sigma(X,z))$, with
$\sigma=\frac{1}{H_0}(H_0+z)\zeta$, is a diffeomorphism mapping the
flat strip $\cS$ to the fluid domain $\Omega$,
and denote $\omega=\bom\circ\Sigma$. Writing
$\nabla^\sigma_{X,z}=(J_\Sigma^{-1})^T\nabla_{X,z}$ (with
$J_\Sigma=d_{X,z}\Sigma$ the Jacobian matrix), we can integrate by parts in
the above formula to find
\begin{eqnarray*}
l(v)&=&\int_{\R^d} \omega_b\cdot N_bv^{\rm ext}_b+\int_\cS (1+\dz\sigma) v^{\rm ext}
\nabla_{X,z}^\sigma\cdot
\omega+\int_{\cS}(1+\dz\sigma)\nabla_{X,z}^\sigma v^{\rm ext}\cdot \omega
\\
&=&\int_{\R^d} (\omega_b\cdot
N_b)v^{\rm ext}_b+\int_{\cS}(1+\dz\sigma)\nabla_{X,z}^\sigma v^{\rm ext}\cdot \omega
\end{eqnarray*}
where for all $v\in \dot{H}^1(\R^d)$, $v^{\rm ext}$ is the solution of the boundary value problem
$$
\left\lbrace
\begin{array}{l}
\Delta_{X,z} v^{\rm ext}=0 \qquad \mbox{ in }\quad \cS,\\
v^{\rm ext}_\surff=v,\qquad (\pa_zv^{\rm ext})_\bottf=0,
\end{array}\right.
$$
i.e. $\dsp v^{\rm ext}=\frac{\cosh((z+H_0)\abs{D})}{\cosh(H_0\abs{D})}
v$ (note that we use the Fourier multiplier notation even though $v$
belongs to $\dot{H}^1(\R^d)$ which is not a space of tempered
distributions; we refer to Notation 2.28 of \cite{Lannes_book} to see
that this makes sense).  One has therefore
$$
l(v)\lesssim \abs{\omega_b\cdot N_b}_{H_0^{-1/2}}\abs{v_b^{\rm ext}}_{\dot{H}^{1/2}}+
C(\abs{\zeta}_{W^{1,\infty}})\Abs{\omega}_2\Abs{\nabla_{X,z}^\sigma
  v^{\rm ext}}_2
$$
where for the first term, we used the fact that $H_{0}^{-1/2}(\R^d)$ is the
dual space of $\dot{H}^{1/2}(\R^d)$ (see \cite{BGSW} for a proof). Using
the explicit expression of $v^{\rm ext}$, one readily deduces that
$$
l(v)\lesssim C(\abs{\zeta}_{W^{1,\infty}},\frac{1}{h_{\min}})\Abs{\bom}_{2,b}\abs{\nabla v}_2,
$$
which implies the desired continuity property, and therefore the
existence and uniqueness of a variational solution $\tpsi\in
\dot{H}^1(\R^d)$ to $\Delta\tpsi=\uom\cdot N$. We also directly get
from the above that
$$
\abs{\nabla\tpsi}_2\leq C(\abs{\zeta}_{W^{1,\infty}},\frac{1}{h_{\min}})\Abs{\bom}_{2,b}.
$$
In order to obtain an $H^{1/2}$ estimate of $\nabla\tpsi$, let us
take the $L^2$-scalar product of the equation $\Delta\tilde\psi=\underline{\omega}\cdot N$ with
$\Lambda\tilde\psi$ (with $\Lambda=(1-\Delta)^{1/2}$). Integrating
by parts, we then proceed as above to get
\begin{eqnarray*}
\abs{\nabla\tilde\psi}_{H^{1/2}}^2&=&-\int_{\R^d}\Lambda\tpsi
\underline{\omega}\cdot N\\
&=&-\int_{\R^d}\Lambda\tpsi^{\rm ext}_b
\omega_b\cdot
N_b-\int_{\cS}(1+\dz\sigma)\nabla_{X,z}^\sigma\Lambda\tpsi^{\rm ext}\cdot \omega.
\end{eqnarray*}
One readily deduces from the Cauchy-Schwarz inequality that
\begin{eqnarray*}
\abs{\nabla\tilde\psi}_{H^{1/2}}^2&\leq &\abs{\omega_b\cdot
  N_b}_{H_0^{-1/2}}\abs{\Lambda \tpsi^{\rm ext}_\bott}_{\dot{H}^{1/2}}+
C(\abs{\zeta}_{W^{1,\infty}})\Abs{\omega}_2\Abs{\nabla_{X,z}^\sigma
  \Lambda\tpsi^{\rm ext}}_2\\
&\leq&
C(\abs{\zeta}_{W^{1,\infty}},\frac{1}{h_{\min}})\Abs{\bom}_{2,b}\Abs{\Lambda\nabla_{X,z}
  \tpsi^{\rm ext}}_2,
\end{eqnarray*}
Since $\Abs{\Lambda\nabla_{X,z} \tpsi^{\rm ext}}_2\lesssim
\abs{\nabla\tpsi}_{H^{1/2}}$ (this is a standard smoothing property, see for instance Lemma 2.20 of
\cite{Lannes_book}), one deduces that
$$
\abs{\nabla\tpsi}_{H^{1/2}}\leq
C(\abs{\zeta}_{W^{1,\infty}},\frac{1}{h_{\min}})\Abs{\bom}_{2,b},
$$
which concludes the proof of the lemma.
\end{proof}
The following lemma complements this $L^2$-estimate on $\tbU$ provided
by \eqref{L2est} by an
$H^1$-estimate.
\begin{lemma}\label{lemlem}
Let $\zeta\in W^{2,\infty}(\R^d)$ satisfy \eqref{hmin} and  $\bom\in H_b(\dive\!\!_0, \Omega)$, and $\tpsi\in \dot{H}^{3/2}(\R^d)$ be the
solution to the equation $\Delta\tpsi=\uom\cdot N$ furnished by Lemma \ref{existtildepsi}.
If $\tbU\in L^2(\Omega)^{3}$ solves \eqref{divrotter}
then $\tbU\in H^1(\Omega)^{3}$ and one has
$$
\Abs{\tbU}_2+\Abs{\nabla_{X,z}\tbU}_2\leq  C\big(\abs{\zeta}_{W^{2,\infty}},H_0,\frac{1}{h_{\min}}\big)\Abs{\bom}_{2,b}.
$$
\end{lemma}
\begin{proof}
Let us first remark that the fact that $\dive \tbU=0$ implies
\begin{equation*}
\curl\curl \tbU=-\Delta_{X,z} \tbU  =\curl \bom
\end{equation*}
and therefore
\begin{equation*}
-\int_{\Omega}\Delta_{X,z} \tbU\cdot \tbU = \int_{\Omega}\curl \bom \cdot \tbU
\end{equation*}
so that, after integrating by parts we have that
\begin{eqnarray*}
\int_{\Omega}\abs{\nabla_{X,z} \tbU}^2&=&\int_{\pa \Omega}
\left(\pa_{\overrightarrow{n}} \tbU+ \overrightarrow{n}\times \bom\right)\cdot \tbU  +\int_{\Omega} \bom\cdot
\curl \tbU\\
&=&\int_{\pa \Omega} \left(\pa_{\overrightarrow{n}} \tbU+ \overrightarrow{n}\times \curl \tbU \right)\cdot \tbU  +\int_{\Omega} \abs{\bom}^2,
\end{eqnarray*}
where $\overrightarrow{n}$ is the {\it outward} unit normal vector to the boundary
$\Gamma$ of $\Omega$.
At the bottom $\overrightarrow{n}=-N_b$ and $w_b=0$;  since moreover
$$\pa_{\overrightarrow{n}} \tbU_\bott =-\pa_z \tbU_\bott
\quad\mbox{ and }\quad
{(\curl \tbU)_h}_\bott=\dz \widetilde{\bf V}^\perp_\bott,$$
one easily gets that $\left(\pa_{\overrightarrow{n}} \tbU+
  \overrightarrow{n}\times \curl \tbU \right)_\bott=0$ and therefore
\begin{eqnarray}
\nonumber
\int_{\Omega} \abs{\nabla_{X,z} \tbU}^2 &=& \int_{\rm surf} \left(\pa_n
  \tbU+ n\times \curl\tbU \right)\cdot \tbU  +\int_{\Omega}
\abs{\bom}^2\\
\label{gaz}
&=& \int_{\R^d} \left(N\cdot \nabla_{X,z}
  \tbU+ N\times \curl\tbU \right)_\surf\cdot \widetilde\uU  +\int_{\Omega}
\abs{\bom}^2.
\end{eqnarray}
In order to evaluate the boundary integral, let us remark that
$$
(N\cdot \nabla_{X,z}
  \tbU)_\surf=\left(\begin{array}{l}
-\nabla\zeta\cdot
\nabla\widetilde{\underline{V}}+(1+\abs{\nabla\zeta}^2)\dz\widetilde{\bf
  V}_\surf\\
-\nabla\zeta\cdot (\nabla \widetilde{\bf w})_\surf+\dz\widetilde{\bf w}_\surf
\end{array}\right)
$$
and
$$
\left(N\times \curl\tbU \right)_\surf=\left(\begin{array}{l}
-\dz\widetilde{\bf V}_\surf+\nabla\widetilde{\underline{w}}-\nabla\zeta \dz \widetilde{\bf
  w}_\surf\\
\hspace{1.55cm}-\nabla\cdot \widetilde{\underline{{V}}}^\perp\nabla^\perp
\zeta+\nabla\zeta\cdot \dz \widetilde{\bf V}^\perp \nabla^\perp \zeta\\
-\nabla\zeta\cdot \dz \widetilde{\bf V}_\surf+\nabla\zeta\cdot (\nabla \widetilde{\bf w})_\surf
\end{array}\right),
$$
and use the identity
$$
\abs{\nabla\zeta}^2\dz \widetilde{\bf V}_\surf=(\nabla\zeta\cdot\dz \widetilde{\bf
  V}_\surf )\nabla\zeta+(\nabla\zeta^\perp\cdot\dz \widetilde{\bf V}_\surf )\nabla\zeta^\perp
$$
to obtain
\begin{eqnarray*}
\left(N\cdot \nabla_{X,z}
  \tbU+ N\times \curl\tbU \right)_\surf&=&\big(N\cdot  \dz \widetilde{\bf
  U}_\surf)N\\
&+&
\left(\begin{array}{l}
\nabla\widetilde{\underline{w}}-(\nabla\zeta\cdot \nabla) \widetilde{\underline{
  V}}-\nabla\cdot \widetilde{\underline{V}}^\perp \nabla^\perp \zeta\\
0
\end{array}\right).
\end{eqnarray*}
Using the fact that $\dive \tbU=0$ we now remark that
\begin{eqnarray*}
\big(N\cdot  \dz \widetilde{\bf U}_\surf)&=&-\dz \widetilde{\bf V}_\surf\cdot
\nabla\zeta+\dz \widetilde{\bf w}_\surf\\
&=&-\dz \widetilde{\bf V}_\surf\cdot
\nabla\zeta-(\nabla\cdot \widetilde{\bf V})_\surf\\
&=&-\nabla\cdot \widetilde{\underline{V}},
\end{eqnarray*}
and we can then deduce the following expression
\begin{eqnarray}
\nonumber
\left(N\cdot \nabla_{X,z}
  \tbU+ N\times \curl\tbU \right)_\surf&=&\left(\begin{array}{l}
\nabla\widetilde{\underline{w}}-(\nabla^\perp\zeta\cdot \nabla)\widetilde{\underline{
  V}}^\perp\\
-\nabla\cdot \widetilde{\underline{V}}
\end{array}\right)\\
\label{defF}
&:=&F
\end{eqnarray}
Going back to \eqref{gaz}, this gives
\begin{eqnarray}
\label{noel}
\int_{\Omega} \abs{\nabla_{X,z} \tbU}^2&=& \int_{\R^d} F\cdot \widetilde{\uU}  +\int_{\Omega}
\abs{\bom}^2\\
\label{noel1}
&=&\int_{\R^d} 2 \widetilde{\underline{V}}\cdot \nabla\widetilde{\underline{{w}}}-(\nabla^\perp\zeta\cdot \nabla)\widetilde{\underline{
  V}}^\perp\cdot \widetilde{\underline{V}}+\int_{\Omega}
\abs{\bom}^2.
\end{eqnarray}
Noting that $\widetilde{\underline{
    V}}=\nabla^\perp\tpsi-\widetilde{\underline{w}}\nabla\zeta$ and integrating by parts if necessary, all the components of the integrand
of the first term can be put under the form
$$
P(\zeta)\partial_i \nabla^\perp \tpsi \tU_j ,\quad
Q(\zeta) (\nabla\tpsi)_i\tU_j
\quad\mbox{ or }\quad
Q(\zeta) \widetilde{\underline{w}}^2
\qquad (1\leq i,j\leq d),
$$
with $P(\zeta)$ (resp. $Q(\zeta)$) a generic notation for a polynomial
in the first (resp. and second) order derivatives of $\zeta$.
We deduce therefore by standard product estimates that
\begin{eqnarray}
\nonumber
\Abs{\nabla_{X,z}\tbU}_2^2 &\leq& C(\abs{\zeta}_{W^{2,\infty}})
\big(\abs{\nabla\tpsi}_{H^{1/2}}\abs{\widetilde\uU}_{H^{1/2}}+
\abs{\nabla\tpsi}_{2}\abs{\widetilde\uU}_{2}+\abs{\widetilde{\underline{
    w}}}_2^2\big)+\Abs{\bom}_2^2\\
\label{H1estint}
&\leq&C(\abs{\zeta}_{W^{2,\infty}})
\big(\abs{\nabla\tpsi}_{H^{1/2}}\abs{\widetilde\uU}_{H^{1/2}}+\abs{\widetilde{\underline{w}}}_2^2\big)+\Abs{\bom}_2^2.
\end{eqnarray}
Remarking now that
\begin{align*}
\abs{\widetilde{\underline{w}}}_2^2
&=2\int_\Omega \widetilde{\bf w}\partial_z \widetilde{\bf w}\\
&\leq C(H_0,\abs{\zeta}_\infty)\Abs{\widetilde{\bf w}}_2\Abs{\dz \widetilde{\bf w}}_2
\end{align*}
and recalling that the trace lemma and Lemma \ref{lempoincare} yield
$$
\abs{\widetilde\uU}_{H^{1/2}}\leq C(\abs{\zeta}_{W^{1,\infty}})\Abs{\tbU}_{H^1},
$$
we easily deduce from \eqref{H1estint} and Young's inequality that
$$
\Abs{\tbU}_{H^1}^2 \leq C(H_0,\abs{\zeta}_{W^{2,\infty}})\big(\abs{\nabla\tpsi}_{H^{1/2}}^2+\Abs{\tbU}_2^2+\Abs{\bom}_2^2\big).
$$
Using the upper bound on $\Abs{\tbU}_2$ provided by \eqref{L2est},
we get
$$
\Abs{\tbU}_2+\Abs{\nabla_{X,z}\tbU}_2\leq  C\big(\abs{\zeta}_{W^{2,\infty}},H_0,\frac{1}{h_{\min}}\big)\big(\Abs{\bom}_{2,b}+\abs{\nabla\tpsi}_{H^{1/2}}),
$$
and the $H^1$-estimate of the lemma follows from Lemma \ref{existtildepsi}.
\end{proof}
Letting
$$
\bU=\tbU+\nabla_{X,z}\Phi,
$$
with $\Phi$ solving \eqref{eqPhi}, we have therefore constructed a
solution to \eqref{divrot}. Uniqueness of this solution follows from
Lemma \ref{lemlem}, while the estimate given in the statement of the
theorem follows from Lemma \ref{lemlem} and the estimate
$$
\Abs{\nabla_{X,z}\Phi}_{H^1}\leq C(\frac{1}{h_{\min}},\abs{\zeta}_{W^{2,\infty}})\abs{\nabla\psi}_{H^{1/2}},
$$
which is a well known estimate from the study of irrotational water
waves (see for instance Corollary 2.40 of \cite{Lannes_book}).

\subsection{Inverting the $\curl$ operator}\label{sectinvertcurl}

Theorem \ref{prop1} can be used to construct an inverse of the
curl operator on the space of divergence free vector fields.
\begin{cor}\label{invertcurl}
Let $\zeta\in W^{2,\infty}(\R^d)$ be such that \eqref{hmin} is
satisfied and $\Omega$ be as in \eqref{defdomain}. Let also ${\bf
  C}\in H(\mbox{\textnormal{div}}_0,\Omega)$ be such that $C_b\cdot
N_b=0$. Then there exists a unique solution ${\bf B}\in H^1(\Omega)^3$
to the boundary value problem
$$
\left\lbrace
\begin{array}{lll}
\curl {\bf B} =& {\bf C}&\quad \mbox{ in }\quad \Omega\\
\dive {\bf B} = &0&\quad \mbox{ in }\quad \Omega\\
B_\parallel=& \nabla^\perp \Delta^{-1}(\underline{C}\cdot N)&\quad\mbox{ at the surface}\\
B_b=&0 &\quad\mbox{ at the bottom}
\end{array}\right. ;
$$
we denote this solution ${\bf B}=\mbox{\rm curl}^{-1}{\bf C}$.
\end{cor}
\begin{proof}[Proof of the corollary]
A direct application of Theorem \ref{prop1} (with $\psi=0$) provides
existence and uniqueness of a solution to the same boundary problem as
in the statement of the corollary, but with the bottom boundary condition
replaced by $B_b\cdot N_b=0$. The fact that $B_b\times N_b=0$ comes
from the observation that
$$
(B_b\times N_b)^\perp=\nabla^\perp\Delta^{-1}(C_b\cdot N_b),
$$
which is equal to zero since $C_b\cdot N_b=0$ by assumption.
\end{proof}

\subsection{The straightened div-curl problem}\label{sectstraight}

We know from Theorem \ref{prop1} that there exists a unique solution $\bU\in H^1(\Omega)^{3}$ of
the problem
$$
\left\lbrace
\begin{array}{lll}
\curl \bU =& \bom&\quad \mbox{ in }\quad \Omega\\
\dive \bU = &0&\quad \mbox{ in }\quad \Omega\\
\Vp=&
\nabla\psi+\nabla^\perp\Delta^{-1}(\uom\cdot N)&\quad\mbox{ at the surface}\\
U_b\cdot N_b=&0 &\quad\mbox{ at the bottom}.
\end{array}\right.
$$
The study of the regularity properties of this solution is made easier by
working with a transformed equivalent div-curl problem on the flat
strip $\cS=\R^d\times (-H_0,0)$.
We therefore introduce the diffeomorphism $\Sigma: \cS\to \Omega$ mapping the flat
strip $\cS$
onto the fluid domain and defined as
$$
	\Sigma(t,\cdot):
	\begin{array}{lcl}
	\cS &\to&\Omega\\
	(X,z) &\mapsto & \Sigma(t,X,z)=(X, z+\sigma(t,X,z)),
	\end{array}
$$
where $\sigma$ is the scalar function
$$
        \sigma(X,z)=\frac{1}{H_0}(z+H_0)\zeta.
$$
Defining
$$
U=\left(\begin{array}{c} V\\ w\end{array}\right)=\bU\circ \Sigma, \qquad \omega=\bom\circ \Sigma,
$$
one readily gets that $\bU\in H^1(\Omega)^3$ is the unique solution to
the above div-curl problem if and only if $U\in H^1(\cS)^3$ is the unique
solution to the transformed div-curl problem
\begin{equation}\label{div-rotS}\left\lbrace
\begin{array}{lll}
\curls U =& \omega&\quad \mbox{ in }\quad \cS\\
\dives U = &0&\quad \mbox{ in }\quad \cS\\
\Vp=&
\nabla\psi+\nabla^\perp\Delta^{-1}(\uom\cdot N)&\quad\mbox{ at the surface}\\
U_b\cdot N_b=&0 &\quad\mbox{ at the bottom},
\end{array}\right.
\end{equation}
where we use the notations
\begin{equation}\label{notstar}
\curls U=(\curl \bU)\circ\Sigma=\nabla^\sigma_{X,z}\times U,\qquad \dives
U=(\dive \bU)\circ\Sigma=\nabla_{X,z}^\sigma\cdot U,
\end{equation}
with $\nabla_{X,z}^\sigma$ given by
\begin{equation}\label{defA}
\nabla_{X,z}^\sigma=(J_\Sigma^{-1})^T\nabla_{X,z},\quad\mbox{ with }\quad
(J_\Sigma^{-1})^T=\left(\begin{array}{cc} \mbox{Id}_{d\times d} &
    \frac{-\nabla \sigma}{1+\dz \sigma}\\0 &  \frac{1}{1+\dz \sigma}
\end{array}\right)
\end{equation}
($J_\Sigma=d_{X,z}\Sigma$ is the Jacobian matrix of the diffeomorphism
$\Sigma$). More generally, if $F={\bf F}\circ \Sigma$, we define using
the convenient notations of \cite{MasmoudiRousset}
\begin{equation}\label{notaMR1}
\partial_i^\sigma F=\partial_i^\sigma {\bf F}\circ \Sigma\qquad (i=t,x,y,z)
\end{equation}
and therefore
\begin{equation}\label{notaMR2}
\partial_i^\sigma=\partial_i-\frac{\partial_i\sigma}{1+\dz \sigma}\dz
\qquad (i=t,x,y)
\quad\mbox{ and }\quad
\dz^\sigma=\frac{1}{1+\dz \sigma}\dz.
\end{equation}
\begin{nota}\label{notastr}
We denote by ${\mathbb U}^\sigma[\zeta](\psi,\omega)$ the solution to the
straightened div-curl problem \eqref{div-rotS}. According to
Definition \ref{defimappings}, one has
$$
{\mathbb U}^\sigma[\zeta](\psi,\omega)={\mathbb
  U}[\zeta](\psi,\bom)\circ \Sigma,
$$
and we have the decomposition
\begin{equation}\label{decompS}
{\mathbb U}^\sigma[\zeta](\psi,\omega)={\mathbb U}_I^\sigma[\zeta]\psi+{\mathbb U}_{II}^\sigma[\zeta]\omega,
\end{equation}
with
$
{\mathbb U}_I^\sigma[\zeta]\psi={\mathbb U}_I[\zeta]\psi\circ\Sigma
$ and ${\mathbb U}^\sigma_{II}[\zeta]\omega={\mathbb U}_{II}[\zeta]\bom\circ\Sigma$.
\end{nota}
\begin{remark}\label{remdec}
Straightening the boundary value problems \eqref{eqA} and \eqref{eqPhi},
${\mathbb U}_I^\sigma[\zeta]\psi$ and ${\mathbb
  U}_{II}^\sigma[\zeta]\omega$ can be alternatively defined as
$$
{\mathbb U}_I^\sigma[\zeta]\psi=\nabla_{X,z}^\sigma\phi\quad\mbox{ and
}\quad {\mathbb U}_{II}^\sigma[\zeta]\omega=\curls A
$$
with
\begin{equation}\label{eqPhiS}
\left\lbrace
\begin{array}{l}
\dsp \nabla_{X,z} \cdot P(\Sigma)\nabla_{X,z} \phi=0 \quad\mbox{ in }\quad \cS\\
\dsp \phi_\surff=\psi,\qquad {\bf e}_z\cdot P(\Sigma)\nabla_{X,z} \phi_\bottf=0,
\end{array}\right.
\end{equation}
where $P(\Sigma)=(1+\dz\sigma) J_\Sigma^{-1} (J_\Sigma^{-1})^T$, and
\begin{equation}\label{eqcurlAS}
\left\lbrace
\begin{array}{rl}
\dsp \curls \curls A&=\omega  \quad\mbox{ in }\quad \cS\\
\dsp \dives A&=0  \quad\mbox{ in }\quad \cS\\
\dsp N_b\times A_b&=0\\
\dsp N\cdot \uA&=0\\
\dsp (\curls
  A)_\parallel&=\nabla^\perp \Delta^{-1}\uom\cdot N\\
\dsp N_b\cdot\curls A_b & = 0.
\end{array}\right.
\end{equation}
\end{remark}
\subsection{Higher order estimates}\label{secthigher}

 All the properties that
can be established for $\bU$ have of course a counterpart on the
straightened velocity field $U={\mathbb U}^\sigma[\zeta](\psi,\omega)$ introduced in the previous section. For instance, proceeding as for
\eqref{noel}, we know that
$$
\forall {\bf C}\in H^1(\Omega),\quad \int_\Omega\nabla_{X,z}\bU\cdot
\nabla_{X,z}{\bf C}=\int_{\Omega}\bom\cdot \curl{\bf C}+\int_{\R^d}F\cdot \underline{C},
$$
with $\underline{C}={\bf C}_\surf$ and  $F$ as in \eqref{defF}. Working on the straightened domain $\cS$,
this is equivalent to saying that
\begin{equation}\label{transfineq}
\forall {C}\in H^1(\cS),\quad \int_\cS\nabla_{X,z} U\cdot
P(\Sigma)\nabla_{X,z}{C}=\int_{\cS}(1+\dz\sigma)\omega\cdot \curls{C}+\int_{\R^d}F\cdot
\underline{C},
\end{equation}
with $P(\Sigma)=(1+\dz\sigma) J_\Sigma^{-1} (J_\Sigma^{-1})^T$ and
\begin{align*}
\nabla_{X,z} U\cdot
P(\Sigma)\nabla_{X,z}{C}:=\sum_{i,j,k} \pa_jU^iP(\Sigma)_{jk}\pa_kC^i.
\end{align*}
The $H^1$-estimate on $\bU$ given
in Theorem \ref{prop1} yields the following $H^1$-estimate on $U$,
\begin{equation}\label{estredresse}
\Abs{U}_2+\Abs{\nabla_{X,z}U}_2^2\leq
C(\frac{1}{h_{\min}},H_0,\abs{\zeta}_{W^{2,\infty}})\big(\Abs{\omega}_{2,b}+\abs{\nabla\psi}_{H^{1/2}}\big);
\end{equation}
the theorem below provides higher order estimates on $U$, that is,
estimates on $\Abs{\Lambda^k\nabla_{X,z}U}$ for $k\in \N$. It is
important to remark that this control is not given in terms
of $\abs{\Lambda^k \nabla\psi}_{H^{1/2}}$ but in terms of $\abs{\proj
  \psi_{(\alpha)}}_2$ ($0<\abs{\alpha}\leq k+1$), where
\begin{equation}\label{defproj}
\proj=\frac{\abs{D}}{(1+\abs{D})^{1/2}}
\end{equation}
(so that $\abs{\proj
  \psi_{(\alpha)}}_2\sim \abs{\nabla\psi_{(\alpha)}}_{H^{-1/2}}$), while the {\it good unknowns} $\psi_{(\alpha)}$ are defined as
\begin{equation}\label{defVpb}
\psi_{(\alpha)}=\partial^\alpha \psi-\uw\partial^\alpha\zeta
\end{equation}
(for $\alpha=0$, we take $\psi_{(\alpha)}=\psi$). We also recall that the spaces $H^{s,k}$ have been defined
in \eqref{defHsk}: the norm $\Abs{\cdot}_{H^{s,k}}$ controls a total
number of $s$ derivatives, including at most $k$ vertical ones. Note
also that, with
the convention that a summation over an empty set is equal to zero, the
estimate of the proposition coincides of course with the estimate of
Theorem \ref{prop1} when $k=0$.
\begin{proposition} \label{horizontalregurality} Let $N\in \N$, $N\geq 5$ and
  $\zeta\in H^N(\R^d)$ be such that \eqref{hmin} is
satisfied.  Under the assumptions of Theorem \ref{prop1},
  there is a unique solution $U\in H^1(S)$ to \eqref{div-rotS}; if
  moreover $0\leq k\leq N-1$ and $\Lambda^k\omega \in L^2(\cS)$ then
  the following higher order estimates hold
$$
\Abs{U}_{H^{k+1,1}}\leq M_N \big(\abs{\proj\psi}_{H^{1}}+\sum_{\alpha\in
  \N^d, 1 <\abs{\alpha}\leq k+1}\abs{\proj\psi_{(\alpha)}}_2
+\Abs{\Lambda^k\omega}_{2,b}\big),
$$
with $\dsp M_N=C(\frac{1}{h_{\min}},H_0,\abs{\zeta}_{H^N})$.
\end{proposition}
\begin{proof}
In order to study the regularity of $U$ we take\footnote{Of course,
  this is not correct since $\partial^{2\beta}U$ has not the
  $H^1$-regularity of the test functions $C$ in \eqref{transfineq}. One
should instead consider a regularization of $\partial^{2\beta}U$, such
as, for instance $\chi(\delta \abs{D})\partial^{2\beta}U$, where
$\delta>0$ and
$\chi$ is a smooth, compactly supported function, which is equal to
one in a neighborhood of the origin. One should prove all the estimates
for this regularized version and then deduce the result by letting
$\delta \to 0$. We omit this technical step, and refer to the proof
of Lemma 2.39 in \cite{Lannes_book} where the details are provided in
a similar context.}
$C=\partial^{2\beta}U$ (with $\beta\in \N^d$ and $k=\abs{\beta}\leq N-1$) in
\eqref{transfineq},
\begin{equation*}
\int_{S}\nabla_{X,z} U\cdot P(\Sigma) \nabla_{X,z}
\partial^{2\beta}U=\int_{S}(1+\dz\sigma)\omega\cdot
\curls \partial^{2\beta} U +\int_{\R^d}F\cdot\partial^{2\beta}\underline{U}.
\end{equation*}
Integrating by parts and denoting $\Lambda=(1-\Delta)^{1/2}$, we get
\begin{align*}
\int_{\cS}\nabla_{X,z}\partial^\beta U &\cdot P(\Sigma)\nabla_{X,z}\partial^\beta
U
=\int_{\cS} \nabla_{X,z}\partial^\beta U \cdot
[\partial^\beta,P(\Sigma)]\nabla_{X,z}U\\
&
+\int_{\cS} \Lambda^{k} \omega \cdot \Lambda^{-k}\big((1+\dz\sigma )\curls \partial^{2\beta} U\big)+\int_{\R^d}\partial^\beta F\cdot\partial^\beta \uU\\
&= I_1+I_2+I_3.
\end{align*}
Since $P(\Sigma)$ is coercive in the sense that
\begin{equation}\label{Pcoerc}
\forall \Theta\in \R^{d+1},\qquad \abs{\Theta}^2\leq C(\frac{1}{h_{\min}},\abs{\zeta}_{W^{1,\infty}})\Theta\cdot P(\Sigma)\Theta
\end{equation}
(see Lemma 2.27 of \cite{Lannes_book} or Lemma 2.5 in \cite{Lannes_JAMS}), we have that
\begin{equation}\label{eqI0}
\Abs{\partial^\beta \nabla_{X,z}U}_2^2\leq  C(\frac{1}{h_{\min}},\abs{\zeta}_{W^{1,\infty}})\big(I_1+I_2+I_3),
\end{equation}
and we therefore need upper bounds on $I_i$ ($1\leq i\leq 3$).\\
- {\it Upper bound for $I_1$}.  Denoting
$Q(\Sigma)=P(\Sigma)-\mbox{Id}$, we deduce from the standard
commutator estimate
$$
\forall f \in H^{(t_0+1)}\cap H^{k}(\R^d),\quad \forall g\in
H^{k-1}(\R^d), \qquad \abs{[\partial^\beta ,f]g}_2\lesssim
\abs{f}_{H^{(t_0+1)\vee k}}\abs{g}_{H^{k-1}},
$$
(with $t_0>d/2$) that
\begin{align}
\nonumber
\abs{I_1}&\leq \Abs{\partial^\beta
  \nabla_{X,z}U}_2\Abs{Q(\Sigma)}_{L^\infty
  H^{(t_0+1)\vee
    k}}\Abs{\Lambda^{k-1}\nabla_{X,z}U}_{2}\\
\label{eqI1}
& \leq M_N \Abs{\partial^\beta
  \nabla_{X,z}U}_2\Abs{\Lambda^{k-1}\nabla_{X,z}U}_{2},
\end{align}
where we used the fact that $N\geq t_0+2$ (if $t_0>d/2$
is chosen close enough to $d/2$).\\
- {\it Upper bound for $I_2$}.
We can  write the operator $\Lambda^{-k}\big((1+\dz\sigma )\curls
\partial^{2\beta} \cdot$ as a sum of
vectorial operators with coordinates  of the form
$$
\Lambda^{-k}\partial^{2\beta}\partial_{j_2}\cdot,\quad\Lambda^{-k} (\partial_{j_1}\sigma)^l \partial^{2\beta}\partial_{j_2}\cdot,
\qquad  j_1,j_2=x,y,z, \quad l=1,2.
$$
We can therefore use the product estimate (see \cite{Horm}, p. 240)
\begin{equation}\label{prodHorm}
\forall f\in H^{r_1}(\R^d),\quad g\in H^{r_2}(\R^d),\quad
\abs{fg}_{H^r}\lesssim \abs{f}_{H^{r_1}}\abs{g}_{H^{r_2}},
\end{equation}
for all $r,r_1,r_2\in \R$ such that $r_1+r_2\geq 0$, $r\leq r_j$
($j=1,2$) and $r<r_1+r_2-d/2$ to deduce (taking $r=-k$,
$r_1=k\vee t_0$, $r_2=-k$) that
$$
\Abs{\Lambda^{-k}\big((1+\dz\sigma )\curls
\partial^{2\beta} U\big)}_{2}\leq
C(\frac{1}{h_{\min}},\abs{\nabla_{X,z}\sigma}_{L^\infty H^{k\vee t_0}})
\Abs{\partial^\beta \nabla_{X,z}U}_2.
$$
Using a
simple Cauchy-Schwarz inequality we then get that
\begin{equation}\label{eqI2}
\abs{I_2}\leq M_N\Abs{\Lambda^k \omega}_2\Abs{\partial^\beta\nabla_{X,z}U}_2.
\end{equation}
- {\it Upper bound for $I_3$}. Proceeding as for \eqref{noel1}, we get
\begin{align*}
I_3=&\int_{\R^d} 2\partial^\beta\uV\cdot \nabla\partial^\beta \uw
-(\nabla^\perp\zeta\cdot\nabla)\partial^\beta\uV^\perp\cdot\partial^\beta
\uV-[\partial^\beta,\nabla^\perp\zeta]\cdot\nabla\uV^\perp\cdot\partial^\beta\uV\\
:=&I_{31}+I_{32}+I_{33},
\end{align*}
and we now turn to give upper bounds on the three components of the
right-hand-side. Substituting
\begin{eqnarray*}
\partial^\beta
\uV&=&\partial^\beta\Vp-\partial^\beta (\uw\nabla\zeta)\\
&=&\partial^\beta(\nabla\psi+\nabla^\perp \tpsi)-(\partial^\beta
\uw)\nabla\zeta-\uw\nabla\partial^\beta\zeta-[\partial^\beta,\uw,\nabla\zeta]\\
&=& \nabla\partial^\beta\psi-\uw\partial^\beta\nabla\zeta -(\partial^\beta
\uw)\nabla\zeta+\partial^\beta \nabla^\perp\tpsi
-[\partial^\beta,\uw,\nabla\zeta]
\end{eqnarray*}
 one gets
$$
I_{31}=2\int_{\R^d}(\nabla\partial^\beta \psi-\uw\nabla\partial^\beta \zeta)\cdot
\nabla\partial^\beta\uw+2\int_{\R^d}\big(-\partial^\beta\uw\nabla\zeta-[\partial^\beta,\uw,\nabla\zeta]\big)\cdot\nabla\partial^\beta\uw
$$
(note that the term involving $\nabla^\perp\tpsi$ vanishes).
Using the product estimate
\begin{equation}\label{prodP}
\forall f,g\in H^{1/2}(\R^d),\qquad \int_{\R^d} f\partial_j g \leq
\abs{\proj f}_2\abs{g}_{H^{1/2}} \quad (1\leq j\leq d)
\end{equation}
to control the first term, and integrating by parts in the second one,
we obtain
\begin{align*}
\abs{I_{31}}\leq& 2\abs{\proj\big(
  \nabla \partial^\beta\psi-\uw\nabla\partial^\beta \zeta\big)}_2\abs{\partial^\beta
  \uw}_{H^{1/2}}\\
&+\abs{\Delta\zeta}_{\infty}\abs{\partial^\beta
  \uw}_2^2+\abs{\nabla[\partial^\beta,\uw,\nabla\zeta]}_2 \abs{\partial^\beta
  \uw}_2\\
\lesssim & \abs{\proj \big(\nabla \partial^\beta\psi-\uw\nabla\partial^\beta \zeta\big)}_2\Abs{\partial^\beta
  w}_{H^{1}}+\abs{\zeta}_{H^N} \abs{\Lambda^k \uw}_2^2,
\end{align*}
where the last inequality stems from the trace lemma, the assumption
that $N\geq 5$ and standard product estimates; we also have $\Abs{\partial^\beta w}_{H^1}\lesssim
\Abs{\nabla_{X,z}\partial^\beta w}$ by Poincar\'e's inequality. For the second term of the
right-hand-side, we also use the fact that $\Lambda^k w$ vanishes
at the bottom to write
$$
\abs{\Lambda^k \uw}_{2}^2=2\int_{\cS}
 \Lambda^kw \dz\Lambda^k w
\leq2\Abs{\Lambda^k w}_2\Abs{\dz\Lambda^k w}_2.
$$
We therefore get
$$
\abs{I_{31}}\leq M_N\big(\abs{\proj\big(\nabla \partial^\beta\psi-\uw\nabla\partial^\beta \zeta\big)}_2+\Abs{\Lambda^{k-1}\nabla_{X,z}U}_{2}\big)\Abs{\Lambda^k\nabla_{X,z}U}_2.
$$
For $I_{32}$, we make the same substitution for $\partial^\beta\uV$ as
above to obtain
\begin{align*}
I_{32}=&\int_{\R^d}\nabla^\perp\zeta\cdot\nabla\big[\big(\partial^\beta\nabla\psi-\uw\nabla\partial^\beta\zeta\big)+\partial^\beta\nabla^\perp\tpsi-[\partial^\beta,\uw,\nabla\zeta]\big]
\cdot (\partial^\beta \uV-\partial^\beta\uw \nabla\zeta)^\perp\\
&+\int_{\R^d}\big((\nabla^\perp\zeta\cdot\nabla)\nabla\zeta\cdot \nabla^\perp\zeta\big)\abs{\partial^\beta\uw}^2;
\end{align*}
proceeding as for $I_{31}$ we therefore obtain
\begin{align*}
\abs{I_{32}}\lesssim&M_N \big((\abs{\proj\big(
  \nabla\partial^\beta\psi-\uw\nabla\partial^\beta\zeta\big)}_2 +\abs{\proj\partial^\beta\nabla^\perp\tpsi}_2+\abs{\proj\Lambda^{k-1}w}_2)\abs{\partial^\beta\uU}_{H^{1/2}}+\abs{\partial^\beta\uw}_2^2\big)\\
\leq& M_N \big(\abs{\proj\big(
  \nabla\partial^\beta\psi-\uw\nabla\partial^\beta\zeta\big)}_2+\abs{\proj\partial^\beta\nabla^\perp\tpsi}_2+\Abs{\Lambda^{k-1}\nabla_{X,z}U}_{2}\big)\Abs{\Lambda^k\nabla_{X,z}U}_{2}.
\end{align*}
We therefore need the following lemma.
\begin{lemma}\label{estimtpsi}
Under the assumptions of the proposition, the solution $\tpsi$ to the
equation $\Delta\tpsi=\uom\cdot N$ furnished by Lemma
\ref{existtildepsi} satisfies the estimate
$$
\abs{
  \Lambda^k\nabla\tpsi}_{H^{1/2}} \leq M_N
\Abs{\Lambda^k\omega}_{2,b}.
$$
\end{lemma}
\begin{proof}
Proceeding as in the proof of Lemma \ref{existtildepsi} and using the same notations, we take the $L^2$-scalar
product of the equation $\Delta\tpsi=\uom\cdot N$ with
$\Lambda^{2k+1}\tpsi$ to obtain
\begin{align*}
\abs{
  \Lambda^k\nabla\tpsi}_{H^{1/2}}=&-\int_{\R^d}\Lambda^{2k+1}\tpsi^{\rm
  ext}_b
\omega_b\cdot
N_b-\int_{\cS}(1+\dz\sigma)\nabla_{X,z}^\sigma\Lambda^{2k+1}\tpsi^{\rm
  ext}\cdot \omega\\
=& -\int_{\R^d}\Lambda^{2k+1}\tpsi^{\rm ext}_b
\omega_b\cdot N_b-\int_\cS
(1+\dz\sigma)(J_\Sigma^{-1})^T\nabla_{X,z}\Lambda^{k+1}\tpsi^{\rm ext}\cdot\Lambda^k\omega\\
& -\int_{\cS} \nabla_{X,z}\Lambda^{k+1}\tpsi^{\rm ext}\cdot [\Lambda^k,(1+\dz\sigma)
J_{\Sigma}^{-1}] \omega.
\end{align*}
One readily deduces that
$$
\abs{
  \Lambda^k\nabla\tpsi}_{H^{1/2}} \lesssim \abs{\Lambda^k
  \tpsi_b^{\rm ext}}_{\dot{H}^{1/2}}\abs{\Lambda^k \omega_b\cdot
  N_b}_{H_0^{-1/2}}+ M_N \Abs{\nabla_{X,z}\Lambda^{k+1}\tpsi^{\rm ext}}_2\Abs{\Lambda^k\omega}_2,
$$
and the result follows from the fact that
$$
\abs{\Lambda^k
  \tpsi_b^{\rm
    ext}}_{\dot{H}^{1/2}}+\Abs{\Lambda^{k+1}\nabla_{X,z}\tpsi^{\rm ext}}_2\lesssim
\abs{\Lambda^k\nabla \tpsi}_{H^{1/2}}.
$$
\end{proof}
Using this lemma we obtain the following bound on $I_{32}$,
$$
\abs{I_{32}}
\leq M_N \big(\abs{\proj\big(
  \nabla\partial^\beta\psi-\uw\nabla\partial^\beta\zeta\big)}_2\\
+
\Abs{\Lambda^k\omega}_{2,b}+\Abs{\Lambda^{k-1}\nabla_{X,z}U}_{2}\big)\Abs{\Lambda^k
  \nabla_{X,z}U}_{2}.
$$
For $I_{33}$ we use again the product estimates \eqref{prodHorm} and
\eqref{prodP}, and the trace lemma
to get
\begin{eqnarray*}
\abs{I_{33}}&\leq& M_N \abs{\Lambda^{k-3/2}\nabla\uU}_2
\abs{\Lambda^{1/2}\partial^\beta\uU}_2\\
&\leq& M_N\Abs{\Lambda^{k-1}\nabla_{X,z}U}_{2}\Abs{\Lambda^k\nabla_{X,z}U}_{2}.
\end{eqnarray*}
Gathering the estimates on $I_{31}$, $I_{32}$ and $I_{32}$, we finally get
\begin{equation}
\abs{I_3}\leq M_N \Big(\abs{\proj\big(
  \nabla\partial^\beta\psi\!-\!\uw\nabla\partial^\beta\zeta\big)}_2\\
\label{eqI4}+\Abs{\Lambda^k\omega}_{2,b} )+\Abs{\Lambda^{k-1}\nabla_{X,z}U}_{2}\Big)\Abs{\Lambda^k\nabla_{X,z}U}_{2}.
\end{equation}

\medbreak
We can then deduce from \eqref{eqI0} and \eqref{eqI1}, \eqref{eqI2}
and \eqref{eqI4} that for all $0<k\leq N-1$ and all  $\beta\in \N^d\backslash\{0\}$ such that
$\abs{\beta}\leq k$, one has
\begin{align*}
\Abs{\partial^{\beta}\nabla_{X,z}U}_2^2\leq M_N \Big(\hspace{-0.3cm}\sum_{1<\abs{\alpha}\leq
  k+1}\hspace{-0.3cm}\abs{\proj\psi_{(\alpha)}}_2+\Abs{\Lambda^k\omega}_{2,b}
+\Abs{\Lambda^{k-1}\nabla_{X,z}U}_{2}\Big)  \Abs{\Lambda^k\nabla_{X,z}U}_{2}.
\end{align*}
Summing these inequalities for all $0<\abs{\beta}\leq k$, this
yields a control on $\Abs{\Lambda^{k}\nabla_{X,z}U}_2$,
and using
the $H^1$-estimate furnished by Theorem \ref{prop1} for the
case $\beta=0$, the result follows by a finite induction on $k$.
\end{proof}
Theorem \ref{horizontalregurality} provides an $H^{k+1,1}$-estimate of
$U$. We now deduce a
more general $H^{k+1,l+1}$ estimate of $U$.
\begin{cor}\label{corHkk}
Let $N\in \N$, $N\geq 5$ and
  $\zeta\in H^N(\R^d)$. Under the assumptions of Theorem \ref{prop1},
  there is a unique solution $U\in H^1(S)$ to \eqref{div-rotS}; if
  moreover $0\leq l\leq k\leq N-1$ and $\omega\in H^{k,l}(\cS)$ then
$$
\Abs{U}_{H^{k+1,l+1}}\leq M_N \Big(\abs{\proj\psi}_{H^{1}}+\sum_{1 <\abs{\alpha}\leq
  k+1}\abs{\proj\psi_{(\alpha)}}_2+\Abs{\omega}_{H^{k,l}}
+\abs{\Lambda^k (\omega_b\cdot N_b)}_{H_0^{-1/2}}\Big),
$$
with $\omega_b=\omega_\bottf$ and $M_N$ as in Theorem \ref{horizontalregurality}.
\end{cor}
\begin{proof}
We can rewrite the first two equations of \eqref{div-rotS} under the form
$$
\big[\left(\begin{array}{c}\nabla\\ 0\end{array}\right)+\tilde N\dz^\sigma
\big]\times U=\omega\quad\mbox{ and }\quad
\big[\left(\begin{array}{c}\nabla\\ 0\end{array}\right)+\tilde N\dz^\sigma
\big]\cdot U=0,
$$
with $\tilde N=(-\nabla\sigma^T,1)^T$, and therefore
$$
\tilde N\cdot \dz^\sigma U=-\nabla\cdot V\quad\mbox{ and }\quad
\tilde N\times \dz^\sigma U=\omega+\left(\begin{array}{c}\nabla^\perp
    w\\ -\nabla^\perp\cdot V\end{array}\right).
$$
From the identity
$$
\dz^\sigma U=\frac{1}{1+\abs{\nabla\sigma}^2}\Big(\tilde N\cdot
\dz^\sigma U \tilde N+(\tilde N\times \dz^\sigma U)\times \tilde N\Big),
$$
we deduce
\begin{equation}\label{dzU}
\dz U=\frac{1+\dz\sigma}{1+\abs{\nabla\sigma}^2}\Big(-(\nabla\cdot V)
\tilde N+\omega\times \tilde N  -\left(\begin{array}{c} \nabla
    w+(\nabla^\perp\cdot V) \nabla^\perp\sigma\\
    \nabla\sigma\cdot\nabla w \end{array}\right)\Big).
\end{equation}
This identity will be used to trade one vertical derivative of $U$
with one horizontal one, using the product estimate provided by the
following lemma.
\begin{lemma}\label{lemHkk}
Let $N\in \N$, $N\geq 5$, and $1\leq k\leq N-1$. Then for $f\in
H^{N-1}(\cS)$ and $g\in H^{k}(\cS)$, one has
$$
\forall 0\leq l\leq k,\qquad\Abs{\Lambda^{k-l}\dz^l (fg)}_2\lesssim \Abs{f}_{H^{N-1}} \Abs{g}_{H^{k,l}}.
$$
\end{lemma}
\begin{proof}[Proof of the lemma]
One can decompose $\Lambda^{k-l}\dz^l (fg)$ as a sum of terms of the
form
$$
\Lambda^{k-l} (\dz^{l'}f \dz^{l-l'}g), \quad 0\leq l'\leq l.
$$
Choosing $t_0>d/2$ such that $N>2t_0+3/2$, these terms can then be bounded from above using the product estimate \eqref{prodHorm},
$$
\abs{\Lambda^{k-l} (\dz^{l'}f \dz^{l-l'}g)(z)}_2\lesssim
\abs{\dz^{l'}f (z)}_{A}
\abs{\dz^{l-l'}g(z)}_{B}+\left\langle \abs{\dz^{l'} f(z)}_{B}
  \abs{\dz^{l-l'}g(z)}_{A}\right\rangle_{k-l>t_0},
$$
where the term between brackets should be removed from the r.h.s. when
$k-l\leq t_0$, and
where $(A,B)=(H^{t_0},H^{k-l})$ or $(A,B)=(H^{k-l},H^{t_0})$.
Integrating in $z$, we easily deduce
\begin{eqnarray*}
\Abs{\Lambda^{k-l} (\dz^{l'}f \dz^{l-l'}g)}_2 &\lesssim& \Abs{\dz^{l'}f
  (z)}_{L^a A}
\abs{\dz^{l-l'}g(z)}_{L^b B}\\
& &+\left\langle \abs{\dz^{l'} f(z)}_{L^c B}
  \abs{\dz^{l-l'}g(z)}_{L^d A}\right\rangle_{k-l>t_0},
\end{eqnarray*}
with $(a,b)$ and $(c,d)$ being equal to $(2,\infty)$ or
$(\infty,2)$. The choice of $A$, $B$ and $a$, $b$, $c$ and $d$ depends
on several cases
\begin{enumerate}
\item If $t_0+l'+1/2\leq N-1$ and $k-l+l'+1/2\leq N-1$ then
  $(A,B)=(H^{t_0},H^{k-l})$ and $(a,b)=(c,d)=(\infty,2)$.
\item If $t_0+l'+1/2\leq N-1$ and $k-l+l'+1/2 >N-1$ (and therefore
  $k=N-1$, $l=l'$), then
  $(A,B)=(H^{t_0},H^{k-l})$ and $(a,b)=(\infty,2)$,
  $(c,d)=(2,\infty)$.
\item If $t_0+l'+1/2> N-1$ and $k-l+l'+1/2\leq N-1$ then
  $(A,B)=(H^{k-l},H^{t_0})$ and $(a,b)=(\infty,2)$ (there is no need
  to specify $(c,d)$ since one then has $k-l<t_0$).
\item If $t_0+l'+1/2> N-1$ and $k-l+l'+1/2 > N-1$ then
  $(A,B)=(H^{k-l},H^{t_0})$ and $(a,b)=(c,d)=(2,\infty)$.
\end{enumerate}
The results then follows from the continuous embedding $L^\infty H^{r}\leq
H^{r+1/2,1}$ ($s\in \R$, see for instance Proposition 2.13 of
\cite{Lannes_book}). For instance, in the first case, this yields
\begin{eqnarray*}
\Abs{\Lambda^{k-l} (\dz^{l'}f \dz^{l-l'}g)}_2
&\lesssim&
\Abs{f}_{H^{t_0+l'+1/2,l'+1}}
\Abs{g}_{H^{k-l',l-l'}}\\
& &+\left\langle \Abs{f}_{H^{k-l+l'+1/2,l'+1}}
\Abs{g}_{H^{l-l'+t_0,l-l'+1}}\right\rangle_{k-l>t_0}\\
&\lesssim& \Abs{f}_{H^{N-1}}
\Abs{g}_{H^{k,l}},
\end{eqnarray*}
where we used the assumptions corresponding to the first case. The
other cases are treated similarly.
\end{proof}
Let $1\leq l\leq k$; taking the $H^{k,l}$ norm of \eqref{dzU}, we obtain with the help of
the lemma that
$$
\Abs{\dz U}_{H^{k,l}}\leq
C(\Abs{\sigma}_{H^{N}})\big(\Abs{U}_{H^{k+1,l}}+\Abs{\omega}_{H^{k,l}}\big)
$$
and therefore
$$
\Abs{ U}_{H^{k+1,l+1}}
\leq C(\Abs{\sigma}_{H^{N}})\big(\Abs{U}_{H^{k+1,l}}+\Abs{\omega}_{H^{k,l}}\big).
$$
By a finite induction on $l$, we readily obtain
$$
\Abs{U}_{H^{k+1,l+1}}\leq
C(\Abs{\sigma}_{H^{N}})\big(\Abs{U}_{H^{k+1,1}}+\Abs{\omega}_{H^{k,l}}\big),
$$
and the result then follows from Theorem \ref{horizontalregurality}.
\end{proof}

\subsection{Time derivatives}\label{sectproofpropshape}

For  the analysis of our formulation \eqref{ZCSgen} of the water waves
equations with vorticity, we shall need to control time derivatives of
the solution $U={\mathbb U}^\sigma[\zeta](\psi,\omega)$ to
\eqref{div-rotS}. Such a control cannot be obtained with the same
methods as the control on
space derivatives obtained in the previous section (i.e. by taking
time derivatives of $U$ as test functions in \eqref{transfineq}). We
deal with this issue in this section.
We first need the following notation.
\begin{nota}\label{nothT}
We say that $\zeta\in C([0,T];W^{1,\infty}(\R^d))$ satisfy
\eqref{hmin}$_T$ if \eqref{hmin} is uniformly satisfied by all
$\zeta(t,\cdot)$ with $t\in [0,T]$.
\end{nota}
\begin{proposition}\label{propshapeU}
Let $T>0$ and $\zeta\in C^1([0,T];W^{2,\infty}(\R^d))$
 satisfy
\eqref{hmin}$_T$. Let also $\psi\in C^1([0,T];\dot{H}^{3/2}(\R^d))$ and $\omega \in C^1([0,T];L^2(\cS)^{d+1})$ be
such that $\big(\nabla_{X,z}^{\sigma} \cdot \omega\big) (t)=0$ for all $t\in
[0,T]$ and $\omega_b\cdot N_b\in C^1([0,T];H_0^{-1/2}(\R^d))$. Then one has
$$
\dt  \big({\mathbb U}^\sigma[\cdot](\psi,\omega)\big)={\mathbb
  U}^\sigma[\zeta]\big(\dt\psi -\uw\dt \zeta +\frac{\nabla}{\Delta}\cdot
(\uom^\perp_h \dt\zeta),\dt^\sigma\omega\big)+\dt\sigma
\dz^\sigma {\mathbb U}^\sigma[\zeta](\psi,\omega),
$$
where $(\uV^T,\uw)^T={\mathbb U}^\sigma[\zeta](\psi,\omega)_\surff$.
\end{proposition}
\begin{proof}
The main ingredient in the proof is the following identity
\begin{equation}\label{goodunknown}
\delta (\partial^\sigma_j f)=\partial_j^\sigma(\delta
f-\delta\sigma \dz^\sigma f)+\delta\sigma
\dz^\sigma \partial_j^\sigma f\qquad (j=x,y,z),
\end{equation}
where $\delta$ can be any linearization operator ($\delta=\dt$ here). The quantity $\delta
f-\delta\sigma\dz^\sigma f$ is called Alinhac's good unknown after
\cite{Alinhac}. Its role in the water waves equations was noticed in
\cite{LannesJAMS} but it was in \cite{AlazardMetivier} that its
interpretation as Alinhac's good unknown was understood (see also the
discussion in \cite{MasmoudiRousset}).\\
Decomposing ${\mathbb U}^\sigma[\zeta](\psi,\omega)$ as in
Remark \ref{remdec}, we are led to compute the time derivatives of
${\mathbb U}_I^\sigma[\zeta]\psi$ and ${\mathbb
  U}_{II}^\sigma[\zeta]\omega$:\\
- {\it Computation of $\dt {\mathbb U}_I[\zeta]\psi$}. Recalling
  ${\mathbb U}_I^\sigma[\zeta]\psi=\nabla^\sigma_{X,z}\phi$ with
  $\phi$ solving \eqref{eqPhiS}, we have, according to \eqref{goodunknown},
$$
\dt  {\mathbb U}_I^\sigma[\cdot]\psi=\nabla^\sigma_{X,z}(\dt\phi-\dt\sigma\dz^\sigma
\phi)+\dt\sigma \dz^\sigma\nabla_{X,z}^\sigma \phi.
$$
On the other hand, and after remarking that
$$
(1+\dz\sigma)\nabla_{X,z}^\sigma\cdot\nabla_{X,z}^\sigma=\nabla_{X,z}\cdot
P(\Sigma)\nabla_{X,z},
$$
 we can differentiate \eqref{eqPhiS}
  with respect to time to obtain
$$
\left\lbrace
\begin{array}{l}
\dsp \nabla_{X,z} \cdot P(\Sigma)\nabla_{X,z} (\dt\phi-\dt\sigma \dz^\sigma\phi)=0 \quad\mbox{ in }\quad \cS\\
\dsp
(\dt\phi-\dt\sigma\dz^\sigma\phi)_\surff=(\dt \phi-\dt\zeta
\dz^\sigma \phi)_\surff,\qquad \dz (\dt\phi-\dt\sigma\dz^\sigma\phi)_\bottf=0,
\end{array}\right.
$$
where $P(\Sigma)=(1+\dz\sigma) J_\Sigma^{-1} (J_\Sigma^{-1})^T$, and
where we used the fact that
$\dt\sigma_\surff=\dt\zeta$ and $\dt\sigma_\bottf=0$. It
follows that
\begin{equation}\label{formderUI}
\dt {\mathbb U}_I^\sigma[\cdot]\psi={\mathbb U}_I^\sigma[\zeta](\dt \psi-\uw_I\dt\zeta )
+\dt\sigma \dz^\sigma {\mathbb U}_I^\sigma[\zeta]\psi,
\end{equation}
where $\uw_I$ is the vertical component of ${\mathbb
  U}_I^\sigma[\zeta]\psi$ evaluated at the surface.\\
- {\it Computation of $\dt {\mathbb U}_{II}[\zeta]\omega$}. Recalling that ${\mathbb
    U}_{II}^\sigma[\zeta]\omega=\curls A$ with
  $A$ solving \eqref{eqcurlAS}, we have, according to \eqref{goodunknown},
$$
\dt {\mathbb U}_{II}^\sigma[\cdot]\omega=\curls(\dt A-\dt\sigma\dz^\sigma
A)+\dt\sigma \dz^\sigma \curls A.
$$
Differentiating \eqref{eqcurlAS} with respect to time, we also have
$$
\left\lbrace
\begin{array}{rl}
\dsp \curls\curls (\dt A-\dt\sigma
\dz^\sigma A)&=\dt\omega-\dt\sigma \dz^\sigma\omega\\
\dives(\dt A-\dt\sigma
\dz^\sigma A)&=0,
\end{array}\right.
$$
inside the flat strip $\cS$, together with the boundary conditions
(with the notation $U_{II}=(V_{II}^T,w_{II})^T:=\curls A$)
$$
\left\lbrace
\begin{array}{rl}
\dsp N_b\times  (\dt A-\dt\sigma
\dz^\sigma A)_\bottf&=0\\
\dsp N\cdot  (\dt A-\dt\sigma
\dz^\sigma A)_\surff&=\nabla\dt\zeta\cdot \uA_h-\dt\sigma N\cdot
\dz^\sigma A_\surff\\
\dsp \left( \curls(\dt A-\dt\sigma
\dz^\sigma A)_\surff \right)_\parallel&=\nabla^\perp
\dt\tpsi-\dt\zeta (\dz U_{II})_\parallel-\uw_{II}\nabla\dt\zeta\\
\dsp N_b\cdot\curls  (\dt A-\dt\sigma
\dz^\sigma A)_\bottf & = 0.
\end{array}\right.
$$
In order to simplify the boundary conditions, let us observe that
when evaluated at the surface, the equations $\dives A=0$ and $\curls
U_{II}=\omega$ give
$$
N\cdot \dz^\sigma A_\surff=-\nabla\cdot \uA_h
\quad \mbox{ and }\quad
-(\dz^\sigma U_{II})_\parallel =-\nabla \uw_{II}+\uom_h^\perp
$$
and that
\begin{eqnarray*}
\Delta\dt\tpsi&=&\dt\uom\cdot N-\uom_h\cdot \nabla\dt\zeta\\
&=&(\dt\uom-\pa_t\sigma\dz^\sigma \omega_\surff)\cdot N +(\dz^\sigma \omega)_\surff\cdot N\dt\zeta-\uom_h\cdot
\nabla\dt\zeta\\
&=&(\dt\uom-\pa_t\sigma\dz^\sigma \omega_\surff)\cdot N-\nabla\cdot (\uom_h\dt\zeta),
\end{eqnarray*}
where we used the fact that $\dives \omega=0$ to derive the last
equation. It follows that
$$
\nabla^\perp\dt\tpsi=\frac{\nabla^\perp}{\Delta}\big((\dt\uom-\pa_t\sigma\dz^\sigma \omega_\surff)\cdot N\big)-\Pi_\perp (\uom_h^\perp\dt\zeta),
$$
where we recall that $\Pi=\frac{\nabla\nabla^T}{\Delta}$ and
$\Pi_\perp=\frac{\nabla^\perp (\nabla^\perp)^T}{\Delta}$ are
respectively the
orthogonal projectors onto gradient and orthogonal gradient vector fields.
These three identities imply that the boundary conditions
simplify into
$$
\left\lbrace
\begin{array}{rl}
\dsp N_b\times  (\dt A-\dt\sigma
\dz^\sigma A)_\bottf&=0\\
\dsp N\cdot  (\dt A-\dt\sigma
\dz^\sigma A)_\surff&=\nabla\cdot (\dt\zeta\cdot \uA_h)\\
\dsp \left( \curls(\dt A-\dt\sigma
\dz^\sigma A)_\surff \right)_\parallel&=\Pi(\uom_h^\perp \dt\zeta)-\nabla(\uw_{II}\dt\zeta) +\frac{\nabla^\perp}{\Delta}\big(\dt^\sigma\omega_\surf\cdot N\big)\\
\dsp N_b\cdot\curls  (\dt A-\dt\sigma
\dz^\sigma A)_\bottf & = 0.
\end{array}\right.
$$
Let us now decompose $ (\dt A-\dt\sigma
\dz^\sigma A)$ into
$$
 (\dt A-\dt\sigma
\dz^\sigma A)=B+\nabla^\sigma_{X,z}\varphi,
$$
where $\varphi$ solves
$$
\left\lbrace
\begin{array}{l}
\nabla_{X,z}\cdot P(\Sigma)\nabla_{X,z} \varphi=0 \quad\mbox{ in }\quad \cS,\\
N\cdot \nabla_{X,z}^\sigma\,_\surff \varphi=\nabla\cdot (\dt
\zeta\cdot \uA_h),\qquad N_b\cdot \nabla_{X,z}^\sigma\,_\bottf
\varphi=0;
\end{array}\right.
$$
and $B$ solves therefore the same equations as $\dt A-\dt\sigma
\dz^\sigma A$ but where the second boundary condition is now
homogeneous. It follows that
\begin{eqnarray*}
\curls(\dt A-\dt\sigma\dz^\sigma A)&=&\curls B\\
&=&{\mathbb U}^\sigma[\zeta]\big(\frac{\nabla}{\Delta}\cdot
(\uom^\perp_h \dt\zeta)-\uw_{II}\dt \zeta ,\dt^\sigma \omega\big)
\end{eqnarray*}
and therefore
\begin{equation}\label{formderUII}
\dt {\mathbb U}_{II}^\sigma[\cdot]\omega ={\mathbb U}^\sigma[\zeta]\big(\frac{\nabla}{\Delta}\cdot
(\uom^\perp_h \dt\zeta)-\uw_{II}\dt\zeta ,\dt^\sigma\omega\big)+\dt\sigma
\dz^\sigma {\mathbb U}_{II}^\sigma[\zeta]\omega.
\end{equation}

\medbreak

The proposition is then a direct consequence of \eqref{decompS},
\eqref{formderUI} and \eqref{formderUII}.
\end{proof}

The proposition also allows us to derive the following $H^{N-1}$
control on $\dt U$.
\begin{cor}\label{corotime}
Let the assumptions of Proposition \ref{propshapeU} be satisfied and
let $N\in \N$, $N\geq 5$. Then one has
\begin{align*}
\Abs{\dt U}_{H^{N-1,1}}&\leq C(M_N,\abs{\dt\zeta}_{H^{N-1}}) \\
&\times \Big(\abs{\nabla \dt \psi}_{H^{1/2}}+\sum_{1 <\abs{\alpha}\leq N-1}\abs{\proj\dt
  \psi_{(\alpha)}}_2+\Abs{\Lambda^{N-2}\dt\omega}_{2,b}\\
& \qquad +\abs{\nabla\psi}_{H^{1/2}}+\!\!\!\sum_{1<\abs{\alpha}\leq N}\abs{\proj
  \psi_{(\alpha)}}_2+\Abs{\omega}_{H^{N-1,1}}+\abs{\omega_b\cdot
  N_b}_{H_0^{-1/2}}\Big).
\end{align*}
\end{cor}
\begin{proof}
Introducing the notations
$$
\psi^t=\dt \psi -\uw \dt \zeta+\frac{\nabla}{\Delta}\cdot
(\uom_h^\perp \dt\zeta) \quad\mbox{ and }\quad
\psi^t_{(\alpha)}=\partial^\alpha \psi^t-{\mathbb w}^\sigma[\zeta](\psi^t,\pa_t^\sigma \omega)_\surff\partial^\alpha
\zeta
$$
(with $\alpha\in \N^d$ and ${\mathbb w}^\sigma[\zeta]$ is the vertical
component of the mapping ${\mathbb U}^\sigma[\zeta]$ defined in
Notation \ref{notastr}), we get from Proposition \ref{propshapeU}
that
\begin{eqnarray}
\nonumber
\dt U&=&{\mathbb U}^\sigma[\zeta](\psi^t,\dt^\sigma \omega)+\dt\sigma
\dz^\sigma U\\
\label{decompdtU}
&=&A+B,
\end{eqnarray}
we  therefore turn to control $A$ and $B$:\\
- {\it Control of $A$}. Using Proposition \ref{horizontalregurality},
\begin{equation}
\Abs{A}_{H^{N-1,1}}\leq M_N\big(\abs{\nabla \psi^t}_{H^{1/2}}+\sum_{1 <\abs{\alpha}\leq
  N-1}\abs{\proj\psi^t_{(\alpha)}}_2
\label{CT1}
+\Abs{\Lambda^{N-2}\dt^\sigma\omega}_{2,b}\big).
\end{equation}
We now make the following observations,
\begin{eqnarray}
\nonumber
\abs{\nabla\psi^t}_{H^{1/2}}&\leq& \abs{\nabla\dt
  \psi}_{H^{1/2}}+C(\abs{\dt\zeta}_{W^{2,\infty}})\big(\abs{\uw}_{H^{3/2}}+\abs{\uom_h}_{H^{1/2}}\big)\\
\label{CT2}
&\leq& \abs{\nabla\dt
  \psi}_{H^{1/2}}+C(\abs{\dt\zeta}_{W^{2,\infty}})\big(\Abs{U}_{H^{2,1}}+\Abs{\omega}_{H^{1,1}}\big)
\end{eqnarray}
(the second inequality stemming from the trace lemma), and, for all
$\alpha\in \N^d$, $\abs{\alpha}\leq N-1$,
\begin{align*}
\psi^t_{(\alpha)}&=\dt \psi_{(\alpha)}+\big(\dt \uw -{\mathbb w}^\sigma[\zeta](\psi^t,\pa_t^\sigma
\omega)_\surff\big) \partial^\alpha\zeta
-[\partial^\alpha,\uw]\dt
\zeta+\partial^\alpha\frac{\nabla}{\Delta}\cdot (\uom_h^\perp\dt
\zeta)\\
&=\dt \psi_{(\alpha)}+\big( \dt \zeta \dz^\sigma w_\surff \big) \partial^\alpha\zeta
-[\partial^\alpha,\uw]\dt
\zeta+\partial^\alpha\frac{\nabla}{\Delta}\cdot (\uom_h^\perp\dt
\zeta),\\
\end{align*}
where the formula of Proposition \ref{propshapeU} has been used to
derive the last identity. By standard product estimates and the trace lemma, this yields
\begin{eqnarray}
\nonumber
\abs{\proj\psi^t_{(\alpha)}}_2&\leq&\abs{\proj \dt
  \psi_{(\alpha)}}_2\\
\nonumber
& &+C(\abs{\zeta}_{H^N},\abs{\dt
  \zeta}_{H^{N-3/2}})\big(  \abs{\dz^\sigma w_\surff}_{H^{3/2}}+
\abs{\uw}_{H^{N-1/2}}+\abs{\uom_h^\perp }_{H^{N-3/2}}\big)   \\
\label{CT3}
&\leq& \abs{\proj \dt
  \psi_{(\alpha)}}_2+C(\abs{\zeta}_{H^N},\abs{\dt
  \zeta}_{H^{N-3/2}})\big(\Abs{U}_{H^{N,1}}+\Abs{\omega}_{H^{N-1,1}}\big).
\end{eqnarray}
Using Proposition \ref{horizontalregurality} to control $\Abs{U}_{H^{N,1}}$,
we obtain from \eqref{CT1},
\eqref{CT2} and \eqref{CT3} that
\begin{align*}
&\Abs{A}_{H^{N-1,1}}\leq M_N\big(\abs{\nabla \dt \psi}_{H^{1/2}}+\sum_{ 1<\abs{\alpha}\leq N-1}\abs{\proj\dt
  \psi_{(\alpha)}}_2+\Abs{\Lambda^{N-2}\dt^\sigma\omega}_{2,b}\big)\\
&\quad+C(\abs{\zeta}_{H^N},\abs{\dt \zeta}_{H^{N-\frac32}})\big(\abs{\nabla\psi}_{H^{\frac12}}+\!\!\!\!\!\!\sum_{1 <\abs{\alpha}\leq N}\!\!\abs{\proj
  \psi_{(\alpha)}}_2+\Abs{\omega}_{H^{N-1,1}}+\abs{\omega_b\cdot
  N_b}_{H_0^{-\frac12}}\big).
\end{align*}
- {\it Control of $B$}. We get from the product estimates
  of Lemma \ref{lemHkk} that
\begin{align*}
\Abs{B&}_{H^{N-1,1}}\leq \Abs{\dt \sigma}_{H^{N-1,1}}\Abs{\dz^\sigma
  U}_{H^{N-1,1}}\\
&\leq  C(M_N,\abs{\dt\zeta}_{H^{N-1}}) \big(\abs{\nabla\psi}_{H^{1/2}}+\!\!\!\!\sum_{1 <\abs{\alpha}\leq N}\!\!\!\abs{\proj
  \psi_{(\alpha)}}_2+\Abs{\omega}_{H^{N-1,1}}+\abs{\omega_b\cdot
  N_b}_{H_0^{-1/2}}\big),
\end{align*}
where we used Corollary \ref{corHkk} to derive the
second inequality.

\medbreak

The estimate of the corollary is then a consequence of these two
controls and of the observation that
$$\Abs{\Lambda^{N-2}\dt^\sigma
  \omega}_{2,b}\leq \Abs{\Lambda^{N-2}\dt \omega}_{2,b}+C(M_N,\abs{\dt
  \zeta}_{H^{N-3/2}})\Abs{\omega}_{H^{N-1,1}}.$$
\end{proof}
Another corollary of Proposition \ref{propshapeU} is that
${\mathbb U}^\sigma[\zeta](\psi,\omega)$ has a Lipschitz dependence on
its coefficients.
\begin{cor}\label{lipomega}
Let $N\in \N$, $N\geq 5$. Let also $(\zeta_j,\psi_j,\omega_j)\in
H^N(\R^d)\times \dot{H}^N(\R^d)\times H^{N-2}(\cS)$ be such that
$\nabla_{X,z}^{\sigma_j}\cdot \omega_j=0$ for $j=1,2$. Then one has
\begin{align*}
\Abs{{\mathbb U}^{\sigma_2}[\zeta_2](\psi_2,\omega_2)&-{\mathbb
    U}^{\sigma_1}[\zeta_1](\psi_1,\omega_1)}_{H^{N-2}}\leq
C(\abs{\zeta}_{H^N},\abs{\psi}_{\dot{H}^N},\Abs{\omega}_{H^{N-2}})\\
&\times \big(\abs{\zeta_2-\zeta_1}_{H^N}+\abs{\psi_2-\psi_1}_{\dot{H}^N}+\Abs{\omega_2-\omega_1}_{H^{N-2}}\big).
\end{align*}
\end{cor}
\begin{proof}
Let us define time dependent functions on $[0,1]$ as
$$
\forall t\in [0,1],\qquad \zeta^{(t)}=\zeta_1+t(\zeta_2-\zeta_1),
\qquad
\psi^{(t)}=\psi_1+t (\psi_2-\psi_1).
$$
For every value of $\zeta^{(t)}$, one can define an explicit diffeomorphism
$\Sigma^{(t)}$ as in \S \ref{sectstraight}; we then define
$$
\omega^{(t)}=\big(\bom_1+t(\bom_2-\bom_1)\big)\circ\Sigma^{(t)}
\quad\mbox{ with }\quad
\bom_j=\omega_j\circ \Sigma_j^{-1}\quad (j=1,2);
$$
by construction, one has $\nabla_{X,z}^{\sigma^{(t)}}\cdot
\omega^{(t)}=0$. We can therefore write
$$
{\mathbb U}^{\sigma_2}[\zeta](\psi_2,\omega_2)-{\mathbb
    U}^{\sigma_1}[\zeta_1](\psi_1,\omega_1)=\int_0^1 \dt \big({\mathbb
    U}^{(\sigma^{(t)})}[\zeta^{(t)}](\psi^{(t)},\omega^{(t)}){\rm dt}
$$
and use Proposition \ref{propshapeU} to express the integrand in terms
of the time derivatives of $\zeta^{(t)}$, $\psi^{(t)}$ and
$\omega^{(t)}$. The desired estimate is then a direct consequence of
Corollary \ref{corHkk}.
\end{proof}

\subsection{Almost incompressibility of the good unknown}\label{sectAIG}

We have already mentioned in the proof of Proposition \ref{propshapeU}
the role of Alinhac's good unknown. As we shall see later, it shall also
play an important role in the energy estimates where we shall
typically have to control terms of the form
$$
(\varphi,(\partial^\alpha \uU)\cdot
N)=\int_{\cS}\varphi^\dagger(\nabla_{X,z}^\sigma\cdot
\partial^\alpha U)(1+\dz\sigma)+\int_\cS
\nabla_{X,z}\varphi^\dagger\cdot \partial^\alpha U (1+\dz\sigma),
$$
where $\varphi^\dagger$ is defined in $\cS$ and satisfies
$\varphi^\dagger_{\vert_{z=0}}=\varphi$, and with $U={\mathbb
  U}^\varphi[\zeta](\psi,\omega)$. When $\alpha=0$, the first term in
the right-hand-side vanishes since $U$ is incompressible by
construction. When $\alpha\neq 0$, this is no longer true and this
component cannot be controlled by the $L^2(\cS)$-norm of
$\abs{\alpha}$ derivatives of $U$. It is however possible to get rid of
this difficulty by working with the {\it good unknown} $U_{(\alpha)}$
instead of $\partial^\alpha U$ (and this is actually what we do in \S
\ref{sectEE} below), with
\begin{equation}\label{defUalpha}
\forall \alpha\in \N^d\backslash\{0\},\quad U_{(\alpha)}=\partial^\alpha U-\partial^\alpha \sigma \dz^\sigma U
\end{equation}
while for $\alpha=0$ we simply take $U_{(0)}=U$.
 Indeed, the good unknown is almost incompressible,
as remarked in \cite{MasmoudiRousset} and stated in the proposition
below. We also give an estimate on the curl of the good unknown.
\begin{proposition}\label{propalmostinc}
Let $N\in \N$, $N\geq 5$ and
  $\zeta\in H^N(\R^d)$. Under the assumptions of Theorem \ref{prop1}
  and denoting by $\bU={\mathbb U}[\zeta](\psi,\bom)$ the solution to
  \eqref{divrot} and by $U=\bU\circ\Sigma$ its straightened version, one
  has, for all $\alpha\in \N^d$, $0<\abs{\alpha}\leq N$,
\begin{align*}
\Abs{\nabla_{X,z}^\sigma&\cdot
  U_{(\alpha)}}_{2}+\Abs{\nabla_{X,z}^\sigma\times
  U_{(\alpha)}-\partial^\alpha \omega}_{2}\\
&\leq M_N \big(\abs{\nabla\psi}_{H^{1/2}}+\sum_{1 <\abs{\alpha'}\leq
  \abs{\alpha}}\abs{\proj\psi_{(\alpha')}}_2+\Abs{\omega}_{H^{\abs{\alpha}-1}}+\abs{\omega_b\cdot
N_b}_{H_0^{-1/2}}\big),
\end{align*}
with $M_N$, $\psi_{\alpha}$, $U_{(\alpha)}$ as in Theorem \ref{prop1},
\eqref{defVpb} and \eqref{defUalpha} respectively.
\end{proposition}
\begin{proof}
For the estimate on the divergence, we reproduce here the proof of \cite{MasmoudiRousset}, which is
based on the following identity, with $i=x,y,z$,
$$
\partial^\alpha\partial_i^\sigma f=\partial_i^\sigma\partial^\alpha
f-\dz^\sigma f \partial_i^\sigma \partial^\alpha\sigma+C_i(f),
$$
and where
\begin{eqnarray*}
C_i(f)&=&-[\partial^\alpha,\frac{\partial_i\sigma}{1+\dz\sigma},\dz
f]-[\partial^\alpha,\partial_i\sigma,\frac{1}{1+\dz\sigma}]\dz
f\\
& &-\partial_i\sigma
[\partial^\alpha(\frac{1}{1+\dz\sigma})+\frac{\dz\partial^\alpha\sigma}{(1+\dz\sigma)^2}]\dz f.
\end{eqnarray*}
It follows that
\begin{eqnarray*}
0&=&\partial^\alpha \nabla_{X,z}^\sigma \cdot U\\
&=& \nabla_{X,z}\cdot U_{(\alpha)}+C(U),
\end{eqnarray*}
with $C(U)=C_1(V_1)+C_2(V_2)+C_3(w)$. Using the product estimates of Lemma \ref{lemHkk}, we have in particular
$$
\Abs{C(U)}_2\leq C(\Abs{\sigma}_{H^{N}},\Abs{U}_{H^{N}}),
$$
and the result is therefore a consequence of Corollary
\ref{corHkk}. The estimate on the vorticity is obtained along the same lines.
\end{proof}
We can deduce the following property that, together with its proof, will play an important role
in the derivation of the energy estimates in \S \ref{sectEE}
below.
\begin{cor}\label{cor6}
Under the assumptions of Proposition \ref{propalmostinc}, one has, for
all $\varphi\in H^{1/2}(\R^d)$, and for all $k=x,y$, $\abs{\beta}\leq N-1$ and $\alpha$
such that $\partial^\alpha=\partial_k\partial^\beta$,
\begin{align*}
(\varphi,&\partial_k\uU_{(\beta)}\cdot N)\\
&\leq M_N
\big(\abs{\nabla\psi}_{H^{1/2}}+\sum_{ 1 <\abs{\alpha'} \leq
  \abs{\alpha}}\abs{\proj\psi_{(\alpha)}}_2+\Abs{\omega}_{H^{\abs{\alpha}-1}}+\abs{\omega_b\cdot
  N_b}_{H_0^{-1/2}}\big)\abs{\varphi}_{H^{1/2}}.
\end{align*}
\end{cor}
\begin{proof}
Remarking that, when $\beta\neq 0$,
$$
\partial_k U_{(\beta)}=U_{(\alpha)}-\partial^\beta\sigma \partial_k\dz^\sigma U
$$
(the adaptations to the case
$\beta=0$ are straightforward),
it is enough to prove  the estimate of the corollary on
$(\varphi,\uU_{(\alpha)}\cdot N)$.
Let us first give the following integration by parts formula that will
be used several times in the sequel,
\begin{equation}\label{IPP}
\int_{\cS} (\nabla_{X,z}^\sigma f)\cdot {\bf g} h =-\int_{\cS}
f\nabla_{X,z}^\sigma \cdot {\bf g} h +\int_{z=0} f {\bf
  g}\cdot N -\int_{z=-1} f {\bf g}\cdot N_b,
\end{equation}
where $h=1+\zeta$ (just remark that $h=1+\zeta=1+\dz\sigma$ is the Jacobian determinant
of the diffeomorphism $\Sigma: \cS\to \Omega$ so that this formula is
just the pullback in $\cS$ of the standard integration by parts
formula in $\Omega$). It follows from this formula that
$$
(\varphi,\uU_{(\alpha)}\cdot N)=\int_{\cS}(1+\dz\sigma) \varphi^\dagger\nabla_{X,z}^\sigma
\cdot U_{(\alpha)}+\int_{\cS}(1+\dz\sigma)\nabla_{X,z}^\sigma
\varphi^\dagger\cdot  U_{(\alpha)},
$$
with $\varphi^\dagger$ is the extension of $\varphi$ to $\cS$ given by
$\varphi^\dagger=\chi(z\abs{D})\varphi$, with $\chi$ a smooth, compactly
supported and even function equal to $1$ in a neighborhood of the
origin. Consequently,
$$
(\varphi,\uU_{(\alpha)}\cdot N)\leq
C(\frac{1}{h_{\min}},\abs{\zeta}_{W^{1,\infty}})\Abs{\varphi^\dagger}_{H^1}
\big(\Abs{U_{(\alpha)}}_2+\Abs{\nabla_{X,z}^\sigma\cdot U_{(\alpha)}}_2\big);
$$
since $\Abs{\varphi^\dagger}_{H^1}\lesssim \abs{\varphi}_{H^{1/2}}$,
the result follows from Propositions \ref{horizontalregurality} and \ref{propalmostinc}.
\end{proof}
In \S \ref{sectEE} below, we shall derive a priori estimates on
$\omega$ and on $\Abs{\partial_kU_{\beta}}_2$ ($k=x,y$,
$\abs{\beta}\leq N-1$); the corollary below shall play a crucial role
to deduce a priori estimates on the quantities
$\abs{\proj\psi_{(\alpha)}}_2$ ($\abs{\alpha}\leq N$) more closely
related to the formulation \eqref{ZCSgen} of the water waves equations
with vorticity.
\begin{cor}\label{tracecoer}
Under the assumptions of Proposition \ref{propalmostinc}, one has
$$
\abs{\proj\psi_{(\alpha)}}_2\leq M_N \big(\abs{\proj\psi}_{H^{3}}+\sum_{k=x,y,1\leq \abs{\beta}\leq
  \abs{\alpha}-1}\Abs{\partial_kU_{(\beta)}}_2+\Abs{\omega}_{H^{N-1}}+\abs{\omega_b\cdot
  N_b}_{H_0^{-1/2}}\big).
$$
\end{cor}
\begin{proof}
Since $\abs{\proj\psi_{(\alpha)}}_2\leq \abs{\proj
  \psi_{(\beta)}}_{H^1}+\abs{\zeta}_{H^N}\abs{\uw}_{W^{1,\infty}}$, we
first need an upper bound for $\abs{\proj
  \psi_{(\beta)}}_{H^1} $. Let us remark  that for all $\beta\in \N^d$, $\abs{\beta}\leq N-1$,
\begin{eqnarray*}
U_{(\beta)\parallel}&=&(\partial^\beta
U)_\parallel-\partial^\beta\zeta (\dz^\sigma U)_\parallel\\
&=& \partial^\beta U_\parallel-\uw\nabla\partial^\beta \zeta-[\partial^\beta,\uw,\nabla\zeta]-\partial^\beta\zeta (\dz^\sigma U)_\parallel
\end{eqnarray*}
and therefore, substituting
$U_\parallel=\nabla\psi+\nabla^\perp\tpsi$,
\begin{equation}\label{Uapa0}
U_{(\beta)\parallel}=\nabla\psi_{(\beta)}+\nabla^\perp \partial^\beta\tpsi+\nabla\uw\partial^\beta\zeta-[\partial^\beta,\uw,\nabla\zeta]-\partial^\beta\zeta (\dz^\sigma U)_\parallel,
\end{equation}
from which we get
\begin{eqnarray*}
\abs{\proj\psi_{(\beta)}}_{H^1}&\leq& \abs{\proj
  \frac{\nabla}{\Delta}\cdot
  U_{(\beta)\parallel}}_{H^1}+\abs{
  \nabla\uw\partial^\beta\zeta-[\partial^\beta,\uw,\nabla\zeta]-\partial^\beta\zeta
  (\dz^\sigma U)_\parallel }_{H^{1/2}}\\
&\leq& \abs{\proj
  \frac{\nabla}{\Delta}\cdot
  U_{(\beta)\parallel}}_{H^1}+M_N \Abs{U}_{H^{\abs{\beta}}},
\end{eqnarray*}
the last line being a consequence of the product estimates
\eqref{prodHorm} and the trace lemma. Owing to Corollary
\ref{corHkk}, we then get
\begin{eqnarray}
\nonumber
\abs{\proj\psi_{(\beta)}}_{H^1}&\leq& \abs{\proj
  \frac{\nabla}{\Delta}\cdot
  U_{(\beta)\parallel}}_{H^1}\\
\label{afterB}
& &+ M_N \big(\abs{\proj\psi}_{H^1}+\!\!\!\sum_{1 <\abs{\beta'}\leq \abs{\beta}}\!\!\abs{\proj\psi_{(\beta')}}_2+\Abs{\omega}_{H^{\abs{\beta}-1}}+\abs{\omega_b\cdot
  N_b}_{H_0^{-1/2}}\big).
\end{eqnarray}
By the trace lemma, one gets a control of $\abs{\proj
  \frac{\nabla}{\Delta}\cdot
  U_{(\beta)\parallel}}_{H^1}$ in terms of the $H^{1}$-norm of
  $U_{(\beta)}$. The lemma below shall be used together with
  Proposition \ref{propalmostinc} to get a control involving only
  horizontal derivatives of $U_{(\beta)}$ and therefore more adapted to
  the energy estimates of \S \ref{sectEE}.
\begin{lemma}
Let $\zeta\in W^{2,\infty}(\R^d)$ satisfy \eqref{hmin}. Let also $U\in
H^{1}(\cS)^3$. Then one has
$$
\abs{\proj\frac{\nabla}{\Delta}\cdot U_\parallel}_{H^1}\leq
C(\frac{1}{h_{\rm min}},H_0,\abs{\zeta}_{W^{2,\infty}})\big(\Abs{\Lambda
  U}_2+\Abs{\nabla_{X,z}^\sigma\cdot U}_2+\Abs{\curls U}_{2,b}\big).
$$
\end{lemma}
\begin{proof}
Let us denote by $u$ the solution to the boundary problem
\begin{equation}\label{BVPu}
\left\lbrace
\begin{array}{l}
\nabla_{X,z}\cdot P(\Sigma)\nabla_{X,z}u=(1+\dz\sigma)\nabla_{X,z}^\sigma\cdot
U,\\
u_\surff=0,\qquad {\bf e}_z\cdot P(\Sigma)\nabla_{X,z}u_\bottf=0;
\end{array}\right.
\end{equation}
recalling that $(1+\dz\sigma)\nabla_{X,z}^\sigma\cdot
\nabla_{X,z}^\sigma=\nabla_{X,z}\cdot P(\Sigma)\nabla_{X,z}\sigma$,
the quantity
\begin{equation}\label{3j1}
\tilde U=U-\nabla_{X,z}^\sigma u
\end{equation}
 solves the
div-curl problem
$$
\left\lbrace
\begin{array}{l}
\dsp \nabla^\sigma_{X,z}\times \tilde
U=\nabla^\sigma_{X,z}\times U,\\
\dsp \nabla^\sigma_{X,z}\cdot \tilde
U =0,\\
\dsp \tilde U_\parallel=\nabla(\frac{\nabla}{\Delta}\cdot
U_\parallel)+\nabla^\perp (\frac{\nabla^\perp}{\Delta}\cdot
U_\parallel),\\
\dsp \tilde U_\bottf\cdot N_b=0,
\end{array}\right.
$$
and therefore,
\begin{eqnarray*}
\nonumber
\tilde U&=&{\mathbb
  U}^\sigma[\zeta](\frac{\nabla}{\Delta}\cdot
U_{\parallel},\nabla^\sigma_{X,z}\times U)\\
&=&{\mathbb
  U}_I^\sigma[\zeta]\frac{\nabla}{\Delta}\cdot
U_{\parallel}+{\mathbb
  U}_{II}^\sigma[\zeta]\nabla^\sigma_{X,z}\times U\\
&:=&\tilde U_I+\tilde U_{II}
\end{eqnarray*}
 (using
the notations of \S \ref{sectstraight}). One also gets from the
irrotational theory (e.g. Proposition 3.19 in \cite{Lannes_book}) that
$$
\abs{\proj \frac{\nabla}{\Delta}\cdot U_\parallel}_{H^1}\leq
C(\frac{1}{h_{\rm min}},\abs{\zeta}_{W^{2,\infty}})\Abs{\Lambda \tilde U_I}_2;
$$
and therefore
\begin{eqnarray*}
\abs{\proj \frac{\nabla}{\Delta}\cdot U_\parallel}_{H^1}
&\leq& C(\frac{1}{h_{\rm min}},\abs{\zeta}_{W^{2,\infty}})\big(\Abs{\Lambda
  \tilde U}_2+\Abs{\Lambda \tilde U_{II}}_2\big)\\
&\leq& C(\frac{1}{h_{\rm min}},\abs{\zeta}_{W^{2,\infty}})\big(\Abs{\Lambda
  U}_2+\Abs{\Lambda \nabla_{X,z}^\sigma u}_2+\Abs{\Lambda \tilde U_{II}}_2\big),
\end{eqnarray*}
the last line stemming from \eqref{3j1}. The result follows therefore
from the estimates
$$
\Abs{\Lambda \nabla_{X,z}^\sigma u}_2\leq C(\frac{1}{h_{\rm min}},\abs{\zeta}_{W^{2,\infty}})
\Abs{\nabla_{X,z}^\sigma\cdot U}_2,
$$
which stems from a standard elliptic estimate on \eqref{BVPu}, and
$$
\Abs{\Lambda \tilde U_{II}}_2\leq C(\frac{1}{h_{\rm min}},H_0,\abs{\zeta}_{W^{2,\infty}}) \Abs{\curls U}_{2,b},
$$
which is a direct consequence of  Theorem \ref{prop1}.
\end{proof}
Using the lemma with $U=U_{(\beta)}$, we get that
\begin{align*}
\abs{\proj
  \frac{\nabla}{\Delta}&\cdot
  U_{(\beta)\parallel}}_{H^1}\leq \,C(\frac{1}{h_{\rm min}},\abs{\zeta}_{W^{2,\infty}})\big(\Abs{\Lambda
  U_{(\beta)}}_2+\Abs{\nabla_{X,z}^\sigma\cdot U_{(\beta)}}_2+\Abs{\curls
  U_{(\beta)}}_{2,b}\big)\\
 &\leq \,M_N \big(\Abs{\Lambda
   U_{(\beta)}}_2+\abs{\nabla\psi}_{H^{1/2}}+\!\!\!\sum_{1<\abs{\beta'}\leq
   \abs{\beta}}\abs{\proj\psi_{(\beta')}}_2+\Abs{\omega}_{H^{\abs{\beta}}}
+\abs{\omega_b\cdot N_b}_{H_0^{-1/2}}\big),
\end{align*}
where Proposition \ref{propalmostinc} has been used to derive the
second inequality (without the cancellations obtained by working with
the good unknown, the sum in the right-hand-side would be over all
$1\leq \abs{\beta'}\leq \abs{\beta}+1$ and the induction strategy used
below could not be implemented);
owing to \eqref{afterB}, the same control holds on
$\abs{\proj\psi_{(\beta)}}_{H^1}$. In order to get the result stated
in the corollary, notice that for all $\alpha\in \N^d$, $\alpha\neq
0$, one can write $\partial^\alpha=\partial_k\partial^\beta$ for some
$k=x,y$ and $\abs{\beta}=\abs{\alpha}-1$ so that
$$
\psi_{(\alpha)}=\partial_k
\psi_{(\beta)}+\partial^\beta\zeta \partial_k\uw,
$$
and therefore
\begin{align*}
\abs{\proj\psi_{(\alpha)}&}_2\leq \abs{\proj
  \psi_{(\beta)}}_{H^1}+\abs{\zeta}_{H^N}\abs{\uw}_{W^{2,\infty}}\\
&\leq  M_N \big(\Abs{\Lambda U_{(\beta)}}_2+\abs{\proj\psi}_{H^{3}}+\!\!\!\!\sum_{1<\abs{\beta'}\leq \abs{\beta}}\!\!\abs{\proj\psi_{(\beta')}}_2+\Abs{\omega}_{H^{N-1}}+\abs{\omega_b\!\cdot\! N_b}_{H_0^{-1/2}}\big),
\end{align*}
where we also used Proposition \ref{horizontalregurality} and the
Sobolev embedding $H^N(\cS)\subset W^{2,\infty}(\cS)$ to derive the second line.
Since for all $\abs{\beta}\geq 2$, one can write $U_{(\beta)}=\partial_{k'}
U_{(\beta')}+\partial_{k'} \dz^\sigma U \partial^{\beta'}\zeta$, for some
$k'=x,y$, and $\beta'\in \N^d$, one has
\begin{align*}
\Abs{\Lambda &U_{(\beta)}}_2\leq
\sum_{k=x,y,1\leq \abs{\beta'}\leq
  \abs{\beta}}\Abs{\partial_kU_{(\beta')}}_2+M_N\Abs{U}_{H^{\abs{\beta}}}\\
&\leq M_N\big(
\sum_{k=x,y,1\leq \abs{\beta'}\leq
 \abs{\beta}}\!\!\!\! \Abs{\partial_kU_{(\beta')}}_2+\!\!\!\sum_{1<\abs{\beta'}\leq \abs{\beta}}\!\!\abs{\proj\psi_{(\beta')}}_2 +\Abs{\omega}_{H^{N-1}}+\abs{\omega_b\!\cdot\! N_b}_{H_0^{-1/2}}\big),
\end{align*}
the last line stemming from Corollary \ref{corHkk}. We finally deduce that
\begin{align*}
\abs{\proj\psi_{(\alpha)}}_2\leq M_N \Big(&\abs{\proj\psi}_{H^{3}}+\sum_{k=x,y,1\leq \abs{\beta'}\leq
  \abs{\beta}}\Abs{\partial_kU_{(\beta')}}_2\\
&+\sum_{1 <\abs{\beta'}\leq
  \abs{\beta}}\abs{\proj\psi_{(\beta')}}_2+\Abs{\omega}_{H^{N-1}}+\abs{\omega_b\!\cdot\! N_b}_{H_0^{-1/2}}\Big).
\end{align*}
It follows from a
finite induction that
$$
\abs{\proj\psi_{(\alpha)}}_2\leq M_N \big(\abs{\proj\psi}_{H^{3}}+\!\!\sum_{k=x,y,1\leq \abs{\beta'}\leq
  \abs{\beta}}\!\!\Abs{\partial_kU_{(\beta')}}_2+\Abs{\omega}_{H^{N-1}}+\abs{\omega_b\!\cdot\! N_b}_{H_0^{-1/2}}\big),
$$
which is the desired result.

\end{proof}

\section{Well-posedness of the Hamiltonian water waves equations
  \eqref{ZCSgen}}\label{sect4}

We prove in this section the local well-posedness of the water waves
equations \eqref{ZCSgen}. Let us introduce here the energy that shall
be used to obtain this result,
\begin{align}
\nonumber
\cE^N(\zeta,\psi,\omega):=&\frac{1}{2}\abs{\zeta}_{H^N}^2+\frac{1}{2}\abs{\proj\psi}_{H^3}^2+\frac{1}{2}\sum_{\alpha\in
\N^d,0<\abs{\alpha}\leq
N}\abs{\proj\psi_{(\alpha)}}_2^2\\
\label{defenergy}
&+\frac{1}{2}\Abs{\omega}_{H^{N-1}}^2+\frac{1}{2}\abs{\omega_b\cdot
N_b}_{H_0^{-1/2}},
\end{align}
where we recall that
$\psi_{(\alpha)}=\partial^\alpha\psi-\uw\partial^\alpha\zeta$, that
$\proj$ is defined in \eqref{defproj}, and $H_0^{-1/2}(\R^d)$ in \eqref{defH0m}.
To this energy, we associate the functional space $E_T^N$ defined for
all $T\geq 0$ as
\begin{align}
\nonumber
E_T^N=\{(\zeta,\psi,\omega)\in C\big([0,T];H^2(\R^d)&\times \dot{H}^2(\R^d)\times
H^2(\cS)^3\big), \\
\label{defETN}
& \cE^N((\zeta,\psi,\omega)(\cdot))\in L^\infty([0,T])\}.
\end{align}
We also denote by $\mfN(\zeta,\psi,\omega)$ any constant of the form
\begin{equation}\label{defmfN}
\mfN(\zeta,\psi,\omega)=C\big(\frac{1}{h_{\rm min}},H_0,\cE^N(\zeta,\psi,\omega)\big).
\end{equation}
\subsection{A priory energy estimates for the vorticity}
We recall that the equation for the vorticity $\bom$ is given by
$$
\pa_t \bom +(\bU\cdot\nabla_{X,z})\bom=
(\bom\cdot\nabla_{X,z})\bU\quad \mbox{in }\quad \Omega_t,
$$
with $\bU={\mathbb U}[\zeta](\psi,\bom)$ and ${\mathbb
  U}[\zeta](\psi,\bom)$ as in Definition \ref{defimappings}. As explained in Remark \ref{straightvort},
the vorticity $\bom$ is studied through the {\it straightened
  vorticity} $\omega=\bom\circ \Sigma$, where $\Sigma:\cS\to \Omega_t$
is the diffeomorphism introduced in \S \ref{sectstraight}. Written in terms of $\omega$ rather than $\bom$, using  the notations
\eqref{notaMR1} and recalling the notation ${\mathbb
  U}^\sigma[\zeta](\psi,\omega)=\bU\circ\Sigma$,  the
vorticity equation becomes,
\begin{equation}\label{eq10bis}
\dt^\sigma \omega+{\mathbb U}^\sigma[\zeta](\psi,\omega)\cdot\nabla_{X,z}^\sigma \omega=\omega\cdot\nabla_{X,z}^\sigma {\mathbb U}^\sigma[\zeta](\psi,\omega)
\quad\mbox{in }\quad \cS.
\end{equation}
Solving this equation on a domain with boundaries like the flat strip
$\cS$ requires in general boundary conditions on the vorticity. We show
below that such boundary conditions are not needed to derive a priori
estimates\footnote{The issue of constructing a solution is a bit more
  complex than the derivation of a priori estimates; this aspect will be addressed in the proof of Theorem \ref{theo1}.} if the kinematic equation holds.
\begin{proposition}\label{EEvorticity}
Let $N\in \N$, $N\geq 5$, $T>0$ and $(\zeta,\psi,\omega)\in E^N_T$ be
such that \eqref{hmin}$_T$ and \eqref{eq10bis} hold on $[0,T]$. If moreover
$\dt\zeta-{\mathbb U}^\sigma[\zeta](\psi,\omega)_\surff\cdot N=0$ on $[0,T]$, then
the following estimate holds
\begin{align*}
\frac{d}{dt} \Big(\Abs{\omega}_{H^k}^2+\abs{\omega_b\cdot N_b}_{H_0^{-1/2}}\Big)\leq \mfN(\zeta,\psi,\omega),
\end{align*}
with $\mfN(\zeta,\psi,\omega)$ as defined in \eqref{defmfN} and
$H_0^{-1/2}(\R^d)$ given by \eqref{defH0m}.
\end{proposition}
\begin{proof}
The first step is to rewrite \eqref{eq10bis} under the form
\begin{equation}\label{eqqq}
 \dt\omega+{\mathbb V}^{\sigma}[\zeta](\psi,\omega)\cdot\nabla
 \omega+{\mathbb a}[\zeta](\psi,\omega)\dz\omega=\omega\cdot\nabla_{X,z}^{\sigma} {\mathbb U}^{\sigma}[\zeta](\psi,\omega)
\end{equation}
with, denoting $\tilde N=(-\nabla{\sigma}^T,1)^T$ (so that $\tilde N_{\vert_{z=0}}=N$),
\begin{align*}
{\mathbb a}[\zeta](\psi,\omega)&=\frac{1}{1+\dz\sigma}\Big( {\mathbb
   U}^{\sigma}[\zeta](\psi,\omega)\cdot \tilde N-\dt \sigma\Big)\\
&=\frac{1}{1+\dz\sigma}\Big( {\mathbb
   U}^{\sigma}[\zeta](\psi,\omega)\cdot \tilde N-\frac{z+H_0}{H_0}
 {\mathbb
   U}^\sigma[\zeta](\psi,\omega)_\surff\cdot N\Big)
\end{align*}
where we used the definition of $\sigma$ (see \S \ref{sectstraight}) and the fact that the
kinematic equation is satisfied to derive the second equality. Denoting $U={\mathbb U}^\sigma[\zeta](\psi,\omega)$, $a={\mathbb
  a}[\zeta](\psi,\omega)$ etc., the first step is to
perform an $L^2$ a priori estimate on the equation
\begin{equation}\label{vortL2}
\dt \omega+V\cdot \nabla  \omega +a \dz \omega=f,
\end{equation}
with $f\in L^2(\cS)$. Taking the
$L^2$ scalar product of this equation with
$\omega$, we get
\begin{eqnarray*}
\frac{1}{2}\dt \Abs{\omega}_{L^2}^2-\frac{1}{2}\int_\cS
(\nabla  \cdot V+\dz a)\abs{\omega}^2
=\int_\cS  f \cdot \omega,
\end{eqnarray*}
where we used the fact that $a$ vanishes at the bottom and at the
surface. We deduce that
\begin{eqnarray}
\label{estvortL2}
\dt \Abs{\omega}_{2}^2
&\lesssim& \Abs{U}_{W^{1,\infty}}\Abs{\omega}_{2}^2+\Abs{\omega}_2\Abs{f}_{2}.
\end{eqnarray}
In order to get an $L^2$ a priori estimate for
\eqref{eq10bis} we must take  $f=\omega\cdot
\nabla_{X,z}^\sigma U$, and therefore
$$
\Abs{f}_{L^2}\leq
C(\abs{\zeta}_{W^{1,\infty}})\Abs{\omega}_{2}\Abs{U}_{W^{1,\infty}}
$$
and we can therefore deduce from \eqref{estvortL2} and Proposition
\ref{horizontalregurality} that
$$
\frac{d}{dt}\Abs{\omega}_{2}^2\leq \mfN(\zeta,\psi,\omega).
$$
We now prove higher order a priori estimates. For all $j\in\N$ and
$\beta\in \N^{d}$ with $\abs{\beta}+j\leq N-1$, we get after applying
$\partial^\beta\dz^j$ to \eqref{eqqq},
$$
\dt \partial^\beta\dz^j\omega+V\cdot
\nabla \partial^\beta\dz^j\omega+a\dz \partial^\beta\dz^j \omega=f_{\beta,j},
$$
with
$$
f_{\beta,j}=-[\partial^\beta\dz^j,V]\dz\omega-[\partial^\beta\dz^j,a]\dz\omega+\partial^\beta\dz^j \big(\omega\cdot
\nabla_{X,z}^\sigma U\big).
$$
This equation is of the form \eqref{vortL2} and we therefore get from
\eqref{estvortL2} that
\begin{equation}\label{estvortL3}
\dt \Abs{\partial^\beta\dz^j\omega}_{L^2_\sigma}^2
\lesssim \Abs{U}_{W^{1,\infty}}\Abs{\omega}_2^2+ \Abs{\partial^\beta\dz^j\omega}_{2}\Abs{f_{\beta,j}}_{2};
\end{equation}
with product estimates similar to those of Lemma \ref{lemHkk}, we also
get
$$
\Abs{f_{\beta,j}}_{L^2_\sigma}\leq
C\big(\abs{\zeta}_{H^{N}},\Abs{U}_{H^{N}},\Abs{\omega}_{H^{N-1}}\big)
$$
and therefore
$$
\dt \Abs{\partial^\beta\dz^j\omega}_{L^2_\sigma}^2\leq \mfN(\zeta,\psi,\omega),
$$
which provides the $H^k$-estimate on $\omega$ of the proposition.\\
Evaluating the
vorticity equation at the bottom, one gets as in Remark \ref{remmm}
that
\begin{equation}\label{estbotvort}
\dt (\omega_b\cdot N_b)+\nabla\cdot (\omega_b\cdot N_b V_b)=0,
\end{equation}
and therefore
$$
\dt \abs{\omega_b\cdot N_b}_{H^{-1/2}_0}^2\leq \abs{(\omega_b\cdot
  N_b)V_b}_{H^{1/2}}\abs{\omega_b\cdot N_b}_{H_0^{-1/2}};
$$
as a  consequence of the trace lemma and Corollary \ref{corHkk}, we deduce finally
\begin{equation}\label{88b}
\dt \abs{\omega_b\cdot N_b}_{H^{-1/2}_0}^2\leq \mfN(\zeta,\psi,\omega),
\end{equation}
and the proof is complete.
\end{proof}

\subsection{Quasilinearization of the equations}

The following proposition gives the structure of the equations solved
by the derivatives of the solutions to the water waves equations
\eqref{ZCSgen}. Note that for the equations on $\psi$, it is crucial
to work with Alinhac's good unknown. We also recall that the condition
\eqref{hmin}$_T$ is defined in Notation \ref{nothT}, while the
function space $E^N_T$ is introduced in \eqref{defETN}.
\begin{proposition}\label{prop3}
Let $N\geq 5$, $T>0$ and
$(\zeta,\psi,\omega)\in E^N_T$ be a solution to the water waves
equations \eqref{ZCSgen} on the time interval $[0,T]$. If
\eqref{hmin}$_T$ is satisfied then for all
$k=x,y$, and $\beta\in \N^d$, $\abs{\beta}\leq N-1$, one has, with
$\alpha\in \N^d$ such that
$\partial^\alpha=\partial^\beta \partial_k$.
\begin{eqnarray*}
(\dt +\uV\cdot\nabla)
\partial^\alpha\zeta-\partial_k\uU_{(\beta)}\cdot N&=&R^1_\alpha,\\
(\dt +\uV\cdot \nabla) (U_{(\beta)\parallel}\cdot {\bf e}_k)+\mfa \partial^\alpha
\zeta&=&R^2_\alpha,\\
(\dt^\sigma+U\cdot\nabla_{X,z}^\sigma)\partial^\beta \omega&=&R^3_\beta,
\end{eqnarray*}
with $\mfa=g+(\dt+\uV\cdot\nabla)\uw$ and where
$$
\abs{R^1_\alpha}_2+\abs{R^2_\alpha}_2+\abs{\proj R^2_\alpha}_2+\Abs{R^3_\beta}_2\leq \mfN(\zeta,\psi,\omega),
$$
with $\mfN(\zeta,\psi,\omega)$ as in \eqref{defmfN}.
\end{proposition}
\begin{proof}
Let us consider the first equation of \eqref{ZCSgen}. We use the
notation
$$
f\sim 0 \quad\mbox{ iff } \quad \abs{f}_2\leq \mfN(\zeta,\psi,\omega).
$$
 Applying
$\partial^\alpha$ to the first equation of \eqref{ZCSgen}, we get
$$
\dt\partial^\alpha\zeta-\partial^\alpha \uU\cdot
N-\uU\cdot \partial^\alpha N=[\partial^\alpha,\uU,N].
$$
One also has
$$
[\partial^\alpha,\uU,N]=\sum_{0 < \beta
  <\alpha}*_\beta \partial^{\alpha-\beta}\uU\cdot \partial^\beta N,
$$
where $*_\beta$ are scalar coefficients of no importance. We deduce
easily from the product estimate \eqref{prodHorm}  and the assumption
$N\geq 5$ that
$$
\abs{[\partial^\alpha,\uU,N]}_2\leq C(\abs{\uU}_{H^{N-1}},\abs{\zeta}_{H^N}),
$$
and therefore, owing to the trace lemma and Proposition \ref{horizontalregurality},
$[\partial^\alpha,\uU,N]\sim 0$. This yields
$$
\dt\partial^\alpha\zeta-\partial^\alpha \uU\cdot
N-\uU\cdot \partial^\alpha N\sim 0.
$$
Since $\uU\cdot \partial^\alpha N=-\uV\cdot\nabla\partial^\alpha
\zeta$ and $\partial^\alpha \uU\cdot N\sim \partial_k\uU_{(\beta)}\cdot N$, one
readily deduces the first estimate of the proposition.\\
We now consider the second equation of \eqref{ZCSgen}.
Remarking that
$$
\uw=\frac{\uU\cdot N+\nabla\zeta\cdot\Vp}{1+\abs{\nabla\zeta}^2},
$$
we can write the second equation of \eqref{ZCSgen} under the form
$$
\dt\psi+g\zeta+\frac{1}{2}\abs{\Vp}^2-\frac{1}{2}(1+\abs{\nabla\zeta}^2)\uw^2-\frac{\nabla^\perp}{\Delta}\cdot(\uom\cdot
N \uV)=0.
$$
After applying $\partial_k$ we obtain that
$$
\dt\partial_k\psi+g\partial_k\zeta+\uV\cdot(\partial_k\Vp-\uw\nabla\partial_k\zeta)-\uw\partial_k(\uU\cdot
N)-\partial_k\frac{\nabla^\perp}{\Delta}\cdot(\uom\cdot
N\uV)=0,
$$
or equivalently, after substituting
$\Vp=\nabla\psi+\nabla^\perp\tpsi$,
$$
\dt\partial_k\psi+g\partial_k\zeta+\uV\cdot\big((\partial_k\nabla\psi-\uw\nabla\partial_k\zeta)+\partial_k\nabla^\perp\tpsi\big)-\uw\partial_k(\uU\cdot
N)-\partial_k\frac{\nabla^\perp}{\Delta}\cdot(\uom\cdot
N\uV)=0.
$$
Before differentiating $\beta$ times this equation, it is convenient to introduce the notation
$$
f\approx g \iff \abs{f-g}_2 +\abs{\proj(f-g)}_2 \leq \mfN(\zeta,\psi,\omega)
$$
and to state the following lemma.
\begin{lemma}
The following identities hold,
\begin{eqnarray*}
[\partial^\beta,\uV]\cdot
\big((\nabla\partial_k\psi-\uw\nabla\partial_k\zeta)+\nabla^\perp\partial_k\tpsi
\big)&\approx& 0,\\
- \left[\partial^\beta,\uw\right]\partial_k(\uU\cdot N)-\underline{V}\cdot [\pa^\beta, \underline{w}]\nabla\pa_k\zeta&\approx& 0.
\end{eqnarray*}
\end{lemma}
\begin{proof}[Proof of the lemma]
Let us first observe that
\begin{eqnarray*}
\abs{\uU}_{{H}^{N-1/2}} &\leq& \mfN(\zeta,\psi,\omega),\\
\abs{[\proj ,f]g}_2&\lesssim& \abs{\nabla f}_{H^{t_0}}\abs{g}_2 ,\\
\abs{[\partial^\beta,f]g}_{{H}^{1/2}}&\lesssim & \abs{f}_{H^{N-1/2}}\abs{g}_{{H}^{N-3/2}};
\end{eqnarray*}
the first estimate is a direct consequence of the trace lemma and
Proposition \ref{horizontalregurality}, while the second and third ones
follow from the general commutator estimates of Theorem 3 in
\cite{Lannes_JFA}. These identities allow us to write
\begin{eqnarray*}
\abs{\proj [\partial^\beta,\uV]\cdot
(\nabla\partial_k\psi+\nabla^\perp\partial_k\tpsi-w\nabla\partial_k
\zeta)}_2&\leq& \mfN(\zeta,\psi,\omega)(1+\abs{\nabla\tpsi}_{H^{N-1/2}})\\
&\leq &\mfN(\zeta,\psi,\omega),
\end{eqnarray*}
the last line being a consequence of Lemma \ref{estimtpsi}; this
yields the first assertion of the lemma. For the second one, we
proceed along the same lines (it is important to note the most
singular terms in $\zeta$ of both terms cancel one each other).
\end{proof}
Applying $\partial^\beta$ (with $\beta\neq 0$) to the equation on $\partial_k\psi$ we therefore
get, using the lemma,
$$
\dt\partial^\alpha\psi+g\partial^\alpha\zeta+\uV\cdot\big((\partial^\alpha\nabla\psi-\uw\nabla\partial^\alpha\zeta)+\partial^\alpha\nabla^\perp\tpsi\big)-\uw\partial^\alpha(\uU\cdot
N)-\partial^\alpha\frac{\nabla^\perp}{\Delta}\cdot
(\uom\cdot
N\uV)\approx 0.
$$
Using the evolution equation on $\zeta$, we can substitute $\partial^\alpha
(\uU\cdot N)=\partial_t\partial^\alpha\zeta $ to obtain
\begin{eqnarray*}
\dt\psi_{(\alpha)}+\mfa\partial^\alpha\zeta+\uV\cdot\nabla
\psi_{(\alpha)}&\approx&\partial^\alpha\frac{\nabla^\perp}{\Delta}\cdot(\uom\cdot
N\uV)-\uV\cdot \nabla^\perp\partial^\alpha\tpsi\\
&=&[\frac{\partial^\alpha\nabla^\perp}{\Delta},\uV](\uom\cdot N).
\end{eqnarray*}
Now, a  general commutator estimate (Theorem 3 in
\cite{Lannes_JFA}) implies that
\begin{eqnarray*}
\abs{[\frac{\partial^\alpha\nabla^\perp}{\Delta},\uV](\uom\cdot
  N)}_{H^{1/2}} &\lesssim& \abs{\uV}_{H^{N-1/2}}\abs{\uom\cdot
  N}_{H^{N-3/2}}\\
&\lesssim& \Abs{V}_{H^{N,1}}\Abs{\omega}_{H^{N-1,1}}\\
&\lesssim&  \mfN(\zeta,\psi,\omega),
\end{eqnarray*}
where we used the trace lemma to derive the second inequality, and
Proposition \ref{horizontalregurality} to get the third one. It
follows that
\begin{equation}\label{eqpsialpha}
\dt\psi_{(\alpha)}+\mfa\partial^\alpha\zeta+\uV\cdot\nabla
\psi_{(\alpha)}\approx 0.
\end{equation}
We now have to relate $\psi_{(\alpha)}$ to the quantity
$U_{(\beta)\parallel} \cdot {\bf e}_k$ used in the statement of the
proposition. Proceeding as for \eqref{Uapa0}, we get
\begin{equation}\label{Uapa}
U_{(\beta)\parallel}\cdot {\bf e}_k=\psi_{(\alpha)}+{\bf
  e}_k\cdot \nabla^\perp \partial^\beta
\tpsi-[\partial^\beta,\uw,\partial_k\zeta]-\partial^\beta\zeta
(\dz^\sigma U)_\parallel\cdot {\bf e}_\kappa.
\end{equation}
We now need the following lemma.
\begin{lemma}
Let $f_1$, $f_2$ and $f_3$ be defined as
$$
f_1={\bf
  e}_k\cdot \nabla^\perp \partial^\beta
\tpsi,\qquad
f_2=[\partial^\beta,\uw,\partial_k\zeta],\qquad
f_3=\partial^\beta\zeta
(\dz^\sigma U)_\parallel\cdot {\bf e}_k;
$$
then one has
$$
(\dt +\uV\cdot \nabla)f_j\approx 0 \qquad (j=1,2,3).
$$
\end{lemma}
\begin{proof}[Proof of the lemma]
Let us first prove the lemma for $f_1$. We recall that
$\Delta\tpsi=\uom\cdot N$ and that we get from the vorticity equation
evaluated at the surface that
$$
(\dt +\uV\cdot \nabla) (\uom\cdot N)+\uom\cdot N
\nabla\cdot \uV=0
$$
(see Remark \ref{remmm}). Applying ${\bf e}_k\cdot
\nabla^\perp\partial^\beta\Delta^{-1}$ to this equation, we get
therefore
\begin{eqnarray*}
(\dt +\uV\cdot \nabla)f_1&=&-{\bf e}_k\cdot
\nabla^\perp\partial^\beta\Delta^{-1}\big(\uom\cdot N\nabla\cdot
\uV\big)\\
&\approx& 0,
\end{eqnarray*}
the second identity stemming from standard commutator estimates, the
trace lemma, and Proposition \ref{horizontalregurality}.\
To treat the case of $f_2$, it is sufficient to prove that
$$
(i)\quad (\dt+\uV\cdot \nabla)\partial^{\beta'}\uw\approx 0 \quad\mbox{ and
}\quad (ii)\quad (\dt+\uV\cdot \nabla)\partial^{\beta'}\partial_k
\zeta\approx 0,
$$
for all $0<\beta'<\beta$. For $(i)$, the space derivative is
controlled through Proposition \ref{horizontalregurality}, while the
time derivative is controlled by Corollary \ref{corotime} in terms of $\mfN$
and $\abs{\proj \dt \psi_{\alpha'}}_2$, with $\abs{\alpha'}\leq
N-1$. Now, \eqref{eqpsialpha} gives a control of this last quantity in
terms of
$\mfN$ since $\mfa \partial^{\alpha'}\zeta\approx 0$ when
$\abs{\alpha'}\leq N-1$; the identity (i) is therefore proved. For
(ii), we just have to remark that $\uU_{(\alpha')}\cdot N\sim 0$
when $\abs{\alpha'}\leq N-1$ (which is the case if we take
$\partial^{\alpha'}=\partial^{\beta'}\partial_k$) and to use the first identity given in
the statement of the proposition and proved above.\\
We finally turn to $f_3$, which is a direct consequence of Proposition
\ref{horizontalregurality}, Corollary \ref{corotime} and (ii) above.
\end{proof}
Owing to this lemma and \eqref{Uapa}, we can replace $\psi_{(\alpha)}$
by $U_{(\beta)\parallel}\cdot {\bf e}_k$ in \eqref{eqpsialpha}, and
the second assertion of the proposition is proved. The third and last
assertion of the proposition is a simple byproduct of the proof of
Proposition \ref{EEvorticity}.
\end{proof}

\subsection{A priori estimates on the full equations}\label{sectEE}

We recall that the energy $\cE^N(\zeta,\psi,\omega)$ is defined in
\eqref{defenergy},
$$
\cE^N(\zeta,\psi,\omega):=\frac{1}{2}\abs{\zeta}_{H^N}^2+\frac{1}{2}\abs{\proj\psi}_{H^3}^2+\frac{1}{2}\sum_{0<\abs{\alpha}\leq
N}\abs{\proj\psi_{(\alpha)}}_2^2+\frac{1}{2}\Abs{\omega}_{H^{N-1}}^2+\frac{1}{2}\abs{\omega_b\cdot
N_b}^2.
$$
The following proposition gives an a priori estimate for this energy
provided that the {\it Rayleigh-Taylor} criterion is uniformly
satisfied (i.e. $\mfa$ remains strictly positive on the time interval
considered) and that the water depth does not vanish.
\begin{proposition}\label{prop3000}
Let $N\geq 5$, $T>0$ and
$(\zeta,\psi,\omega)\in E^N_T$ be a solution to the water waves
equations \eqref{ZCSgen} on the time interval $[0,T]$. If
\eqref{hmin}$_T$ is satisfied and if
$$
\exists a_0>0, \quad\forall t\in [0,T],\quad \mfa(t)\geq a_0,
$$
with $\mfa$ as defined in Proposition \ref{prop3}, then
$$
\forall 0\leq t\leq T,\qquad \cE^N(\zeta,\psi,\omega)(t)\leq
C(T,\frac{1}{a_0},\frac{1}{h_{\rm min}},H_0,\cE^N(\zeta^0,\psi^0,\omega^0)).
$$
\end{proposition}
\begin{proof}
Energy estimates are classically obtained by controlling the time
derivative of this expression. This has already partially been done in
Proposition \ref{EEvorticity} for the vorticity. We deal here with the
other components of the energy.\\
{\bf Step 1}. {Control of the low order terms in
  $(\zeta,\psi)$}. Computing the time derivative of these terms, one gets
\begin{eqnarray*}
\frac{1}{2}\frac{d}{dt}\big(\abs{\zeta}_2^2+\abs{\proj\psi}_{H^3}^2\big)&=&(\zeta,\dt
\zeta)+(\Lambda^3\proj\psi,\Lambda^3\proj \dt\psi)\\
&\leq& \abs{\zeta}_2\abs{\dt \zeta}_2+\abs{\proj\psi}_{H^3}\abs{\proj\dt\psi}_{H^3}.
\end{eqnarray*}
Replacing $\dt\zeta$ by $\uU\cdot N$ according to the first equation
of \eqref{ZCSgen} and replacing similarly $\dt \psi$ using the second
equation of \eqref{ZCSgen}, we get by standard product estimates and
the trace lemma that
\begin{eqnarray*}
\abs{\dt \zeta}_2+\abs{\proj\dt\psi}_{H^3}&\leq&
C(\abs{\zeta}_{H^N},\abs{\uU}_{H^{7/2}},\abs{\uom}_{H^{5/2}})\\
&\leq& C(\abs{\zeta}_{H^N},\Abs{\uU}_{H^{4,1}},\Abs{\omega}_{H^{3}})
\end{eqnarray*}
We then easily deduce with the help of Proposition
\ref{horizontalregurality} that
\begin{equation}\label{control0}
\frac{1}{2}\frac{d}{dt}\big(\abs{\zeta}_2^2+\abs{\proj\psi}_{H^3}^2\big)\leq \mfN(\zeta,\psi,\omega).
\end{equation}
{\bf Step 2}. {Control of the higher order terms in $(\zeta,U)$}. We
do not directly give a control of the components of the energy
that involve $\partial^\alpha\zeta$ and $\psi_{(\alpha)}$, for all
$\alpha\neq 0$. The control of $\psi_{(\alpha)}$ will indeed be obtained
indirectly (through Corollary \ref{tracecoer}) from the control of $\partial_kU_{(\beta)}$ (in the flat strip)
derived here.
We do not for this
purpose directly use \eqref{ZCSgen} as above, but rather the system
exhibited in Proposition \ref{prop3}, namely,
\begin{equation}\label{eq1alpha}
\left\lbrace\begin{array}{l}
\dsp (\dt +\uV\cdot\nabla)
\partial^\alpha\zeta-\partial_k\uU_{(\beta)}\cdot N=R^1_\alpha,\\
\dsp (\dt +\uV\cdot \nabla) (U_{(\beta)\parallel}\cdot {\bf e}_k)+\mfa \partial^\alpha \zeta=R^2_\alpha,
\end{array}\right.
\end{equation}
where $\partial^\alpha=\partial_k\partial_\beta$,
with $k=x,y$ and $0\leq \abs{\beta}\leq N-1$.
The
nondiagonal terms in \eqref{eq1alpha} can be cancelled out in the
energy estimates by multiplying the first equation by
$\mfa \partial^\alpha\zeta$, and the second one by $\partial_k\uU_{(\beta)}\cdot N$. More precisely,
multiplying the first equation of \eqref{eq1alpha} by
$\mfa\partial^\alpha\zeta$ and integrating over $\R^d$, we get
\begin{eqnarray*}
\nonumber
\frac{1}{2}\dt
(\mfa \partial^\alpha\zeta,\partial^\alpha\zeta)-(\partial_k\uU_{(\beta)}\cdot
N,\mfa \partial^\alpha\zeta)&=&(R_\alpha^1,\mfa \partial^\alpha
\zeta)-\frac{1}{2}(\dt \mfa \partial^\alpha\zeta,\partial^\alpha
\zeta)\\
& &-(\uV\cdot
\nabla \partial^\alpha\zeta,\mfa\partial^\alpha\zeta)\\
&\leq& \mfN(\zeta,\psi,\omega),
\end{eqnarray*}
the last line being a consequence of Proposition \ref{prop3} and Corrollay \ref{corotime} (and a
simple integration by parts for the last term of the right-hand-side).
Multiplying the second equation of \eqref{eq1alpha} by $\partial_k\uU_{(\beta)}\cdot N$
and integrating over $\R^d$, we also get
\begin{align*}
\nonumber
 \big(\dt (\Vpb \cdot {\bf e}_k),\partial_k\uU_{(\beta)}\cdot N\big)
 +\big(\mfa \partial^\alpha \zeta,\partial_k\uU_{(\beta)}\cdot
 N\big) &+\big(\uV\cdot
 \nabla (\Vpb\cdot {\bf e}_k),\partial_k\uU_{(\beta)}\cdot
 N\big)\\
&=\big(R^2_\alpha,\partial_k \uU_{(\beta)}\cdot N)\\
&\leq \mfN(\zeta,\psi,\omega),
\end{align*}
the last line stemming from Proposition \ref{prop3} and Corollary \ref{propalmostinc}.
Summing up these two identities, we get
\begin{equation}\label{esten3}
\frac{1}{2}\dt
(\mfa \partial^\alpha\zeta,\partial^\alpha\zeta)+
\big((\dt+\uV\cdot\nabla) (\Vpb\cdot {\bf e}_k),\partial_k\uU_{(\beta)}\cdot N\big)\leq \mfN(\zeta,\psi,\omega).
\end{equation}
Let us denote by $\Vpb^\flat$ the following
  extension of $\Vpb$,
$$
\Vpb^\flat=V_{(\beta)}+w_{(\beta)}\nabla\sigma;
$$
focusing our attention on the second term of the
left-hand-side of \eqref{esten3}, and remarking that the kinematic
equation implies
\begin{equation}\label{lift}
(\dt+\uV\cdot\nabla) (\Vpb\cdot {\bf
  e}_k)={(\dt^\sigma+U\cdot\nabla_{X,z}^\sigma) (\Vpb^\flat\cdot {\bf
  e}_k)}_\surff,
\end{equation}
we can deduce from Green's identity that
\begin{align*}
\big((\dt+\uV\cdot\nabla) &(\Vpb\cdot {\bf
  e}_k),\partial_k\uU_{(\beta)}\cdot N\big)\\
&=\int_\cS (1+\dz\sigma)(\dt^\sigma+U\cdot\nabla_{X,z}^\sigma) (\Vpb^\flat\cdot {\bf
  e}_k)\nabla_{X,z}^\sigma\cdot (\partial_k U_{(\beta)})\\
&+ \int_\cS (1+\dz\sigma)\nabla_{X,z}^\sigma\big[(\dt^\sigma+U\cdot\nabla_{X,z}^\sigma) (\Vpb^\flat\cdot {\bf
  e}_k)\big]\cdot \partial_k U_{(\beta)}.
\end{align*}
The first term of the right-hand-side is controlled by
$\mfN(\zeta,\psi,\omega)$ by a simple application of Cauchy-Schwartz
inequality and Propositions \ref{propshapeU} and
\ref{propalmostinc}. We can also remark that it is possible to commute
the operators $\nabla_{X,z}^\sigma$ and $(\dt^\sigma+U\cdot\nabla_{X,z}^\sigma) (\Vpb^\flat\cdot {\bf
  e}_k)$ in the second term of the right-hand-side, up to terms that
can similarly be controlled by $\mfN(\zeta,U,\omega)$. We therefore have
\begin{align}
\nonumber
\big((\dt+\uV\cdot\nabla) &(\Vpb\cdot {\bf
  e}_k),\partial_k\uU_{(\beta)}\cdot N\big)\\
\label{BTZ}
&=\int_\cS (1+\dz\sigma)\big[(\dt^\sigma+U\cdot\nabla_{X,z}^\sigma) \nabla_{X,z}^\sigma (\Vpb^\flat\cdot {\bf
  e}_k)\big]\cdot \partial_k U_{(\beta)}+\mbox{l.o.t.}
\end{align}
where l.o.t. stands for lower order terms that can be controlled by
$\mfN(\zeta,\psi,\omega)$.
We now need the following lemma to relate the quantity
$\nabla_{X,z}^\sigma (\Vpb^\flat\cdot{\bf e}_k)$ to $\partial_k U_{(\beta)}$.
\begin{lemma}\label{lemma2}
Denoting ${\bf t}_k={\bf e}_k+\partial_k \sigma {\bf e}_z $, one has
$$
\nabla_{X,z}^\sigma (\Vpb^\flat\cdot{\bf e}_k)=\partial_kU_{(\beta)}+{\bf t}_k\times \partial^\beta\omega+r_\alpha,
$$
with
$$
r_\alpha=(\big(\nabla_{X,z}^\sigma \times
U_{(\beta)}\big)-\partial^\beta \omega)+w_{(\beta)}\nabla_{X,z}^\sigma \partial_k\sigma;
$$
in particular,
$$
\Abs{(\dt^\sigma+U\cdot\nabla_{X,z}^\sigma)r_\alpha}_2\leq \mfN(\zeta,\psi,\omega).
$$
\end{lemma}
\begin{proof}
Remarking that
$$
\nabla_{X,z}^\sigma (U_{(\beta)}\cdot {\bf e_j})=\partial_j^\sigma
U_{(\beta)}+{\bf e}_j\times \big(\nabla_{X,z}^\sigma \times U_{(\beta)}\big)\qquad (j=x,y,z),
$$
one  computes
\begin{eqnarray*}
\nabla_{X,z}^\sigma
(\Vpb^\flat\cdot{\bf e}_k)&=&\partial^{\sigma}_kU_{(\beta)}+{\bf e}_k\times \big(\nabla_{X,z}^\sigma \times U_{(\beta)}\big)+\dz^\sigma U_{(\beta)}\partial_k
\sigma\\
& &+{\bf
  e}_z\times \big(\nabla_{X,z}^\sigma \times U_{(\beta)}\big)\partial_k\sigma
+w_{(\beta)}\nabla_{X,z}^\sigma \partial_k\sigma\\
&=& \partial_k U_{(\beta)}+{\bf
  t}_\kappa\times
 \big(\nabla_{X,z}^\sigma \times U_{(\beta)}\big)+w_{(\beta)}\nabla_{X,z}^\sigma \partial_k\sigma,
\end{eqnarray*}
and the result follows from Proposition \ref{propalmostinc} and
Corollary \ref{corotime}.
\end{proof}
A direct consequence of \eqref{BTZ} and the Lemma is that
\begin{align*}
\big((\dt+\uV\cdot\nabla) &(\Vpb\cdot {\bf
  e}_k),\partial_k\uU_{(\beta)}\cdot N\big)\\
&=\int_\cS (1+\dz\sigma)\big[(\dt^\sigma+U\cdot\nabla_{X,z}^\sigma)
\partial_k U_{(\beta)}\big]\cdot \partial_k U_{(\beta)}\\
&+\int_\cS (1+\dz\sigma)\big[{\bf t}_k\times (\dt^\sigma+U\cdot\nabla_{X,z}^\sigma)
\partial^\beta\omega\big]\cdot \partial_k U_{(\beta)}+\mbox{l.o.t.}\\
&=A+B
\end{align*}
(with the same meaning as above for l.o.t.). We now turn to analyze
$A$ and $B$.
\begin{itemize}
\item Analysis of $A$. Using the integration by
part formula \eqref{IPP} together with the identity
$$
\dt \int_\cS (1+\dz\sigma) fg=\int_\cS (1+\dz\sigma)\dt^\sigma f\, g
+\int_{\cS} (1+\dz\sigma) f \dt^\sigma g +\int_{z=0} fg \dt \zeta,
$$
we can write
\begin{eqnarray}
\nonumber
A&=&\frac{1}{2}\dt \int_\cS (1+\dz\sigma) \abs{\partial_k
  U_{(\beta)}}^2-\frac{1}{2}\int_\cS (\nabla_{X,z}^\sigma\cdot
U)\abs{\partial_k U_{(\beta)}}_2^2\\
\nonumber
& &-\frac{1}{2}\int_{z=0}\abs{\partial_k\uU_{(\beta)}}^2(\dt\zeta-\uU\cdot
N)\\
\label{Baiona1}
&=&\frac{1}{2}\dt \int_\cS (1+\dz\sigma) \abs{\partial_k
  U_{(\beta)}}^2,
\end{eqnarray}
where we used that fact that $\nabla_{X,z}^\sigma\cdot U=0$ and
$\dt\zeta-\uU\cdot N=0$ by the first equation of \eqref{ZCSgen}.
\item Analysis of $B$. By Cauchy-Schwarz inequality, we have
\begin{eqnarray}
\nonumber
B&\leq& C(\abs{\zeta}_{W^{1,\infty}})\Abs{(\dt^\sigma+U\cdot
  \nabla_{X,z}^\sigma)\partial^\beta \omega}_2\Abs{\partial_k
  U_{(\beta)}}_2\\
\label{Baiona2}
&\leq & \mfN(\zeta,\psi,\omega),
\end{eqnarray}
the last line stemming from the third assertion of Proposition \ref{prop3}
and Proposition \ref{horizontalregurality}.
\end{itemize}
We deduce from this analysis that
$$
\big((\dt+\uV\cdot\nabla) (\Vpb\cdot {\bf
  e}_k),\partial_k\uU_{(\beta)}\cdot N\big)
=\frac{1}{2}\dt \int_\cS(1+\dz\sigma)\abs{\partial_k U_{(\beta)}}_2^2+\mbox{l.o.t.}
$$
so that \eqref{esten3} yields
\begin{equation}\label{EEmain}
\dt\left\lbrace
(\mfa \partial^\alpha\zeta,\partial^\alpha\zeta)+ \int_\cS(1+\dz\sigma)\abs{\partial_k U_{(\beta)}}_2^2\right\rbrace
\leq \mfN(\zeta,\psi,\omega).
\end{equation}
{\bf Step 4}. {Energy estimate on the modified energy.} Let us define
the modified energy $\tilde\cE^N=\tilde\cE^N(\zeta,\psi,\omega)$ as
\begin{eqnarray*}
\tilde\cE^N&=&\abs{\zeta}_{2}^2+\abs{\proj\psi}_{H^3}^2\\
& +&\sum_{k=x,y,0<\abs{\beta}\leq
  N-1}\big[(\partial_k\partial^\beta\zeta,\mfa\partial_k\partial^\beta
\zeta)+\int_\cS (1+\dz\sigma)\abs{\partial_k U_{(\beta)}}_2^2 \big]
+\Abs{\omega}_{H^{N-1}}^2.
\end{eqnarray*}
It follows directly from Steps 1, 2 and 3, and the vorticity estimates of
Proposition \ref{EEvorticity} that
$$
\frac{d}{dt}\tilde\cE^N(\zeta,\psi,\omega)\leq \mfN(\zeta,\psi,\omega).
$$
Under the assumption made on $\mfa$, one can write, for all
$\alpha\neq 0$,
$$
\abs{\partial^\alpha\zeta}_2^2\leq \frac{2}{a_0}(\mfa \partial^\alpha\zeta,\partial^\alpha\zeta)
$$
so that, with the help of Corollary \ref{tracecoer} we have
$$
\mfN(\zeta,\psi,\omega)\leq C\big(\frac{1}{a_0},\tilde\cE^N(\zeta,\psi,\omega)\big)
$$
and therefore
\begin{equation}\label{pouf}
\frac{d}{dt}\tilde\cE^N(\zeta,\psi,\omega)\leq C\big(\frac{1}{a_0},\tilde\cE^N(\zeta,\psi,\omega)\big).
\end{equation}
We classically deduce from this differential inequality that for all
$0\leq t \leq T$,
\begin{equation}\label{EEmod}
\tilde\cE^N(\zeta,\psi,\omega)(t)\leq
C(T,\frac{1}{a_0},\frac{1}{h_{\rm min}},H_0,\tilde\cE^N(\zeta^0,\psi^0,\omega^0)).
\end{equation}
{\bf Step 5.} {Conclusion.} Corollary \ref{tracecoer} implies that
$$
\cE^N(\zeta,\psi,\omega)\leq C(\frac{1}{a_0},\frac{1}{h_{\rm min}}, \tilde\cE^N(\zeta,\psi,\omega))
$$
while we get from Proposition \ref{horizontalregurality} that
$$
\tilde\cE^N(\zeta^0,\psi^0,\omega^0)\leq C(\frac{1}{h_{\rm min}},\cE^N(\zeta^0,\psi^0,\omega^0)).
$$
These two inequalities, together with \eqref{EEmod} imply that for all
$0\leq t\leq T$, one has
$$
\cE^N(\zeta,\psi,\omega)(t)\leq C(T,\frac{1}{a_0},\frac{1}{h_{\rm min}},H_0,\cE^N(\zeta^0,\psi^0,\omega^0));
$$
this is exactly the result stated in the proposition.
\end{proof}

\subsection{Main result}

As explained in Remark \ref{straightvort}, it is easier to give a functional
framework for our well-posedness result if we replace the equation on
the vorticity $\bom$
by an equation on the straightened vorticity $\omega=\bom\circ \Sigma$ in the water
waves equations \eqref{ZCSgen}, where we recall that
\begin{equation}\label{diffS}
\Sigma(t,X,z)=(X,z+\sigma(t,X,z))
\quad\mbox{ and }\quad
\sigma(t,X,z)=\frac{1}{H_0}(z+H_0)\zeta(t,X).
\end{equation}
Recalling that according to  Notation \ref{notastr}, we have ${\mathbb
  U}^\sigma[\zeta](\psi,\omega)={\mathbb
  U}[\zeta](\psi,\omega)\circ\Sigma$, with horizontal and vertical
components ${\mathbb V}^\sigma[\zeta](\psi,\omega)$ and ${\mathbb w}^\sigma[\zeta](\psi,\omega)$, we also denote
\begin{align*}
{\mathbb U}_\parallel^\sigma[\zeta](\psi,\omega):=&\underline{\mathbb V}^\sigma[\zeta](\psi,\omega)+\underline{\mathbb w}^\sigma[\zeta](\psi,\omega)\nabla\zeta    ={\mathbb U}_\parallel[\zeta](\psi,\bom)
\end{align*}
(as usual, quantities evaluated at the surface $z=0$ are underlined).
Owing to \eqref{eq10bis}, instead of the set of evolution equations \eqref{ZCSgen} on
$(\zeta,\psi,\bom)$, we are therefore concerned with the following set of
evolution equations on $(\zeta,\psi,\omega)$,
\begin{equation}\label{ZCSgens}
\left\lbrace
\begin{array}{l}
\dsp \dt \zeta-\underline{\mathbb U}^\sigma[\zeta](\psi,\omega)\cdot N=0,\\
\dsp \dt \psi +g \zeta+\frac{1}{2}
 \babs{{\mathbb
     U}_\parallel^\sigma[\zeta](\psi,\omega)}^2-\frac{1}{2}(1+\abs{\nabla\zeta}^2)\underline{{\mathbb
   w}}^\sigma[\zeta](\psi,\omega)^2 \\
\hspace{4.2cm}\dsp -\frac{\nabla^\perp}{\Delta}\cdot\big(
\uom\cdot N \underline{\mathbb
  V}^\sigma[\zeta](\psi,\omega)\big)=0,\\
\dsp \dt^\sigma \omega+{\mathbb U}^\sigma[\zeta](\psi,\omega)\cdot\nabla_{X,z}^\sigma
\omega=\omega\cdot\nabla_{X,z}^\sigma {\mathbb U}^\sigma[\zeta](\psi,\omega).
\end{array}\right.
\end{equation}
Note that as for \eqref{ZCSgen}-\eqref{DFvort}, these equations make
sense if $\omega$ is divergence free in the sense that
\begin{equation}\label{DFvort2}
\dives\omega=0\quad \mbox{ in }\quad {\mathcal S},
\end{equation}
but that we omit this constraint since it is propagated by the
equation on $\omega$ if is initially satisfied.\\
We also recall the definition \eqref{defenergy} of the energy
$$
\cE^N(\zeta,\psi,\omega):=\frac{1}{2}\abs{\zeta}_{H^N}^2+\frac{1}{2}\abs{\proj\psi}_{H^3}^2+\frac{1}{2}\sum_{0<\abs{\alpha}\leq
N}\abs{\proj\psi_{(\alpha)}}_2^2+\frac{1}{2}\Abs{\omega}_{H^{N-1}}^2+\frac{1}{2}\abs{\omega_b\cdot
N_b}_{H_0^{-1/2}}^2
$$
and of the associated functional space $E_T^N$ defined in \eqref{defETN}
for all $T\geq 0$ as
\begin{align*}
\nonumber
E_T^N=\{(\zeta,\psi,\omega)\in C\big([0,T];H^2(\R^d)&\times \dot{H}^2(\R^d)\times
H^2(\cS)\big), \\
& \cE^N((\zeta,\psi,\omega)(\cdot))\in L^\infty([0,T])\};
\end{align*}
we also denote by $E_0^N$ the set of $(\zeta,\psi,\omega)\in
H^2(\R^d)\times \dot{H}^2(\R^d)\times H^2(\cS)^3$ of finite energy. We
recall finally that the Raighley-Taylor coefficient ${\mathfrak a}$ is
defined\footnote{The notation ${\mathfrak a}={\mathfrak
    a}(\zeta,\psi,\omega)$ suggests that ${\mathfrak a}$ is a function
of $(\zeta,\psi,\omega)$ and not of their time derivatives. We
actually use an alternative definition of ${\mathfrak a}$ where $\dt \big(\underline{\mathbb
  w}^\sigma[\zeta](\psi,\omega)\big)$ is written in terms of $\dt \psi$, $\dt\zeta$
and $\dt^\sigma\omega$ through Proposition \ref{propshapeU}, and where
these time derivatives are replaced by purely spatial operators using
the time evolution equations
provided by \eqref{ZCSgens}.} as
$$
{\mathfrak a}={\mathfrak a}(\zeta,\psi,\omega)=g+(\dt+\underline{\mathbb
  V}^\sigma[\zeta](\psi,\omega))\underline{\mathbb
  w}^\sigma[\zeta](\psi,\omega).
$$
\begin{thm}\label{theo1}
Let $N\geq 5$ and $\Theta^0=(\zeta^0,\psi^0,\omega^0)\in E_0^N$ be such that
$\omega^0$ satisfies the divergence free condition
\eqref{DFvort2}. Assume moreover that
$$
\exists h_{\rm min}>0,\exists a_0>0,\quad H_0+\zeta^0>h_{\rm
  min},\quad {\mathfrak a}(\zeta^0,\psi^0,\omega^0)>a_0.
$$
Then there exists $T>0$ and a unique solution $\Theta\in E^N_T$ to
\eqref{ZCSgens} satisfying the divergence free constraint
\eqref{DFvort2}, and with initial condition $\Theta^0$. Moreover,
$$
\frac{1}{T}=c^1\quad\mbox{ and }\quad \sup_{t\in
  [0,T]}\cE^N(\Theta(t))=c^2
$$
with $\dsp c^j=C(\cE^N(\Theta^0),\frac{1}{h_{\rm min}},H_0,\frac{1}{a_0})$ for $j=1,2$.
\end{thm}
\begin{proof}
The strategy of the proof is the following. In Step 1, we give an
equivalent formulation of the water waves equations \eqref{ZCSgens2}; in view of
regularizing these equation, we define and study in Step 2 mollifiers
in the horizontal and vertical variables. Since the divergence free
constraint \eqref{DFvort2} is a consequence of the particular
structure of the vorticity equation in \eqref{ZCSgens2}, it may not
hold with regularized equations; this would be an obstruction to solving the div-curl problem studied in
\S \ref{sectdivcurl1}, which requires that the vorticity be divergence
free. We therefore
explain in Step 3 how to project any vector field in $H^1(\cS)^{d+1}$ on
its divergence free component. The regularized equations are then
defined in Step 4, and solved with standard ODEs tools. The solution
thus obtained is studied in Step 5, where we prove {\it a posteriori }that its vorticity
is  divergence free. Energy estimates inspired by the a priori
estimates of Proposition \ref{prop3000} are then established in Step 6
and used in Step 7 to prove that the solutions to the regularized
equations converge to a solution of \eqref{ZCSgens2}. \\
{\it For the sake of simplicity, we often write $U={\mathbb
    U}^\sigma[\zeta](\psi,\omega)$, $a={\mathbb
    a}[\zeta](\psi,\omega)$, etc. when no confusion is possible.}\\
{\bf Step 1.} Modification of the vorticity equation. As explained in
the proof of Proposition \ref{EEvorticity}, it is equivalent to solve
the equations \eqref{ZCSgens} with the vorticity equation replaced by
$$
 \dt\omega+{\mathbb V}^{\sigma}[\zeta](\psi,\omega)\cdot\nabla
 \omega+{\mathbb a}[\zeta](\psi,\omega)\dz\omega=\omega\cdot\nabla_{X,z}^{\sigma} {\mathbb U}^{\sigma}[\zeta](\psi,\omega)
$$
with, denoting $N(z)=(-\nabla{\sigma}^T,1)^T$ (so that $N(0)=N$),
$$
{\mathbb a}[\zeta](\psi,\omega)=\frac{1}{1+\dz\sigma}\Big( {\mathbb
   U}^{\sigma}[\zeta](\psi,\omega)\cdot N(z)-\frac{z+H_0}{H_0}
 {\mathbb
   U}^\sigma[\zeta](\psi,\omega)_\surff\cdot N\Big).
$$
The equations \eqref{ZCSgens} are therefore equivalent to
\begin{equation}\label{ZCSgens2}
\dt \Theta+ {\mathcal F}(\Theta)=0,
\end{equation}
with
$$
\Theta= \left(\begin{array}{c} \zeta\\ \psi \\
    \omega \end{array}\right),\qquad {\mathcal F}(\Theta)\left(\begin{array}{c} {\mathcal F}_1(\Theta)\\ {\mathcal F}_2(\Theta)\\ {\mathcal F}_3(\Theta) \end{array}\right)
$$
and
\begin{eqnarray*}
{\mathcal
  F}_1(\Theta)&=&-\underline{\mathbb
  U}^\sigma[\zeta](\psi,\omega)\cdot N,\\
{\mathcal
  F}_2(\Theta)&=&g  \zeta+\frac{1}{2}
 \babs{{\mathbb U}^\sigma_\parallel[\zeta](\psi,\omega)}^2-\frac{1}{2}(1+\abs{\nabla\zeta}^2)\underline{{\mathbb
   w}}^\sigma[\zeta](\psi,\omega)^2,\\
{\mathcal
  F}_3(\Theta)&=&{\mathbb
  V}^{\sigma}[\zeta](\psi,\omega)\cdot\nabla
\omega+{\mathbb
  a}[\zeta](\psi,\omega)\dz\omega-\omega\cdot \nabla_{X,z}^{\sigma} {\mathbb U}^{\sigma}[\zeta](\psi,\omega).
\end{eqnarray*}

\noindent
{\bf Step 2.} Definition of the mollifiers. For the horizontal
variables, we use a standard mollifier. We define, for all
$0<\iota<1$, the mollifier $J^\iota=\chi(\iota \abs{D})$, where $\chi:
\R\to\R$ is a smooth, even, and compactly supported function equal to
$1$ in a neighborhood of the origin. The mollifying properties of
$J^\iota$ are classical and straightforwardly deduced from standard
results on Fourier multipliers. We shall in particular use the fact
that
\begin{equation}\label{propJ}
\forall s,t\in \R,\quad \exists C^\iota_{s,t}>0,\quad\forall f\in H^t(\R^d)\qquad \abs{J^\iota
  f}_{H^s}\leq C^\iota_{s,t}\abs{f}_{H^t}
\end{equation}
and
\begin{equation}\label{propJ2}
\forall \iota^1,\iota^2>0,\, \forall s\in \R,\, \exists
C>0,\,\forall f\in H^{s+1}(\R^d),\quad
\abs{(J^{\iota^1}\!\!-J^{\iota^2}\!)f}_{H^s}\leq C \abs{\iota^1\!-\iota^2}\abs{f}_{H^{s+1}}.
\end{equation}
For the vertical variable for which Fourier transform cannot be used,
we use another kind or regularization based on the following lemma; in
the statement, we use the following functional spaces, where $a$ is
some smooth enough function defined on $\cS$,
\begin{align*}
H(a\dz,\cS):=&\{f\in L^2(\cS); a\dz f\in L^2(\cS)\},\\
H^k(a\dz,\cS):=&\{f\in H^k(\cS),\forall \beta\in \N^d,\forall j\in \N,
\abs{\beta}+j\leq k, a\dz \partial^\beta\dz^j f\in L^2(\cS)\}\\
V_a(\cS):=&\{f\in H^{-1}(\cS), \exists \tilde f\in L^2(\R^d), f=a\dz \tilde f \}.
\end{align*}
\begin{lemma}\label{lemmaK}
Let $N\geq 5$. Let $a\in W^{1,\infty}(\cS)$ be such that $a_\surff=a_\bottf=0$ and $0<\iota<1$. \\
{\bf i.} The
mapping
$$
(1-\iota^2\dz(a^2\dz\cdot)): \begin{array}{lcl}
H(a\dz,\cS)&\to& H^{-1}(\cS)\\
f &\mapsto &f-\iota^2\dz(a^2\dz f)
\end{array}
$$
is well defined and one-to-one.\\
{\bf ii.} One also has $L^2(\cS)+V_a(\cS)\subset
\mbox{\textnormal{Range }}(1-\iota^2\dz(a^2\dz\cdot))$.\\
{\bf iii.} Let $a\in H^{N-1}(\cS)$, and $0\leq k\leq N-1$. Then, for
$\iota>0$ small enough, the mapping
 $$
K^\iota[a\dz]:
\begin{array}{lcl}
H^k(\cS)&\to &H^k(a\dz,\cS)\\
f&\mapsto& (1-\iota^2\dz(a^2\dz\cdot))^{-1}(a\dz f)
\end{array}
$$
is well defined and continuous; if $k\leq N-2$, the result remains true if
one assumes only that $a\in H^{N-2}(\cS)$.
\\
{\bf iv.} Under the same assumptions, one also has
$$
\forall \beta\in \N^d, \forall j\in \N, \abs{\beta}+j \leq k, \qquad
(\partial^\beta \dz^j K^\iota[a\dz]f,\partial^\beta \dz^jf)\leq C(\Abs{a}_{H^{N-1}})\Abs{f}_{H^k};
$$
if $k\leq N-2$, one can replace $\Abs{a}_{H^{N-1}}$ by
$\Abs{a}_{H^{N-2}}$ in the right-hand-side.\\
{\bf v.} For all $0<\iota_1\leq \iota_2<1$ and all $f\in H^{N-1}(\cS)$, one has
$$
\Abs{K^{\iota_1}[a\dz]f-K^{\iota_2}[a\dz] f}_{H^{N-3}}\leq \abs{\iota_1-\iota_2}C(\Abs{a}_{H^{N-1}})\Abs{f}_{H^{N-1}}.
$$
\end{lemma}
\begin{proof}[Proof of the lemma]
The fact that $(1-\iota^2\dz(a^2\dz\cdot))$ is one-to-one follows immediately by taking
$f=g$ in the following integration by parts formula (recall that $a$
is assumed to cancel at the boundaries)
\begin{align}
\nonumber
\forall f,g\in H(a\dz,S),\qquad \big((1-\iota^2\dz(a^2\dz\cdot)) f,g\big)&=(f,g)+\iota^2(a\dz
f,a\dz g)\\
\label{IPPreg}
&:=B(f,g).
\end{align}
In order to prove the second point of the lemma, we need to prove that
for all $f_1,f_2\in L^2(\cS)$, there exists $u\in H(a\dz,\cS)$ such
that
$$
(1-\iota^2\dz(a^2\dz\cdot)) u=f_1+\dz( a f_2).
$$
We prove the existence of a variational solution, i.e. of $u\in
H(a\dz,\cS)$ such that
$$
\forall g\in H(a\dz;\cS),\qquad B(u,g)=(f_1+\dz(a f_2), g);
$$
since the bilinear form is obviously continuous and coercive on
$H(a\dz,\cS)$, the existence of a variational solution is granted by
Lax-Milgram's theorem if we can prove that $g\mapsto (f_1+\dz(a f_2),
g)$ defines a continuous linear form on $H(a\dz,\cS)$. Since $a$
vanishes at the boundaries, one has
\begin{align*}
(f_1+\dz (a f_2),g)=&(f_1,g)-(f_2,a\dz g)\\
\leq&\Abs{f_1}_2\Abs{g}_2+\Abs{f_2}_2\Abs{a\dz g}_2,
\end{align*}
which implies the desired continuity property. Note moreover for later
use that, owing to \eqref{IPPreg}, the solution $u=(1-\iota^2\dz(a^2\dz\cdot))^{-1}(f_1+\dz(a
f_2))$ satisfies the bound
\begin{equation}\label{eqregreg}
\Abs{u}_2+\iota \Abs{a\dz u}_2\leq 2 \Abs{f_1}_2+\frac{2}{\iota}\Abs{f_2}_2.
\end{equation}
From the second point, one can define $K^\iota[a\dz]: L^2(\cS)\to
H(a\dz,\cS)$. For $f\in L^2(\cS)$ let $u=K^\iota[a\dz] f$. We
need to prove that if $f\in H^k(\cS)$, then $u\in H^k(a\dz,\cS)$.
Applying $\partial^\beta\dz^j$, with $\abs{\beta}+j\leq k$ to the
relation
$$
(1-\iota^2\dz(a^2\dz\cdot)) u =a\dz f=-(\dz a)f +\dz(a f),
$$
 we obtain
\begin{equation}\label{ihold}
(1-\iota^2\dz(a^2\dz\cdot)) \partial^\beta\dz^j u=f_1+\dz(a f_2)
\end{equation}
where $f_1$ and $f_2$ are of the form
\begin{eqnarray*}
f_1&=&\dz [\partial^\beta\dz^j,a]f
-\partial^\beta\dz^j\big((\dz a)f\big)+\\
& &+\iota^2\dz\Big[\sum_{\abs{\beta'}+j'\geq 1,\abs{\beta''}+j''\geq 1}*(\partial^{\beta'}\dz^{j'}a) (\partial^{\beta''}\dz^{j''}a)\partial^{\beta-\beta'-\beta''}\dz^{j-j'-j''}\dz u\Big]\\
f_2&=&
\partial^\beta\dz^j
f+\iota^2\sum_{\abs{\beta'}+j'\geq 1}*(\partial^{\beta'}\dz^{j'}a)\partial^{\beta-\beta'}\dz^{j-j'}\dz u,
\end{eqnarray*}
where we denoted by $*$ numerical coefficients of no importance
here. From \eqref{eqregreg} and \eqref{ihold}, and using the product
estimates as in the proof of Lemma \ref{lemHkk}, we have
\begin{eqnarray*}
\Abs{\partial^\beta\dz^ju}_2+\iota
\Abs{a\dz \partial^\beta\dz^ju}_2&\lesssim&
\Abs{a}_{H^{N-1}}\Abs{f}_{H^k}+\iota^2\Abs{a}_{H^{N-1}}^2\Abs{u}_{H^{k}}\\
& &+\frac{1}{\iota}\big(\Abs{f}_{H^k}+\iota^2\Abs{a}_{H^{N-1}}\Abs{u}_{H^k}\big).
\end{eqnarray*}
Summing over all $\abs{\beta}+j\leq k$ and taking $\iota$ small enough
to absorb the terms in $\Abs{u}_{H^k}$ in the right-hand-side by the
left-hand-side of the inequality, we get
\begin{equation}\label{star}
\Abs{u}_{H^k}+\iota \sum_{\abs{\beta}+j\leq k}
\Abs{a\dz\partial^\beta\dz^j u}_2\leq C(\Abs{a}_{H^{N-1}},\frac{1}{\iota})\Abs{f}_{H^k},
\end{equation}
which proves the third point of the lemma.\\
For the fourth point, let us first prove the case $k=0$. With $\tilde
f=(1-\iota^2\dz(a^2\dz \cdot))^{-1}f$, we have
\begin{align*}
(K^\iota[a\dz]f,f)=&\big(a\dz (1-\iota^2\dz(a^2\dz\cdot)) \tilde
f,\tilde f\big)\\
=&-\frac{1}{2}\big((\dz a) \tilde f,  \tilde f\big)+\iota^2 \big( a\dz \tilde f, \dz
(a \tilde f)\big),
\end{align*}
and therefore
\begin{align*}
(K^\iota[a\dz]f,f)&\leq C(\Abs{\dz a}_\infty)(\Abs{\tilde
  f}_2^2+\iota^2\Abs{a\dz\tilde f}_2^2)\\
&\leq C(\Abs{\dz a}_\infty)\Abs{f}_2,
\end{align*}
the second inequality stemming from \eqref{eqregreg}; the result is
therefore proved when $k=0$. When $0<k\leq N-1$, we write, for all
$\abs{\beta}+j\leq k$,
$$
(\partial^\beta\dz^j K^\iota[a\dz]f,\partial^\beta\dz^j
f)=(\partial^\beta\dz^j f,K^\iota[a\dz]\partial^\beta\dz^j
f)+\big(\partial^\beta\dz^j f,  [\partial^\beta\dz^j, K^\iota[a\dz ]]f\big);
$$
we can use the case $k=0$ to control the first term of the r.h.s. The
commutator $[\partial^\beta\dz^j, K^\iota[a\dz ]]f$ can then be
controlled with the same computations as in the proof of the third
point, but without the term $\partial^\beta\dz^j f$ in $f_2$, which
was responsible for the $\iota^{-1}$ singularity in \eqref{star}. We
therefore have
$$
\Abs{[\partial^\beta\dz^j, K^\iota[a\dz ]]f}\leq C(\Abs{a}_{H^{N-1}})\Abs{f}_{H^k},
$$
and the result follows.\\
For the last point of the lemma, let us write $u_j=K^{\iota_j}[a\dz]f$
($j=1,2$). One computes that
$$
u_1-u_2=-(\iota_2^2-\iota^2_2)(1-\iota_1^2\dz(a^2\dz\cdot))^{-1}\dz(a^2\dz u_2);
$$
using \eqref{star}, we deduce that
\begin{align*}
\Abs{u_1-u_2}_{H^{N-3}}&\leq
\abs{\iota_2^2-\iota_1^2}C(\Abs{a}_{H^{N-1}})\Abs{\dz(a^2\dz u_2)}_{H^{N-3}}\\
&\leq
\abs{\iota_2^2-\iota_1^2}C(\Abs{a}_{H^{N-1}})\Abs{u_2}_{H^{N-1}}\\
&\leq
\frac{\abs{\iota_2^2-\iota_1^2}}{\iota_2}C(\Abs{a}_{H^{N-1}})\Abs{f}_{H^{N-1}}
\end{align*}
(the last line follows from the computations performed in the proof of
the third point), which implies the result stated in the lemma.
\end{proof}

\noindent
{\bf Step 3.} Relaxing the divergence-free condition on the
vorticity. The div-curl problem has been solved in \S
\ref{sectdivcurl1} assuming that
$\nabla_{X,z}\cdot \bom=0$, or equivalently $\nabla_{X,z}^\sigma\cdot
\omega=0$. Consistently with Definition \ref{divfree}, we introduce
the notations
\begin{align*}
H^k(\dives_0,\cS)=&\{\omega\in H^k(\cS),\nabla_{X,z}^\sigma\cdot \omega=0\},\\
H^k_b(\dives_0,\cS)=&\{\omega\in H^k(\dives_0,\cS),\omega_\bottf\cdot N_b\in
H_0^{-1/2}\}.
\end{align*}
If $\omega\not\in H^k_b(\dives_0,\cS)$, it is therefore
not possible to define ${\mathbb U}^\sigma[\zeta](\psi,\omega)$. The vorticity equation in
\eqref{ZCSgens} preserves the divergence free-property,  so that these equations make
sense if the vorticity field is initially diverge-free. However, the
regularization of the vorticity equation introduced in Step 3 below
does not preserve {\it a priori} the divergence free condition (we
only show {\it a posteriori} that it does so) and we therefore have
to define the projection
onto divergence-free vector fields as follows, for all $1\leq k\leq N-1$,
$$
\pi[\zeta]: \begin{array}{lcl}
H_b^k(\cS)&\to &H^k(\dives_0,\cS)\\
\omega&\mapsto& \omega-\nabla_{X,z}^\sigma \phi,
\end{array}
$$
where we denote by $H_b^k(\cS)$ the set of all $\omega\in H^k(\cS)$
such that $\omega_b\cdot N_b\in H_0^{-1/2}(\R^d)$, while $\phi$ solves the boundary value problem
$$
\left\lbrace
\begin{array}{l}
\nabla_{X,z}\cdot
P(\Sigma)\nabla_{X,z}\phi=(1+\dz\sigma)\nabla_{X,z}^\sigma\cdot \omega,\\
\phi_\surff=0,\qquad {\bf e}_z\cdot
P(\Sigma)\nabla_{X,z}\phi_\bottf=0,
\end{array}\right.
$$
and where we recall that $P(\Sigma)=(1+\dz\sigma) J_\Sigma^{-1}
(J_\Sigma^{-1})^T$ and $(1+\dz\sigma)\nabla_{X,z}^\sigma\cdot
\nabla_{X,z}^\sigma=\nabla_{X,z}\cdot P(\sigma)\nabla_{X,z}$. By
simple elliptic estimates similar to those of Corollary \ref{corHkk},
we get that
for all $0\leq j\leq N-1$,
\begin{align*}
\Abs{\nabla_{X,z}\phi}_{H^k}\leq &M_N
(\Abs{\omega}_{H^k}+\abs{\omega_b\cdot N_b}_{H_0^{-1/2}}).
\end{align*}
In particular, if $\omega\in H^1(\cS)$ and $\omega_b\cdot N_b\in
H_0^{-1/2}(\R^d)$  (note that the $H^1$ regularity is imposed on
$\omega$ so that the normal trace of $\omega$ at the bottom makes
sense) then $\pi[\zeta]\omega \in H^1_b(\cS,\dives_0,\cS)$;
when $\omega\in H^1(\cS)\backslash H^1_b(\dives_0,\cS)$, we
can therefore replace ${\mathbb U}^\sigma[\zeta](\psi,\omega)$
 by ${\mathbb U}^\sigma[\zeta](\psi,\pi[\zeta]\omega)$, which is well
defined. We also have that for all $1\leq k\leq N-1$,
\begin{equation}\label{eqPi}
\forall \omega\in H^k_b(\cS), \qquad
\Abs{\pi[\zeta]\omega}_{H^k}\leq
M_N\big(\Abs{\omega}_{H^k}+\abs{\omega_b\cdot N_b}_{H_0^{-1/2}}\big).
\end{equation}

\noindent
{\bf Step 4.} Existence of a local solution for a regularized
system. We consider the following regularization of the water waves
equations \eqref{ZCSgens2},
\begin{equation}\label{WWreg}
\dt \Theta+ {\mathcal F}^{\delta,\iota}(\Theta)=0,
\end{equation}
with $0<\delta,\iota<1$ and
\begin{align*}
{\mathcal
  F}_1^{\delta,\iota}(\Theta)=&-J^\iota \big(\underline{\mathbb
  U}^\sigma[\zeta](\psi,\pi[\zeta]\omega)\cdot N\big),\\
{\mathcal
  F}_2^{\delta,\iota}(\Theta)=&g J^\iota \zeta+\frac{1}{2} J^\iota
 \babs{{\mathbb U}^\sigma_\parallel[\zeta](\psi,\pi[\zeta]\omega)}^2
-\frac{1}{2}J^\iota\big[(1+\abs{\nabla\zeta}^2)\underline{{\mathbb
   w}}^\sigma[\zeta](\psi,\omega)^2\big]+\delta J^\iota \Lambda \zeta,\\
{\mathcal
  F}_3^{\delta,\iota}(\Theta)=&{\bf g}^\iota-\nabla_{X,z}^\sigma Q,
\end{align*}
with
\begin{align*}
{\bf g}^\iota=&
J^\iota\big({\mathbb
  V}^{\sigma}[\zeta](\psi,\pi[\zeta]\omega)\cdot\nabla
\omega\big)+K^\iota[{{\mathbb
  a}}[\zeta](\psi,\pi[\zeta]\omega)\dz]\omega\\
&-J^\iota\big(\omega\cdot \nabla_{X,z}^{\sigma} {\mathbb
  U}^{\sigma}[\zeta](\psi,\pi[\zeta]\omega)\big).
\end{align*}
We also consider regularized initial conditions,
$$
\Theta_{\vert_{t=0}}=(J^\iota \zeta^0, J^\iota \psi^0,\omega^0).
$$
The reasons why such a regularization is introduced are listed below:
\begin{itemize}
\item Projection onto the divergence free component of the
  vorticity. As already mentioned, the mapping ${\mathbb
    U}^\sigma[\zeta](\psi,\omega)$ is well defined provided that
  $\nabla_{X,z}^\sigma\cdot \omega=0$. Since such a condition is not
  necessarily propagated from the initial condition by the regularized
  vorticity equation, we have to replace ${\mathbb
    U}^\sigma[\zeta](\psi,\omega)$ by ${\mathbb
    U}^\sigma[\zeta](\psi,\pi[\zeta]\omega)$ in all its occurrences.
\item Mollifiers. Horizontal derivatives are regularized with the
  mollifier $J^\iota$. For the vertical derivative, the mollifier
  $K^\iota[{a}\dz]$ is used; this is made possible by the fact
  that ${a}$ vanishes at the surface and at the bottom.
\item Dispersive regularization with parameter $\delta$. The purpose
  of the term $\delta J^\iota \Lambda \zeta$ in ${\mathcal
    F}_2^{\delta,\iota}(\Theta)$ is to allow for the control of an extra
  $1/2$-derivative on $\zeta$, uniformly with respect to $\iota$, and
  that will be used to control commutator terms due to the above mollifiers.
\item Pressure term $Q$. If we simply regularize the equation as
  explained above, the equation does not preserve the divergence free
  condition anymore. In order to recover this property, we
  choose the pressure $Q$ so that $\dt\dives \omega=0$. This leads to\footnote{We use the fact that
  $\dt^\sigma=\dt-\frac{\dt\sigma}{1+\dz\sigma\dz}\dz$, and substitute
$$
\dt\sigma=\frac{z+H_0}{H_0}J^\iota(\uU\cdot N),
$$
which holds provided that the regularized kinematic equation is satisfied.}
\begin{equation}\label{eqpress}
\left\lbrace\begin{array}{l}
\dsp \nabla_{X,z}\cdot P(\Sigma)\nabla_{X,z}
Q=(1+\dz\sigma)\nabla_{X,z}^\sigma\cdot {\bf g}^\iota\\
\hspace{4cm}\dsp +\nabla_{X,z}^\sigma\Big(\frac{z+H_0}{H_0}J^\iota(\uU\cdot N)\Big)\cdot\dz\omega,\\
\dsp Q_\surff=0,\qquad {\bf e}_z\cdot P(\Sigma)\nabla_{X,z} Q=0
\end{array}\right.
\end{equation}
(we recall that $(1+\dz\sigma)\nabla_{X,z}^\sigma\cdot
\nabla_{X,z}^\sigma=\nabla_{X,z}\cdot P(\Sigma)\nabla_{X,z}$).
\end{itemize}
Before proceeding to construct a solution to the regularized equations
\eqref{WWreg}, we need to provide some estimates on the pressure term
$\nabla_{X,z}^\sigma Q$. We shall use the following lemma.
\begin{lemma}\label{estpressQ}
Let $k=N-2,N-1$, ${\bf g}\in H^k(\cS)$ and ${\bf h}\in
  W^{k,\infty}((-H_0,0);H^k(\R^d))$ be such that ${\bf
    h}_\bottf=0$. Then there is a unique solution $Q\in
  \dot{H}^{k+1}(\cS)$ to
$$
\left\lbrace\begin{array}{l}
\nabla_{X,z}\cdot P(\Sigma)\nabla_{X,z}
Q=(1+\dz\sigma)\nabla_{X,z}^\sigma\cdot {\bf g}+{\bf h}\cdot \dz \omega,\\
Q_\surff=0,\qquad {\bf e}_z\cdot P(\Sigma)\nabla_{X,z}^\sigma Q_\bottf=0,
\end{array}\right.
$$
and one has
$$
\Abs{\nabla_{X,z}^\sigma Q}_{H^{k}}\leq M_N\big(\Abs{{\bf
    g}}_{H^{k}}+\Abs{{\bf h}}_{W^{k,\infty}H^k}\Abs{\omega}_{H^k}\big).
$$
\end{lemma}
\begin{proof}[Proof of the lemma]
Existence of a solution follows classically from Lax-Milgram's
theorem. In order to get an $L^2$-estimate on $\nabla_{X,z}Q$ (and
therefore on $\nabla_{X,z}^\sigma Q$), one
multiplies by $Q$, integrates by parts, and use the coercivity
property \eqref{Pcoerc} of $P(\Sigma)$ to get
$$
\Abs{\nabla_{X,z}Q}_2^2\leq M_N \Abs{{\bf
    g}}_2\Abs{\nabla_{X,z}Q}_2+\int_{\R^d} g_b\cdot N_b
Q_b+\Abs{{\bf h}}_{W^{1,\infty}}\Abs{\omega}_2\Abs{\nabla_{X,z}Q},
$$
where we used for the last term the fact that ${\bf h}$ vanishes at the
bottom, and $Q$ vanishes at the surface. From the trace lemma we then
get
$$
\Abs{\nabla_{X,z}Q}_2\leq M_N (\Abs{{\bf
    g}}_{H^{0,1}}+\Abs{{\bf h}}_{W^{1,\infty}}\Abs{\omega}_2\big).
$$
In order to control horizontal derivatives of $\nabla_{X,z} Q$, just
remark that $\tilde Q=\Lambda^s Q$ solves
$$
\left\lbrace\begin{array}{l}
\dsp \nabla _{X,z}\cdot P(\Sigma)\nabla_{X,z}
\tilde Q=(1+\dz\sigma)\nabla_{X,z}^\sigma\cdot \big(\Lambda^s{\bf
  g}-J_\Sigma[\Lambda^s,P(\Sigma)]\nabla_{X,z}Q\big)\\
\hspace{3.3cm}\dsp+{\bf h}\cdot \dz (\Lambda^s\omega)+f,\\
\dsp \tilde Q_\surff=0,\qquad {\bf e}_z\cdot P(\Sigma)\nabla_{X,z}^\sigma
\tilde Q_\bottf=-{\bf e}_z\cdot [\Lambda^s,P(\Sigma)]\nabla_{X,z}Q_\bottf,
\end{array}\right.
$$
with $f=[\Lambda^s,\frac{N(z)}{1+\dz\sigma}]\dz{\bf g}+[\Lambda^s,{\bf
h}]\dz\omega$. Proceeding as above, we get
$$
\Abs{\nabla_{X,z}\tilde Q}_2\leq M_N (\Abs{{\bf
    g}}_{H^{s,1}}+\Abs{[\Lambda^s,P(\Sigma)]\nabla_{X,z}Q}_2+\Abs{{\bf h}}_{W^{1,\infty}}\Abs{\omega}_{H^{s,0}}+\Abs{f}_2\big).
$$
Controlling the commutator  in terms of $s-1$ derivatives of
$\nabla_{X,z}Q$ one get after a finite induction (see Proposition 2.36
if \cite{Lannes_book} for details on the control of this commutator
term) that for all $0\leq s\leq N-1$,
$$
\Abs{\Lambda^s\nabla_{X,z}Q}_2=\Abs{\nabla_{X,z}\tilde Q}_2\leq M_N (\Abs{{\bf
    g}}_{H^{s,1}}+\Abs{{\bf h}}_{W^{1,\infty}}\Abs{\omega}_{H^{s,0}}+\Abs{f}_2\big).
$$
Using commutator estimates (e.g. Corollary B.16 in \cite{Lannes_book})
and the expression of $f$, we also get (with $t_0>d/2$)
$$
\Abs{f}_2\leq M_N\big(\Abs{{\bf g}}_{H^{s,1}}+\Abs{{\bf h}}_{L^\infty
  H^{s\vee (t_0+1)}}\Abs{\omega}_{H^{s,1}}\big),
$$
and therefore
$$
\Abs{\nabla_{X,z}Q}_{H^{s,0}}\leq M_N (\Abs{{\bf
    g}}_{H^{s,1}}+\Abs{{\bf h}}_{W^{1,\infty}
  H^{s\vee (t_0+1)}}\Abs{\omega}_{H^{s,1}}\big).
$$
Using the equation to express $\dz^2 Q$ in terms of first and second
order derivatives of $Q$ involving at most one vertical derivative,
we get, for all $1\leq k\leq s$,
$$
\Abs{\nabla_{X,z}Q}_{H^{s,k}}\leq M_N (\Abs{{\bf
    g}}_{H^{s,k}}+\Abs{{\bf h}}_{W^{k,\infty}
  H^{s\vee (t_0+1)}}\Abs{\omega}_{H^{s,k}}\big),
$$
and the result of the lemma follows.
\end{proof}
Applying the lemma to \eqref{eqpress} with $k=N-2$, ${\bf g}={\bf g}^\iota$ and ${\bf
  h}=\nabla_{X,z}^\sigma
(\frac{z+H_0}{H_0}J^\iota(\uU\cdot N)$), we deduce from the product
estimate $\Abs{fg}_{H^{N-2}}\lesssim
\Abs{f}_{H^{N-2}}\Abs{g}_{H^{N-2}}$ and the third point of Lemma \ref{lemmaK} that
\begin{eqnarray}
\nonumber
\Abs{\nabla_{X,z}^\sigma Q}_{H^{N-2}}&\lesssim&
\big(\Abs{U}_{H^{N-1}}+\Abs{a}_{H^{N-2}}\big)\Abs{\omega}_{H^{N-2}}\\
\label{estQ}
&\lesssim&
C(\abs{\zeta}_{H^{N}},\abs{\nabla\psi}_{H^{N-1}},\Abs{\omega}_{H^{N-2}},\abs{\omega_b\cdot
N_b}_{H_0^{-1/2}}),
\end{eqnarray}
where we used the definition of $a$ in terms of $U$ and Corollary \ref{corHkk} to derive
the second inequality.

We can now proceed to construct a solution to the regularized system
\eqref{WWreg}. Let us  introduce the space ${\mathbb X}$ and its open
subset ${\mathbb X}_0$ as
\begin{align*}
{\mathbb X}&=H^{N}(\R^d)\times
\dot{H}^{N}(\R^d)\times H_b^{N-2}(\cS), \\
{\mathbb X}_0&=\{(\zeta,\psi,\omega)\in {\mathbb X}, \quad
\inf_{\R^d}\abs{H_0+\zeta}>0,\quad {\mathfrak a}(\zeta,\psi,\omega)>0\}.
\end{align*}
From \eqref{propJ} and Corollary \ref{corHkk}, ${\mathcal F}_j^{\delta,\iota}$ ($j=1,2$) define smooth
mappings\footnote{This follows from the fact that
  ${\mathbb U}^\sigma[\zeta](\psi,\omega)$ has a Lipschitz dependence
  on $\zeta$, $\psi$ and $\omega$, as shown in Corollary \ref{lipomega}.} on ${\mathbb X}_0$ with values in $H^\infty(\R^d)$. Using also the third point of Lemma \ref{lemmaK}, together with
\eqref{estQ}, ${\mathcal F}_3^{\delta,\iota}$ defines a smooth mapping
on ${\mathbb X}_0$ with values in $H^{N-2}(\cS)$.
We can therefore deduce that ${\mathcal F}^{\iota,\delta}$ maps
${\mathbb X}_0$ into ${\mathbb X}$ provided that for
all $\Theta\in {\mathbb X}_0$, one has ${\mathcal
  F}^{\delta,\iota}_3(\Theta)_\bottf\cdot N_b\in H_0^{-1/2}(\R^d)$. One computes after recalling
that $a$ and $w$ vanish at the bottom,
$$
{\mathcal
  F}^{\delta,\iota}_3(\Theta)_\bottf\cdot N_b=J^\iota\big(V\cdot\nabla
\omega_v\big)_\bottf-J^\iota\big(\omega_v\dz^\sigma
w\big)_\bottf,
$$
where $\omega_v$ stands for the vertical component of $\omega$. Now,
since by construction $U$ is divergence free, one has $\dz^\sigma
w_\bottf=-\nabla\cdot V_b$ and therefore
$$
{\mathcal
  F}^\iota_3(\Theta)_\bottf\cdot N_b=\nabla\cdot J^\iota\big(V
\omega_v\big)_\bottf,
$$
which implies the desired result. It follows that ${\mathcal F}^{\iota,\delta}$
defines a smooth mapping  ${\mathbb X}_0$ with values in ${\mathbb X}$; by
standard results on ODEs, there exists a maximal existence time
$T^{\iota,\delta}$ such that there exists a unique solution $\Theta\in
C^1([0,T^{\iota,\delta});{\mathbb X})$ to \eqref{WWreg}.

\noindent
{\bf Step 5.} Properties of the solution to the regularized
system. The solution constructed in the previous step has some extra
regularity properties. One has for instance $(\zeta,\psi)\in
C([0,T^{\iota,\delta});H^\infty(\R^d)\times \dot{H}^\infty(\R^d))$. One deduces
in particular that for all $\alpha\in \N^d$, $1<\abs{\alpha}\leq N$,
one has $\abs{\partial^\alpha\psi -\underline{\mathbb
    w}^\sigma[\zeta](\psi,\pi[\zeta]\omega)\partial^\alpha\zeta}_{\dot{H}^{1/2}}<\infty$. We
can therefore deduce from the third point of Lemma \ref{lemmaK},
Corollary \ref{corHkk} and \eqref{eqPi} that, $\zeta$ and $\psi$ being
fixed, the mapping  ${\mathcal
    F}_3^{\iota,\delta}(\zeta,\psi,\cdot)$ maps $H^{N-1}(\cS)$ into
  itself, from which one classically deduces that
  $\omega\in C([0,T^{\iota,\delta});H^{N-1}(\cS))$. In addition to
  this regularity property, we now show that $\omega$ remains
  divergence free. After remarking that
\begin{align*}
\dives\dt\omega&=\dives(\partial^\sigma_t\omega+\dt\sigma\dz^\sigma\omega)\\
&=\partial^\sigma_t\dives\omega+\nabla_{X,z}^\sigma (\dt\sigma)\cdot
\dz^\sigma\omega+\dt\sigma\dz^\sigma\dives \omega\\
&=\dt \dives\omega+\nabla_{X,z}^\sigma(\dt\sigma)\cdot \dz^\sigma \omega,
\end{align*}
we can use the definition \eqref{eqpress} of the pressure $Q$, and
apply $\dives$ to the vorticity equation to get
$$
\dt \dives \omega+\nabla_{X,z}^\sigma(\dt\sigma)\cdot \dz^\sigma \omega-\nabla_{X,z}^\sigma
\Big(\frac{z+H_0}{H_0}J^\iota(\uU\cdot N)\Big)\cdot \dz^\sigma\omega=0.
$$
Since we get from the equation on $\zeta$ and the definition of
$\sigma$ that $\frac{z+H_0}{H_0}J^\iota(\uU\cdot N)=\dt\sigma$, we
finally obtain that
$$
\dt \dives\omega=0;
$$
since the initial condition $\omega^0$ is assumed to be divergence
free, this yields that $\dives \omega=0$ on $[0,T^{\iota,\zeta})$. A
consequence of this is that $\omega=\pi[\zeta]\omega$, so that we can
drop all the occurrences of $\pi[\zeta]$ in \eqref{WWreg}.

\noindent
{\bf Step 6.} Uniform energy estimates. Proceeding exactly as for
Proposition \ref{prop3000}, but with the regularized equations \eqref{WWreg},
we obtain, with the same notations,
\begin{eqnarray*}
(\dt +J^\iota(\uV\cdot\nabla))
\partial^\alpha\zeta-J^\iota\partial_k\uU_{(\beta)}\cdot N&=&J^\iota R^1_\alpha,\\
(\dt +J^\iota(\uV\cdot \nabla)) (U_{(\beta)\parallel}\cdot {\bf e}_k)+J^\iota(\mfa \partial^\alpha
\zeta)+\delta J^\iota \Lambda \partial^\alpha \zeta&=&J^\iota R^2_\alpha+\tilde R^2_\alpha,\\
(\dt^\sigma+J^\iota(V\cdot \nabla) +K^\iota[a^\iota\dz])\partial^\beta \omega&=&R^3_\beta,
\end{eqnarray*}
with the  $R^j_\alpha$ satisfying the same estimates as in Proposition
\ref{prop3} (for the vorticity equation, we use the fact that
$\Abs{\nabla_{X,z}Q}_{H^{N-1}}\leq \mfN(\zeta,\psi,\omega)$, which is a
consequence of Lemma \ref{estpressQ}), and
$$
\tilde R_\alpha^2=-(1-J^\iota)(\dt
\uw \partial^\alpha\zeta)+[\uw,J^\iota]\partial^\alpha
\big(\underline{{\mathbb U}}^\sigma[\zeta](\psi,\omega)\cdot N\big);
$$
the control of this extra term (coming from commutators with the mollifiers
$J^\iota$) in $\dot{H}^{1/2}$ norm requires a control of $\zeta$ in
$H^{N+1/2}(\R^d)$ instead of $H^{N}(\R^d)$, namely,
$$
\abs{\proj \tilde R_\alpha^2}_2\leq \mfN(\zeta,\psi,\omega)(1+\abs{\partial^\alpha\zeta}_{H^{1/2}});
$$
the dispersive regularization $\delta J^\iota \Lambda \zeta$ has been
added to the second equation in order to control this extra term.
Proceeding exactly as for \eqref{esten3} except for the fact that the
first equation is multiplied by $\mfa \partial^\alpha\zeta+\delta
\Lambda\partial^\alpha\zeta$ instead of $\mfa \partial^\alpha\zeta$,
we obtain therefore
\begin{align}
\nonumber
\frac{1}{2}\dt \big[
(\mfa \partial^\alpha\zeta,\partial^\alpha\zeta)+\delta \abs{\partial^\alpha\zeta}_{H^{1/2}}^2\big]&+
\big((\dt+J^\iota(\uV\cdot\nabla\cdot)) (\Vpb\cdot {\bf
  e}_k),\partial_k\uU_{(\beta)}\cdot N\big)\\
\label{esten3reg}
&\leq \mfN(\zeta,\psi,\omega) (1+\abs{\partial^\alpha\zeta}_{H^{1/2}}).
\end{align}
As in the proof of Proposition \ref{prop3000}, the second term is handled using
Green's formula; because of the presence of the mollifiers,
\eqref{lift} must be replaced by
$$
(\dt+J^\iota(\uV\cdot\nabla\cdot)) (\Vpb\cdot {\bf
  e}_k)=\big({(\dt +J^\iota(V \cdot \nabla)+K^\iota[a^\iota \dz] ) (\Vpb^\flat\cdot {\bf
  e}_k)\big)}_\surff
$$
(we recall that $a^\iota$ vanishes at the surface).
We are therefore led to replace \eqref{BTZ} by
\begin{align*}
\nonumber
\big((\dt+&J^\iota(\uV\cdot\nabla\cdot)) (\Vpb\cdot {\bf
  e}_k),\partial_k\uU_{(\beta)}\cdot N\big)\\
=&\int_\cS (1+\dz\sigma)\big[(\dt +J^\iota(V\cdot \nabla)+K^\iota[a^\iota \dz] ) \nabla_{X,z}^\sigma (\Vpb^\flat\cdot {\bf
  e}_k)\big]\cdot \partial_k U_{(\beta)}
+\mbox{l.o.t.}
\end{align*}
From Lemma \ref{lemma2}, we therefore get
\begin{align*}
\big((\dt+&J^\iota(\uV\cdot\nabla)) (\Vpb\cdot {\bf
  e}_k),\partial_k\uU_{(\beta)}\cdot N\big)\\
=&\int_\cS (1+\dz\sigma)\big[(\dt +J^\iota(V\cdot \nabla)+K^\iota[a^\iota \dz] )
\partial_k U_{(\beta)}\big]\cdot \partial_k U_{(\beta)}\\
&+\int_\cS (1+\dz\sigma)\big[{\bf t}_k\times(\dt +J^\iota(V\cdot \nabla)+K^\iota[a^\iota \dz] )
\partial^\beta\omega\big]\cdot \partial_k U_{(\beta)}+\mbox{l.o.t.}\\
=&A+B
\end{align*}
Integrating by parts (using the fourth point of Lemma \ref{lemmaK} for
the term involving $K^\iota[a^\iota\dz]$, one readily obtains that the
identity \eqref{Baiona1} remains true up to lower order terms, while
$B$ is controlled as in \eqref{Baiona2}. The energy inequality
\eqref{pouf} can therefore be adapted into
\begin{equation}\label{pouf2}
\frac{d}{dt}\tilde\cE_\delta^N(\zeta,\psi,\omega)\leq C\big(\frac{1}{a_0},\frac{1}{\delta},\tilde\cE_\delta^N(\zeta,\psi,\omega)\big),
\end{equation}
where
$$
\tilde\cE_\delta^N(\zeta,\psi,\omega)=\tilde\cE^N(\zeta,\psi,\omega)+\delta \abs{\zeta}^2_{H^{N+1/2}}
$$
{\bf Step 7.} Conclusion. The energy estimate \eqref{pouf2} is uniform
with respect to $\iota$. In particular, the solutions
$(\Theta^{\iota,\delta})_{\iota,\delta}$ to \eqref{WWreg} exist on a
non trivial time interval
$[0,T^\delta]$ independent of $\iota$ and one can prove that, $\delta>0$
being fixed, the sequence $(\theta^{\iota,\delta})_{\iota,\delta}$ is
a Cauchy sequence in $H^2\times \dot H^2\times H^{2}(\cS)$ as
$\iota\to 0$.
\begin{lemma}\label{lemCS}
For all $0<\delta<1$, $(\Theta^{\iota,\delta})_{\iota}$ is a
Cauchy sequence, as $\iota\to 0$, in $C([0,T^\delta];H^2(\R^d)\times
\dot{H}^2(\R^d)\times H^{2}(\cS))$.
\end{lemma}
\begin{proof}
We omit the proof of this result because it is very similar to what
happens in the irrotational case, and the proof of Lemma 4.28 in
\cite{Lannes_book} can therefore easily be adapted. Note that this
result is a consequence of \eqref{propJ2}. The only additional
property that is needed here compared to the irrotational case is that
a similar property holds for the vertical mollifier $K^\iota[a\dz]$,
but this is already proved in the fifth point of Lemma \ref{lemmaK}.
\end{proof}
The end of the proof is then quite similar to the irrotational case
and we only give the main steps (see \S 4.3.4.4 of \cite{Lannes_book}
for more details). One first deduces the existence of a limit
$\Theta^\delta \in C([0,T^\delta];H^2(\R^d)\times
\dot{H}^2(\R^d)\times H^{2}(\cS))$ to the sequence
$(\Theta^{\iota,\delta})_{\iota}$ as $\iota\to 0$, and that this
limit solves
$$
\dt \Theta^\delta +{\mathcal F}^{\delta,0}=0,
$$
where ${\mathcal F}^{\delta,0}$ is deduced from ${\mathcal
  F}^{\delta,\iota}$ by setting $\iota$ to $0$; moreover,
$\Theta^\delta$ also satisfies the energy estimate
\eqref{pouf2}. Because of the dependence on $\frac{1}{\delta}$ of this
energy estimate, one cannot directly extract a converging sequence
from $(\Theta^\delta)_{\delta}$ as $\delta\to 0$. However, this
singular dependence on $\delta$ of the energy estimate comes from the
term $(1+\abs{\partial^\alpha\zeta}_{H^{1/2}})$ in the right-hand-side
of \eqref{esten3reg}, which itself comes from the control of $\tilde
R^2_\alpha$. Since this term vanishes when $\iota=0$, one can replace
\eqref{pouf2} by a nonsingular (with respect to $\delta$) energy
estimate,
\begin{equation}\label{pouf3}
\frac{d}{dt}\tilde\cE_\delta^N(\zeta,\psi,\omega)\leq C\big(\frac{1}{a_0},\tilde\cE_\delta^N(\zeta,\psi,\omega)\big).
\end{equation}
This uniform bound can then be used to prove, as in Lemma \ref{lemCS}
that $(\Theta^\delta)_{0<\delta<1}$ is a Cauchy sequence as $\delta
\to 0$ and that it converges in $C([0,T];H^2(\R^d)\times
\dot{H}^2(\R^d)\times H^{2}(\cS))$ (with $T>0$ independent of $\delta$) to a solution $\Theta$ of the
non-regularized equations \eqref{ZCSgens2}. Moreover, the solutions
satisfies \eqref{pouf3} with $\delta$ set to $0$, namely,
$$
\frac{d}{dt}\tilde\cE^N(\zeta,\psi,\omega)\leq
C\big(\frac{1}{a_0},\tilde\cE^N(\zeta,\psi,\omega)\big),
$$
and the bound on $\cE^N(\Theta)$ given in the statement of the theorem
follows as for the a priori estimates of Proposition
\ref{prop3000}. Uniqueness of the solution is then obtained by
estimating the difference of two solutions in $C([0,T];H^2(\R^d)\times
\dot{H}^2(\R^d)\times H^{2}(\cS))$ using the uniform bound provided by
the energy estimates, along lines quite similar to the proof of Lemma \ref{lemCS}.
\end{proof}

\section{Asymptotic regimes}\label{sect5}

\subsection{The dimensionless free surface Euler  equations}

The fluid motion depends qualitatively on several physical parameters:
the typical amplitude $a$ of the waves, the depth at rest $H_0$, and
the typical horizontal scale $L$. Using these quantities, it is
possible to form two dimensionless parameters,
$$
\eps=\frac{a}{H_0},\qquad \mu=\frac{H_0^2}{L^2};
$$
the parameter $\eps$ is often called nonlinearity (or amplitude)
parameters, and the parameter $\mu$ is the shallowness parameter.\\
We also use $a$, $H_0$ and $L$ to define dimensionless variables and
unknowns (written with a tilde),
$$
\tilde z=\frac{z}{H_0},\qquad \tilde X=\frac{X}{L},\qquad \tilde\zeta=\frac{\zeta}{a};
$$
the non dimensionalization of the time variable and of the velocity,
rotational, and pressure
fields is less obvious, and is based on the linear analysis of the
equations (see for instance \cite{Lannes_book}, Chapter 1),
$$
\tilde {\bf V}=\frac{{\bf V}}{V_0}, \quad \tilde {\bf w}=\frac{{\bf
    w}}{w_0},\quad
\tilde
t=\frac{t}{t_0},\quad \tilde P=\frac{P}{P_0}
$$
with
$$
V_0=a\sqrt{\frac{g}{H_0}},\quad
w_0=\frac{aL}{H_0}\sqrt{\frac{g}{H_0}},\quad
t_0=\frac{L}{\sqrt{gH_0}},\quad P_0=\rho g H_0.
$$
With these variables and unknowns, and with the notations
$$
\bU^\mu=\left(\begin{array}{c}\sqrt{\mu}{\bf V}\\ {\bf w}\end{array}\right),\qquad
\nabla^\mu=\left(\begin{array}{c}\sqrt{\mu}\nabla\\
    \dz\end{array}\right),
\qquad
N^\mu=\left(\begin{array}{c}-\eps\sqrt{\mu}\nabla\zeta\\ 1\end{array}\right),
$$
and
$$
\curlm=\nabla^\mu\times,\qquad \divem=(\nabla^\mu)^T,\qquad \uU^\mu=(\sqrt{\mu}\uV^T,\uw)^T:=\bU^\mu_{\vert_{z=\eps\zeta}},
$$
the incompressible Euler equations take the form
(omitting the tildes),
\begin{align*}
\dt \bU^\mu+\frac{\eps}{\mu} {\bf U}^\mu
\cdot\nabla^\mu\bU^\mu&=-\frac{1}{\eps}\big(\nabla^\mu P + {\bf
  e}_z\big)&\mbox{in }\Omega,\\
\divem {\bf U}^\mu&=0&\mbox{in }\Omega,
\end{align*}
where $\Omega$ now stands for the dimensionless fluid domain,
$$
\Omega=\{(X,z)\in \R^{d+1},\quad -1<z<\eps\zeta(t,X)\},
$$
and with the non vanishing depth condition now reading
\begin{equation}\label{hminND}
\exists h_{\rm min}>0,\qquad \forall X\in \R^d, \quad 1+\eps\zeta\geq
h_{\rm min}.
\end{equation}
Finally, the boundary conditions on the velocity read in dimensionless form,
\begin{align*}
\dsp \dt \zeta -\frac{1}{\mu}\uU^\mu\cdot N^\mu&=0\quad \mbox{ at the surface,}\\
{\bU}^\mu_\bottff\cdot N^\mu_b&=0\quad\mbox{ at the bottom},
\end{align*}
where $N_b^\mu={\bf
  e}_z$, while for the pressure, we still have
$$
P=0\quad \mbox{ at the surface}.
$$

\subsection{Notations}

We give here a list of notations specific to the study of the shallow
water regime. Most of them are the dimensionless version of notations
already used in the dimensional case; for the sake of clarity, we
write them in the same way. Below is a list adaptations we need to
make to handle the dimensionless case:
\begin{eqnarray*}
\Omega=\{(X,z), \, -H_0< z<\zeta(X)\} &\leadsto& \Omega=\{(X,z), \,
-1< z<\eps \zeta(X)\} \\
\cS=\R^d\times (-H_0,0)&\leadsto&\cS=\R^d\times (-1,0),\\
\sigma(X,z)=\zeta\frac{z+H_0}{H_0}&\leadsto&
\sigma(X,z)=\eps\zeta(z+1),\\
\proj=\frac{\abs{D}}{(1+\abs{D})^{1/2}} &\leadsto &
\proj=\frac{\abs{D}}{(1+\sqrt{\mu}\abs{D})^{1/2}},\\
\abs{u}_{H_0^{-1/2} }=\abs{\frac{1}{\abs{D}}u}_{H^{1/2}}&\leadsto& \abs{u}_{H_{0}^{-1/2}}=\abs{\frac{(1+\sqrt{\mu}\abs{D})^{1/2}}{\abs{D}}u}_2.
\end{eqnarray*}
We also adapt the notation \eqref{notapar} as follows,
\begin{equation}\label{notaparmu}
A_\parallel=\frac{1}{\sqrt{\mu}}\uA_h+\eps \uA_v\nabla\zeta
\end{equation}
so that $\underline{A}\times
N^\mu=\sqrt{\mu}\left(\begin{array}{l}-\Ap^\perp\\-\eps\sqrt{\mu}\Ap^\perp\cdot
    \nabla\zeta\end{array}\right)$. Finally, we also write
\begin{equation}\label{nablasigmamu}
\nabla^{\sigma,\mu}=\left(\begin{array}{c}
    \sqrt{\mu}\nabla\\0\end{array}\right)+\left(\begin{array}{c}-\sqrt{\mu}\nabla\sigma
    \\1\end{array}\right)\dz^\sigma,
\end{equation}

\subsection{The dimensionless generalized ZCS formulation}

According to the notation \eqref{notaparmu}, we have
$$
U^\mu_\parallel=\uV+\eps\uw\nabla\zeta=\frac{1}{\sqrt{\mu}}\big( \uU^\mu\times
N^\mu)_h^T,
$$
and proceeding as in \S \ref{sectTheeq}, we deduce the following
dimensionless version of \eqref{eqVp},
\begin{equation}\label{eqVpND}
\dsp \dt \Vp^\mu +\nabla\zeta+\frac{\eps}{2}\nabla \abs{\Vp^\mu}^2-\frac{\eps}{2\mu}\nabla\big(
(1+\eps^2\mu\abs{\nabla\zeta}^2)\uw^2\big)+\eps\uom_\mu\cdot N^\mu\uV^\perp=0,
\end{equation}
where $\uom_\mu={\bom_\mu}_{\vert_{z=\eps\zeta}}$ and $\bom_\mu$ is given by
$$
\bom_\mu=\left(\begin{array}{c}\frac{1}{\sqrt{\mu}}\big(\dz {\bf V}^\perp-\nabla^\perp  {\bf w}\big)\\
-\nabla \cdot {\bf V}^\perp
\end{array}\right)=\frac{1}{{\mu}}\curlm \bU^\mu.
$$
The dimensionless version of the orthogonal decomposition of $\Vp$
performed in \S \ref{sectTheeq} is then given by
$$
\Vp^\mu=\nabla\psi+\nabla^\perp \tpsi,
$$
with $\Delta \tpsi=\uom_\mu\cdot N^\mu$. The
equation on $\psi$ corresponding to the dimensionless version of
\eqref{4eq1} is therefore
\begin{equation}\label{4eq1ND}
 \dt \psi +\zeta+\frac{\eps}{2} \abs{\Vp^\mu}^2-\frac{\eps}{2\mu}\big(
(1+\abs{\nabla\zeta}^2)\uw^2\big)+\eps\frac{\nabla}{\Delta}\cdot\big( \uom_\mu\cdot N^\mu\uV^\perp\big)=0.
\end{equation}
Finally, the dimensionless vorticity equation is obtained by applying
$\nabla^\mu$ to the dimensionless Euler equation,
\begin{equation}\label{eqrotND}
\dt \bom_\mu+\frac{\eps}{\mu}\bU^\mu\cdot \nabla^\mu
\bom_\mu=\frac{\eps}{\mu}\bom_\mu\cdot \nabla^\mu \bU^\mu.
\end{equation}
In order to write $\bU^\mu$ as a function of
$\zeta$, $\psi$ and $\bom_\mu$, we need to solve the following dimensionless
version of the div-curl problem \eqref{divrot},
\begin{equation}\label{divrotND}
\left\lbrace
\begin{array}{lll}
\curlm \bU^\mu &= \mu\bom_\mu&\quad \mbox{ in }\quad \Omega\\
\divem \bU^\mu &=0&\quad \mbox{ in }\quad \Omega\\
\Vp^\mu&= \nabla\psi+\nabla^\perp \Delta^{-1}(\uom_\mu\cdot N^\mu)&\quad\mbox{ at the surface}\\
U^\mu_b\cdot N_b&=0 &\quad\mbox{ at the bottom};
\end{array}\right.
\end{equation}
we write the solution
$$
\bU^\mu={\mathbb U}^\mu[\eps\zeta](\psi,\bom_\mu)
=\left(\begin{array}{c}
\sqrt{\mu}{\mathbb
  V}[\eps\zeta](\psi,\bom_\mu)\\
{\mathbb
  w}[\eps\zeta](\psi,\bom_\mu)
\end{array}\right);
$$
the generalized Zakharov-Craig-Sulem formulation takes therefore the
following form in dimensionless form,
\begin{equation}\label{ZCSgenND}
\left\lbrace
\begin{array}{l}
\dsp \dt \zeta-\frac{1}{\mu}\underline{\mathbb
  U}^\mu[\eps\zeta](\psi,\bom_\mu)\cdot N^\mu=0,\\
\dsp \dt \psi +\zeta\!+\!\frac{\eps}{2}
 \babs{{\mathbb
     U}^\mu_\parallel[\eps\zeta](\psi,\bom_\mu)}^2-\frac{\eps}{2\mu}(1+\eps^2\mu\abs{\nabla\zeta}^2)
 \underline{{\mathbb w}}[\eps\zeta](\psi,\bom_\mu)^2 \\
\dsp \hspace{4.35cm}-\eps\frac{\nabla^\perp}{\Delta}\cdot\big(
\uom_\mu\cdot N^\mu \underline{\mathbb
  V}[\eps\zeta](\psi,\bom_\mu)\big)=0,\\
\dsp \dt \bom_\mu+\frac{\eps}{\mu}{\mathbb U}^\mu[\eps\zeta](\psi,\bom_\mu)\cdot\nabla^\mu
\bom_\mu=\frac{\eps}{\mu}\bom_\mu\cdot\nabla^\mu {\mathbb U}^\mu[\eps\zeta](\psi,\bom_\mu).
\end{array}\right.
\end{equation}

\subsection{Statement of the main result}\label{statMRD}

As in the dimensional case, the statement of the well-posedness result
requires to work with a straightened vorticity. We therefore use a diffeomorphism $\Sigma$ to straighten the fluid
domain; it now takes the form $\Sigma(X,z)=(z,z+\sigma(X,z))$ where
$\sigma(X,z)=(1+z)\ep\zeta(X)$ and maps the strip  $\cS=\R^d\times
(-1,0)$ to $\Omega$. We denote $U^\mu:= \bU^\mu\circ \Sigma$,
$\omega_\mu:= \bom_\mu\circ \Sigma$, etc., and also
\begin{align*}
\mathbb{U}^{\sigma,\mu}[\ep\zeta](\psi,\omega_\mu)
\left(\begin{array}{c}
\sqrt{\mu}\mathbb{V}^{\sigma}[\ep\zeta](\psi,\omega_\mu)\\
\mathbb{w}^{\sigma}[\ep\zeta](\psi,\omega_\mu)
\end{array}\right).
:=
\mathbb{U}^{\mu}[\ep\zeta](\psi,\bom_\mu)\circ \Sigma.
\end{align*}
The well-posedness result deals therefore with the following
straightened version of \eqref{ZCSgenND},
\begin{equation}\label{ZCSgenNDS}
\left\lbrace
\begin{array}{l}
\dsp \dt \zeta-\frac{1}{\mu}\underline{\mathbb
  U}^{\sigma,\mu}[\eps\zeta](\psi,\omega_\mu)\cdot N^\mu=0,\\
\dsp \dt \psi +\zeta\!+\!\frac{\eps}{2}
 \babs{{\mathbb
     U}^{\sigma,\mu}_\parallel[\eps\zeta](\psi,\omega_\mu)}^2-\frac{\eps}{2\mu}(1+\eps^2\mu\abs{\nabla\zeta}^2)
 \underline{{\mathbb w}}^{\sigma}[\eps\zeta](\psi,\bom_\mu)^2 \\
\dsp \hspace{4.35cm}-\eps\frac{\nabla^\perp}{\Delta}\cdot\big(
\uom_\mu\cdot N^\mu \underline{\mathbb
  V}^{\sigma}[\eps\zeta](\psi,\omega_\mu)\big)=0,\\
\dsp \dt^\sigma \omega_\mu+\frac{\eps}{\mu}{\mathbb U}^{\sigma,\mu}[\eps\zeta](\psi,\omega_\mu)\cdot\nabla^{\sigma,\mu}
\omega_\mu=\frac{\eps}{\mu}\omega_\mu\cdot\nabla^{\sigma,\mu} {\mathbb U}^{\sigma,\mu}[\eps\zeta](\psi,\omega_\mu),
\end{array}\right.
\end{equation}
together with the divergence free condition on $\omega_\mu$ which is
propagated from the initial condition,
\begin{equation}\label{DFvort2mu}
\nabla^{\sigma,\mu}\cdot \omega_\mu=0\quad \mbox{ in }\quad {\mathcal S}.
\end{equation}
The statement of the theorem also requires the introduction of the dimensionless energy
\begin{align}
\nonumber
\cE^N(\zeta,\psi,\omega_\mu):=&\frac{1}{2}\abs{\zeta}_{H^N}^2+\frac{1}{2}\abs{\proj\psi}_{H^3}^2+\frac{1}{2}\sum_{0<\abs{\alpha}\leq
N}\abs{\proj\psi_{(\alpha)}}_2^2\\
\label{NRJmu}
&+\frac{1}{2}\Abs{\omega_\mu}_{H^{N-1}}^2+\frac{1}{2}\abs{\omega_{\mu,b}\cdot
N_b}_{H^{-1/2}_{0}}^2,
\end{align}
with $\psi_{(\alpha)}=\partial^\alpha\psi-\eps
\underline{w}\partial^\alpha\zeta$ (and $\underline{w}=\underline{\mathbb w}^{\sigma}[\eps\zeta](\psi,\omega_\mu)$),
and we still denote by $\mfN(\zeta,\psi,\omega_\mu)$ any constant of
the form
$$
\mfN(\zeta,\psi,\omega_\mu)=C\big(\frac{1}{h_{\rm
    min}},\cE^N(\zeta,\psi,\omega_\mu)\big),
$$
and by $E_T^N$ the associated functional space defined in
\eqref{defETN}. Note also that the dimensionless version of the
Rayleigh-Taylor coefficient is
$$
{\mathfrak a}={\mathfrak a}(\zeta,\psi,\omega_\mu)=1+\eps (\dt+\underline{\mathbb
  V}^\sigma[\eps\zeta](\psi,\omega_\mu))\underline{\mathbb
  w}^\sigma[\eps\zeta](\psi,\omega_\mu).
$$
The theorem states that the solution furnished by Theorem \ref{theo1} exists on a time interval
$[0,T/\eps]$, with $T$ independent of $\eps\in (0,\eps_0)$ and $\mu\in
(0,\mu_0)$, and that it is uniformly bounded on this time interval.
\begin{thm}\label{theomainND}
Let $\eps_0,\mu_0>0$, $\eps\in(0,\eps_0)$, $\mu\in(0,\mu_0)$, and $N\geq
5$. Let also $\Theta^0=(\zeta^0,\psi^0,\omega_\mu^0)\in E_0^N$ be such that
$\omega_\mu^0$ satisfies the divergence free condition
\eqref{DFvort2mu}. Assume moreover that
$$
\exists h_{\rm min}>0,\exists a_0>0,\quad 1+\eps \zeta^0>h_{\rm
  min},\quad {\mathfrak a}(\zeta^0,\psi^0,\omega^0)>a_0.
$$
Then there exists $T>0$ (independent of $\eps$ and $\mu$), and a unique solution $\Theta\in E^N_{T/\eps}$ to
\eqref{ZCSgenNDS} satisfying the divergence free constraint
\eqref{DFvort2mu}, and with initial condition $\Theta^0$. Moreover,
$$
\frac{1}{T}=c^1\quad\mbox{ and }\quad \sup_{t\in
  [0,T]}\cE^N(\Theta(t))=c^2
$$
with $\dsp c^j=C(\cE^N(\Theta^0),\frac{1}{h_{\rm min}},\frac{1}{a_0},\eps_0,\mu_0)$ for $j=1,2$.
\end{thm}
\begin{remark}
Note that no smallness assumption is made on $\eps_0$ and $\mu_0$. The
theorem furnishes in particular an existence time and bounds on the
solutions that are relevant to study many asymptotic regimes of
interest in oceanography:
\begin{itemize}
\item The (large amplitude) shallow water regime: here $\eps\sim 1$ and $\mu\ll
1$. We get an existence on a time interval of order $1$, uniformly
with respect to $\mu$.
\item The long wave (also called Boussinesq, or KdV in dimension
  $d=1$) regime: here $\eps\sim \mu \ll 1$. The existence is then on a
  larger time interval of order $O(1/\eps)$ with uniform bounds on
  this time scale.
\item The deep water regime. Here $\mu\sim 1$ and $\eps\ll 1$, and
  asymptotics can be studied in terms of $\eps$, on a time interval of
  order $O(1/\eps)$.
\end{itemize}
\end{remark}
\subsection{Proof of Theorem \ref{theomainND}}

Theorem \ref{theo1} furnishes the existence of a solution. We just
need to prove the necessary bounds on the solution with respect to
$\eps$ and $\mu$. This is done by deriving uniform a priori estimates
on the solution. The derivation of these estimates follows sometimes
the same steps as for the dimensional case already treated, but
sometimes require specific attention. We only focus on these latter
aspects, and omit (or only sketch) the proof of the former ones.\\
The dependence on $\eps, \mu$ of the div-curl problem is investigated
in \S \ref{DCpar}, the vorticity energy estimates are addressed in \S
\ref{sectVED}, and the a priori estimates on the full equations are
finally derived in \S \ref{sectAPD}.\\
{\it We always assume throughout this section that $\eps\in
  (0,\eps_0)$ and $\mu\in (0,\mu_0)$ for some $\eps_0,\mu_0>0$. For
  the sake of clarity, we never make explicit the dependence on
  $\eps_0$ and $\mu_0$.}

\subsubsection{The div-curl problem with parameters}\label{DCpar}

We can still invoke Theorem \ref{prop1} to insure the existence and
uniqueness of a solution to \eqref{divrotND}, however, the dependence
on the parameter $\mu$ is not obvious in the estimates on the solution
provided in Theorem \ref{prop1} and special attention must be paid to
avoid singular terms as $\mu \to 0$ (in the estimates involving the
bottom vorticity for instance). This dependence is made precise in
the following proposition.
\begin{proposition}\label{prop1mu}
Let $\zeta\in W^{2,\infty}(\R^d)$ be such that \eqref{hminND} is satisfied. Then
for all $\bom_\mu\in L^2(\Omega)^{3}$ such that $\divem
\bom_\mu=0$, and all $\psi\in \dot{H}^{3/2}(\R^d)$,  there exists a unique solution $\bU\in H^1(\Omega)^{3}$ to
\eqref{divrotND}, and one can decompose it as $\bU^\mu=\curlm{\bf
  A}+\nabla^\mu \Phi$, where $\Phi\in \dot{H}^2(\Omega)$ solves
$$
\left\lbrace\begin{array}{l}
(\dz^2+\mu \Delta)\Phi=0\qquad \mbox{ \rm in }\quad \Omega\\
\Phi_\surf=\psi,\qquad \dz\Phi_\bott=0,
\end{array}\right.
$$
while $\bA\in \dot{H}^2(\Omega)^{3}$
solves
$$
\left\lbrace
\begin{array}{rll}
\dsp \curlm \curlm \bA&=\mu\bom_\mu &\mbox{ \textnormal{in} }\Omega,\\
\dsp \divem \bA&=0 &\mbox{ \textnormal{in} }\Omega,\\
\dsp N^\mu_b\times A_b&=0&\\
\dsp N^\mu\cdot \underline{A}&=0&\\
\dsp (\curlm
  \bA)_\parallel &=\nabla^\perp
\Delta^{-1}\uom_\mu\cdot N^\mu,&\\
\dsp N^\mu_b\cdot(\curlm \bA)_\bott & = 0.&
\end{array}\right.
$$
Moreover, one has
\begin{align*}
\Abs{\bU^\mu}_2
&\leq \sqrt{\mu}C(\abs{\zeta}_{W^{2,\infty}},\frac{1}{h_{\rm
    min}}) \big(\sqrt{\mu}\Abs{\bom_\mu}_{2}+\abs{\bom_{\mu,b}\cdot N_\mu}_{H_{0}^{-1/2}}+\abs{\proj\psi}_2\big),\\
\Abs{\nabla^\mu \bU^\mu}_2&\leq
\mu C(\abs{\zeta}_{W^{2,\infty}},\frac{1}{h_{\rm min}})
\big(\Abs{\bom_\mu}_{2}+\abs{\bom_{\mu,b}\cdot
  N_b^\mu}_{H_{0}^{-1/2}}+\abs{\proj\psi}_{H^1}\big).
\end{align*}
\end{proposition}
\begin{proof}
The existence/uniqueness of  $\bU^\mu$, follows directly (up to a
rescaling) from Theorem \ref{prop1}. The fact that $\nabla^\mu \Phi$
satisfies the estimates of the proposition is known from the
irrotational case (Corollary 2.40 in \cite{Lannes_book}); by
linearity, we can therefore assume that $\psi=0$ and therefore
$\bU^\mu=\curlm {\bf A}$ with ${\bf A}$ as in the statement of the
proposition; we also know from the proof of Theorem \ref{prop1} that $\bA$ is the unique
solution to the dimensionless version of the variational equation
\eqref{varform},
\begin{align}
\int_{\Om}\curlm\bA\cdot \curlm {\bf C}=\mu\int_{\Omega}\bom_\mu\cdot
{\bf C}+\sqrt{\mu}\int_{\R^d}(\nabla\tilde{\psi},\sqrt{\mu}\eps\nabla\tilde{\psi}\cdot\nabla\zeta)\cdot\underline{C}, \label{mu4}\end{align}
for all ${\bf C}\in {\mathfrak X}^\mu$, and where the space $\mathfrak{X}^\mu$ is given by
\begin{align*}
\mathfrak{X}^\mu=\{{\bf C}\in H^1(\Omega)^{d+1}\quad \nabla^\mu\cdot {\bf
  C}=0, \quad N^\mu\times C_b=0, \quad N_b^\mu\cdot\underline{ C}=0\}.
\end{align*}
We shall use the following lemma instead of the standard trace lemma
that does not provide a sharp dependence on $\mu$.
\begin{lemma}\label{lemlemmu}
{\bf i.} For all ${\bf C}\in {\mathfrak X}^\mu$, one has
$$
\Abs{{\bf C}}_2\leq C(\abs{\zeta}_{W^{1,\infty}})\Abs{\dz {\bf C}}_2
\quad\mbox{ and }\quad
\Abs{\nabla^\mu {\bf C}}_2\leq
C(\abs{\zeta}_{W^{2,\infty}})\Abs{\curlm {\bf C}}_2.
$$
(it is not necessary that $\bf C$ be divergence free for the first inequality).\\
{\bf ii.} For all ${\bf C}\in H^1(\Omega)$ such that ${\bf
  C}_\bottff=0$, one has
$$
\abs{(1+\sqrt{\mu}\abs{D})^{1/2} \underline C}_2\leq
C(\abs{\zeta}_{W^{1,\infty}},\frac{1}{h_{\rm
    min}})\Abs{\nabla^\mu{\bf C}}_2.
$$
{\bf iii.} For all ${\bf C}\in H^1(\Omega)$, such that ${\bf
  C}_\bottff=0$, one has
$$
\abs{\proj\underline{C}}_{2}\leq
\frac{1}{\sqrt{\mu}}C(\abs{\zeta}_{W^{1,\infty}})\Abs{\nabla^\mu {\bf C}}_2;
$$
if ${\bf C}$ does not vanish at the bottom, then we still have
$$
\abs{\proj\underline{C}}_{2}\leq
\frac{1}{\sqrt{\mu}}C(\abs{\zeta}_{W^{1,\infty}})\big(\Abs{\nabla^\mu
  {\bf C}}_2+\Abs{{\bf C}}_2).
$$
\end{lemma}
\begin{proof}[Proof of the lemma]
For the first point, one just has to track the dependance on $\mu$ in the proofs of Lemmas
\ref{lempoincare}, \ref{leminterm} and \ref{equivnorm}). For the
second point, denoting by $\Sigma: {\mathcal S}\to \Omega$ the diffeomorphism
defined by $\Sigma(X,z)=(X,z+(1+z)\eps\zeta)$, and writing $C={\bf
  C}\circ\Sigma$, we have
\begin{eqnarray*}
\abs{(1+\sqrt{\mu}\abs{D})^{1/2}\underline
  C}_2^2&=&2\Re\int_{\R^d}\int_{-1}^{\eps\zeta}(1+\sqrt{\mu}\abs{\xi})\widehat{C}\dz\widehat{C}\\
&\leq &2 (\Abs{C}_2+\sqrt{\mu}\Abs{\nabla C}_2)\Abs{\dz C}_2\\
&\leq&  C(\frac{1}{h_{\rm
      min}},\abs{\zeta}_{W^{1,\infty}})(\Abs{{\bf
      C}}_2+\Abs{\nabla^\mu {\bf C}}_2)\Abs{\dz {\bf C}}_2,
\end{eqnarray*}
and the result follows from the first part of the lemma.\\
For the third point, we have with $C={\bf C}\circ\Sigma$,
\begin{align*}
|\mathfrak{P}\underline C|_{2}^2&=
|\mathfrak{P}{C}(\cdot,0)|_{2}^2\\
&=
\int_{\R^d}\frac{|\xi|^2}{1+\sqrt{\mu}|\xi|}|\widehat{{C}}(\xi,0)|_{L^2}^2\\
& \leq 2\int_{\R^d}\int_{-1}^0\frac{|\xi|^2}{1+\sqrt{\mu}|\xi|}|\widehat{{C}}(\xi,z)||\widehat{\pa_z{C}}(\xi,z)|d\xi dz\\
&= \frac{2}{\mu}\int
\frac{\sqrt{\mu}|\xi|}{1+\sqrt{\mu}|\xi|}|\widehat{\sqrt{\mu}\nabla{C}}(\xi,z)||\widehat{\pa_z{C}}(\xi,z)|d\xi
dz.
\end{align*}
From Cauchy-Schwarz inequality and Plancherel's identity, we then get
\begin{align*}
|\mathfrak{P}\underline C|_{2}^2&\leq
\frac{C}{\mu}||\sqrt{\mu}\nabla\tilde{C}||_{L^2(\cS)}||\pa_z
\tilde{C}||_{L^2(\cS)}\\
&\leq \frac{1}{\mu}C(|\zeta|_{W^{1,\infty}})
||\nabla^\mu {\bf C}||^2_{L^2(\Omega)},
\end{align*}
which implies the result.\\
Finally, if ${\bf C}$ does not vanish at the bottom, then one can
apply the result to $\tilde{\bf C}:=\varphi(X,z){\bf C}$, where
$\varphi\in W^{1,\infty}(\Omega)$ is equal to one in a neighborhood of the surface, and
vanishes at the bottom. This yields the result since $\tilde{\bf
  C}_\surf=\underline{C}$ and $\Abs{\nabla^\mu \tilde{\bf
    C}}_2\lesssim \Abs{\nabla^\mu{\bf C}}_2+\Abs{{\bf C}}_2$.
\end{proof}
We can use
\eqref{mu4} and the lemma to get a control on the $L^2$-norm of $\bU^\mu$,
\begin{align*}
\Abs{\bU^\mu}_2^2&\leq
C(\abs{\zeta}_{W^{2,\infty}})\Big(\mu\Abs{\bom_\mu}_2\Abs{{\bf
    A}}_2+\sqrt{\mu}\abs{\proj\tilde\psi}_{2}\babs{(1\!+\!\sqrt{\mu}\abs{D})^{1/2}(\underline A_h\!+\!\eps\sqrt{\mu}\nabla\zeta
  \underline A_v)}_2\Big)\\
&\leq C(\abs{\zeta}_{W^{2,\infty}},\frac{1}{h_{\rm min}})\big(\mu\Abs{\bom_\mu}_2+\sqrt{\mu}\abs{\proj\tilde\psi}_2\big)\Abs{{\bf
    U^\mu}}_2
\end{align*}
and therefore
\begin{align}
\nonumber
\Abs{\bU^\mu}_2&\leq C(\abs{\zeta}_{W^{2,\infty}},\frac{1}{h_{\rm
    min}})\big(\mu\Abs{\bom_\mu}_2+\sqrt{\mu}\abs{\proj\tilde\psi}_2\big)\\
\label{L2mu}
&\leq {\mu} C(\abs{\zeta}_{W^{2,\infty}},\frac{1}{h_{\rm
    min}}) \big(\Abs{\bom_\mu}_{2}+\frac{1}{\sqrt{\mu}}\abs{\omega_{\mu,b}\cdot N^{\mu}_b}_{H_{0}^{-1/2}}\big),
\end{align}
where we used the fact that $\abs{\proj \tpsi}_2\leq
\abs{\nabla\tpsi}_2$ and the following lemma that makes explicit the
dependence on $\mu$ of the estimate given in Lemma \ref{existtildepsi}.
\begin{lemma}\label{lemtpsimu}
The solution $\tpsi$ to the equation
$\Delta\tpsi=\underline\omega_\mu\cdot N^\mu$ satisfies
$$
\abs{\nabla\tpsi}_2\leq
\sqrt{\mu}C(\abs{\zeta}_{W^{1,\infty}},\frac{1}{h_{\rm
    min}})(\Abs{\bom_\mu}_{2}+\frac{1}{\sqrt{\mu}}\abs{\omega_{\mu,b}\cdot N_{b}^\mu}_{H_{0}^{-1/2}}).
$$
and
$$
\abs{(1+\sqrt{\mu}\abs{D})^{1/2}\nabla\tpsi}_{2}\leq
\sqrt{\mu}C(\abs{\zeta}_{W^{1,\infty}},\frac{1}{h_{\rm
    min}})(\Abs{\bom_\mu}_{2}+\frac{1}{\sqrt{\mu}}\abs{\omega_{\mu,b}\cdot N_b^\mu}_{H_{0}^{-1/2}}).
$$
\end{lemma}
\begin{proof}
Multiplying the equation $\Delta\tpsi=\underline\omega_\mu\cdot N^\mu$
by $\tpsi$, we get as in the proof of Lemma \ref{existtildepsi},
\begin{eqnarray*}
\abs{\nabla\tilde\psi}_{2}^2&=&-\int_{\R^d}\tpsi
\underline{\omega}_\mu\cdot N^\mu\\
&=&-\int_{\R^d}\tpsi^{\rm ext}_b
\omega_{\mu,b}\cdot
N^\mu_b-\int_{\cS}(1+\dz\sigma)\nabla^{\sigma,\mu}\tpsi^{\rm ext}\cdot \omega_\mu.
\end{eqnarray*}
where $\omega_\mu=\bom_\mu\circ\Sigma$ (and $\Sigma$ as in the proof of
Lemma \ref{lemlemmu}) and $\nabla^{\sigma,\mu}$ as in \eqref{nablasigmamu},
while $\tpsi^{\rm ext}$ is now given by $\dsp \tpsi^{\rm
  ext}=\frac{\cosh(\sqrt{\mu}(z+1)\abs{D})}{\cosh(\sqrt{\mu}\abs{D})}\tpsi$.
We deduce that
\begin{align*}
\abs{\nabla\tilde\psi}_{2}^2
&\leq\sqrt{\mu}
\babs{\frac{\abs{D}}{(1+\sqrt{\mu}\abs{D})^{1/2}}\tpsi^{\rm
    ext}_\bottff}_2\,\babs{\frac{(1+\sqrt{\mu}\abs{D})^{1/2}}{\sqrt{\mu}\abs{D}}(\bom_{\mu,b}\cdot
N_b^\mu)}_2\\
&+ C(\abs{\zeta}_{W^{1,\infty}},\frac{1}{h_{\rm
    min}})\Abs{\bom_\mu}_2\Abs{\nabla^\mu \tpsi^{\rm ext}}_2\\
&\leq \sqrt{\mu}C(\abs{\zeta}_{W^{1,\infty}},\frac{1}{h_{\rm
    min}})\big(\Vert
\bom_\mu\Vert_{2}+\frac{1}{\sqrt{\mu}}\abs{\omega_{\mu,b}\cdot N_b^\mu}_{H_{0}^{-1/2}}\big)\abs{\nabla\tpsi}_2,
\end{align*}
which gives the first estimate of the lemma. For the second one, we
multiply the equation by $(1+\sqrt{\mu}\abs{D})\psi$ and proceed
as above to get
\begin{align*}
\babs{(1+&\sqrt{\mu}\abs{D})^{1/2}\nabla\tilde\psi}^{2}_{2}
\leq C(\abs{\zeta}_{W^{1,\infty}},\frac{1}{h_{\rm
    min}})\Abs{\bom_\mu}_2\Abs{(1+\sqrt{\mu}\abs{D})\nabla^\mu \tpsi^{\rm ext}}_2\\
&+
\sqrt{\mu}\Babs{(1+\sqrt{\mu}\abs{D})^{1/2}\abs{D}\tpsi^{\rm
    ext}_b}_2\,\Babs{\frac{(1+\sqrt{\mu}\abs{D})^{1/2}}{\sqrt{\mu}\abs{D}}(\bom_{\mu,b}\cdot
N_b^\mu)}_2\\
&\leq \sqrt{\mu}C(\abs{\zeta}_{W^{1,\infty}},\frac{1}{h_{\rm
    min}})\big(\Vert
\bom_\mu\Vert_{2}+\abs{\omega_{\mu,b}\cdot N_b^\mu}_{H_{0}^{-1/2}}\big)\babs{(1+\sqrt{\mu}\abs{D})^{1/2}\nabla\tpsi}_{2},
\end{align*}
(in both terms of the right-hand-side in the first inequality, a
smoothing argument must be used to gain half-a-derivative; the
dependence of this smoothing on $\mu$ is crucial here; it is of the
form $(1+\sqrt{\mu}\abs{D})^{-1/2}$ -- see for instance Lemma 2.20
in \cite{Lannes_book}). The result follows directly.
\end{proof}
Similarly, the dependence on $\mu$ of the $H^1$-estimate of Lemma
\ref{lemlem} must be made precise. For the energy estimates on the
vorticity, it shall be crucial that no $1/\sqrt{\mu}$ singularity
appears in front of the
bottom vorticity term of this $H^1$-estimate. For the
$L^2$-estimate \eqref{L2mu}, this singularity comes from the control
of $\abs{\proj\tpsi}_{2}$; for the $H^1$-estimate this term is
expected to be replaced by $\abs{\proj\tpsi}_{H^1}$, which is also
singular. However, it turns out that a control in terms of
$\abs{\proj\nabla\tpsi}_{2}$ is enough, and that for this term, low
frequencies are damped, and the $1/\sqrt{\mu}$ singularity can be
removed. This is done in the following lemma.
\begin{lemma}\label{lemmasave}
The solution $\tpsi$ to the equation $\Delta\tpsi=\uom_\mu\cdot
N^\mu$ satisfies
$$
\abs{\proj\nabla\tpsi}_2\leq C(\frac{1}{h_{\rm min}},\abs{\zeta}_{W^{1,\infty}})\Abs{\bom_\mu}_2.
$$
\end{lemma}
\begin{proof}
Proceed as in the proof of Lemma \ref{lemtpsimu} to obtain
\begin{align*}
\abs{\proj\nabla\tpsi}_2^2&=-\int_{\R^d}\Big(\frac{D^2}{1+\sqrt{\mu}\abs{D}}\tpsi\Big)\uom_\mu\cdot
N\\
&=-\int_{\cS}(1+\dz\sigma)\nabla^{\sigma,\mu}
\big(\frac{D^2}{1+\sqrt{\mu}\abs{D}}\tpsi^{\rm ext}_0\big)\cdot \omega_\mu,
\end{align*}
where we have chose a different extension of $\tpsi$ than in the proof
of Lemma \ref{lemtpsimu}, namely,
$$
\tpsi_0^{\rm ext}=\frac{\sinh(\sqrt{\mu}(z+1)\abs{D})}{\sinh(\sqrt{\mu}\abs{D})}\tpsi;
$$
in particular, $\tpsi_0^{\rm ext}$ vanishes at the bottom, and this is
the reason why no bottom boundary term appears in the expression
above. One readily deduces that
\begin{equation}\label{save0}
\abs{\proj\nabla\tpsi}_2^2\leq
C(\abs{\zeta}_{W^{1,\infty}},\frac{1}{h_{\rm min}})\Abs{\nabla^{\sigma,\mu}\Big(\frac{D^2}{1+\sqrt{\mu}\abs{D}}\tpsi^{\rm ext}_0\Big)}_2\Abs{\bom_\mu}_2,
\end{equation}
and we therefore need to control
$\Abs{\nabla^{\sigma,\mu}\Big(\frac{D^2}{1+\sqrt{\mu}\abs{D}}\tpsi^{\rm
    ext}_0\Big)}_2$. We distinguish the horizontal and vertical
derivatives involved in $\nabla^{\sigma,\mu}$; we start with the
vertical derivative, which is the the most delicate:\\
- {\it Control of $\Abs{\frac{D^2}{1+\sqrt{\mu}\abs{D}}\dz\tpsi^{\rm
    ext}_0}_2$.} One has
\begin{align*}
\bAbs{\frac{D^2}{1+\sqrt{\mu}\abs{D}}\dz\tpsi^{\rm
    ext}_0}_2^2&= \bAbs{D^2 \frac{\sqrt{\mu}
    \abs{D}}{1+\sqrt{\mu}\abs{D}}\frac{\cosh(\sqrt{\mu}(z+1)\abs{D})}{\sinh(\sqrt{\mu}\abs{D})}\tpsi}_2^2\\
&=\int_{\R^d}\int_{-1}^0\babs{\sqrt{\mu}\abs{\xi}\frac{\cosh(\sqrt{\mu}(z+1)\abs{\xi})}{\sinh(\sqrt{\mu}\abs{\xi})}}^2\babs{\widehat{\proj\nabla\tpsi}}^2
dz d\xi.
\end{align*}
Let $F(r)=\frac{1}{2}r+\frac{1}{4}\sinh(2r)$ ($F$ is the
primitive of $\cosh(r)^2$ vanishing at zero); integrating with respect
to $z$ in the above expression, we get
\begin{align*}
\bAbs{\frac{D^2}{1+\sqrt{\mu}\abs{D}}\dz\tpsi^{\rm
    ext}_0}_2&=
\int_{\R^d}\frac{F(\sqrt{\mu}\abs{\xi})}{\sinh(\sqrt{\mu}\abs{\xi})^2}\babs{\widehat{\proj\nabla\tpsi}}^2
dz d\xi.
\end{align*}
Since $\frac{F(\sqrt{\mu}\abs{\xi})}{\sinh(\sqrt{\mu}\abs{\xi})^2}$ is
uniformly bounded from above (with respect to $\xi$ and $\mu$), we
deduce finally from Plancherel's identity that
$$
\Abs{\frac{D^2}{1+\sqrt{\mu}\abs{D}}\dz\tpsi^{\rm
    ext}_0}_2\lesssim \abs{\proj\nabla\tpsi}_2.
$$
- {\it Control of $\Abs{\frac{D^2}{1+\sqrt{\mu}\abs{D}}\sqrt{\mu}\nabla\tpsi^{\rm
    ext}_0}_2$.} One has
\begin{align*}
\bAbs{\frac{D^2}{1+\sqrt{\mu}\abs{D}}\sqrt{\mu}\nabla\tpsi^{\rm
    ext}_0}_2&\leq \Abs{D^2 \tpsi_0^{\rm ext}}_2\\
&\leq \babs{\frac{D^2}{(1+\sqrt{\mu}\abs{D})^{1/2}}\tpsi}_2=\abs{\proj \nabla\tpsi}_2,
\end{align*}
the second inequality stemming from a smoothing argument similar to
the one used to control the vertical derivative.
\medbreak
Together with \eqref{save0}, these two controls give the result of the lemma.
\end{proof}
We can therefore provide a control of $\nabla^\mu \bU^\mu$ without the
$1/\sqrt{\mu}$ singularity in front of the bottom vorticity component.
\begin{lemma}\label{lemlemmu2}
The following estimate holds,
$$
\Abs{\nabla^\mu \bU^\mu}_2\leq \mu
C(\abs{\zeta}_{W^{2,\infty}},\frac{1}{h_{\rm min}})
\big(\Abs{\bom_\mu}_{2}+\abs{\omega_{\mu,b}\cdot N_b^\mu}_{H_{0}^{-1/2}}\big).
$$
\end{lemma}
\begin{proof}
The proof is similar to that one of Lemma \ref{lemlem}. Here we
present the main differences. The dimensionless version of
\eqref{noel1} is
\begin{align}
\nonumber
\int_{\Om}\abs{\nabla^\mu \bU^\mu}^2&=\mu^2\int_{\Om}\abs{\omega_\mu}^2
+2\mu\int_{\R^d} \uV\cdot\nabla\uw-\mu^2\eps\int_{\R^d}\left(\nabla^\perp\zeta\cdot\nabla\right)
\uV^\perp\cdot \uV\\
\label{sarde}
&:= \mu^2\int_{\Om}\abs{\omega_\mu}^2+\mu I_1+\mu^2 I_2.
\end{align}
Now we proceed to bound the integral $I_1$ and $I_2$. We recall that $\uV=\nabla^\perp\widetilde{\psi}-\eps\uw\nabla\zeta$ and we can write
\begin{align*}
I_1
&\leq 2\eps\int_{\R^d}\babs{\nabla \zeta\cdot \nabla\uw \uw}\\
&\leq
C(|\zeta|_{W^{2,\infty}}) \Abs{\nabla^\mu \bU^\mu}_2\Abs{\bU^\mu}_2,
\end{align*}
where we used the fact that, since $w_b=0$, one has
$|\uw|^2_{2}\leq ||{\bf w}||_{2}||\pa_z{\bf w}||_{2}$. Together with
\eqref{L2mu}, this yields
\begin{align*}
I_1&\leq C(|\zeta|_{W^{2,\infty}},\frac{1}{h_{\rm
    min}})\big(\mu\Abs{\bom_\mu}_2+\sqrt{\mu}\abs{\proj\tilde\psi}_2\big)\Abs{\nabla^\mu
  \bU^\mu}_2\\
&\leq \sqrt{\mu} C(|\zeta|_{W^{2,\infty}},\frac{1}{h_{\rm
    min}})\big(\sqrt{\mu}\Abs{\bom_\mu}_2+\abs{\omega_{\mu,b}\cdot
N_b}_{H_{0}^{-1/2}}\big)\Abs{\nabla^\mu
  \bU^\mu}_2,
\end{align*}
the last inequality stemming from the observation that
$\abs{\proj\tpsi}_2\leq \abs{\nabla\tpsi}_2$ and Lemma \ref{lemtpsimu}.\\
For $I_2$, after
substituting $\uV=\nabla^\perp\widetilde{\psi}-\eps\uw\nabla\zeta$,
the integrand can be written as a sum of terms of the form
$$
C(\zeta)\partial^2\tpsi\partial\tpsi, \quad
C(\zeta) w\pa^2\tpsi, \quad C(\zeta)\pa
{\uw}\pa\tpsi, \quad
C(\zeta){\uw}\pa\tpsi, \quad
C(\zeta){\uw}^2,
$$
where $C(\zeta)$ stand for any polynomial expression in the first and
second order derivatives of $\zeta$, while $\pa$ and $\pa^2$ stand
here for any first and second order partial derivative
respectively. We consequently get
$$
I_2\leq
C(\abs{\zeta}_{W^{2,\infty}})\big[\abs{(1+\sqrt{\mu}\abs{D})^{1/2}\nabla\tpsi}_2(\abs{\proj\nabla\tpsi}_2+\abs{\proj
  {\uw}}_2+\abs{{\uw}}_2)+\abs{\underline{w}}_2^2\big].
$$
Remarking that $\abs{\proj f}_2\leq
\mu^{-1/2}\abs{(1+\sqrt{\mu}\abs{D})^{1/2}f}_2$, recalling that
$\abs{\underline{\bf w}}_2^2\leq \Abs{{\bf w}}_2\Abs{\dz{\bf w}}_2$, and with the help of
\eqref{L2mu} and Lemmas \ref{lemlemmu}, \ref{lemtpsimu} and \ref{lemmasave}, this yields
\begin{align*}
I_2\leq&
\sqrt{\mu}C(\abs{\zeta}_{W^{2,\infty}},\frac{1}{h_{\rm
    min}})\big(\sqrt{\mu}\Abs{\bom_\mu}_2+\frac{1}{\sqrt{\mu}}\abs{\omega_{\mu,b}\cdot
  N_b}_{H_{0}^{-1/2}}\big) \Abs{\bom_\mu}_2\\
&+
C(\abs{\zeta}_{W^{2,\infty}},\frac{1}{h_{\rm
    min}})\big(\Abs{\bom_\mu}_2+\frac{1}{\sqrt{\mu}}\abs{\omega_{\mu,b}\cdot
  N_b}_{H_{0}^{-1/2}}\big)\Abs{\nabla^\mu {\bf U}^\mu}_2.
\end{align*}
Gathering the estimates on $I_1$ and $I_2$, one readily gets
\begin{align*}
I_1+\mu I_2&\leq
\sqrt{\mu} C(\abs{\zeta}_{W^{2,\infty}},\frac{1}{h_{\rm min}})\big(
\sqrt{\mu}\Abs{\bom_\mu}_2+\abs{\omega_{\mu,b}\cdot
  N_b}_{H_{0}^{-1/2}}\big)\Abs{\nabla^\mu {\bf
  U}^\mu}_2\\
&+{\mu}C(\abs{\zeta}_{W^{2,\infty}},\frac{1}{h_{\rm min}})\big(
\sqrt{\mu}\Abs{\bom_\mu}_2+\abs{\omega_{\mu,b}\cdot
  N_b}_{H_{0}^{-1/2}}\big) \Abs{\bom_\mu}_2;
\end{align*}
with
\eqref{sarde}, this yields the result.
\end{proof}
The result of the proposition directly follows from \eqref{L2mu},
Lemma \ref{lemlemmu2} (and the aforementioned estimates on the
irrotational part).
\end{proof}

For higher order regularity estimates, we use as in \S
\ref{sectstraight} the straightened version of the velocity and
vorticity introduced in \S \ref{statMRD}.
The proposition below shows how the higher order estimate of
Corollary \ref{corHkk} depends on $\mu$.
\begin{proposition}\label{horizontalregularitymu} Let $N\in {\mathbb
    N}$, $N\geq 5$. Then for all $0\leq l\leq k\leq N-1$, the straightened
  velocity $U^\mu=\mathbb{U}^{\sigma,\mu}[\ep\zeta](\psi,\omega_\mu)$ satisfies the estimate
\begin{align*}
\Abs{\nabla^\mu U^\mu}_{H^{k,l}}\leq \mu M_N \Big(\abs{\proj\psi}_{H^1}+\hspace{-2mm}\sum_{1<\abs{\alpha}\leq k+1}\hspace{-1mm}\abs{\proj
  \psi_{(\alpha)}}_2+\Abs{\omega_\mu}_{H^{k,l}}+\abs{\Lambda^k\omega_{\mu,b}\cdot N_b}_{H_{0}^{-1/2}}\Big),
\end{align*}
where $\psi_{(\alpha)}:=\pa^\alpha\psi-\eps\uw\pa^\alpha\zeta$ (and $\underline{w}=\mathbb{w}^{\sigma}[\ep\zeta](\psi,\omega_\mu)_\surff$).
\end{proposition}
\begin{proof}
The proof is based on the dimensionless version of \eqref{transfineq}
which reads
\begin{align}\label{weakmubanda}
\int_{S}\nabla^\mu U^\mu \cdot P^\mu(\Sigma)\nabla^\mu U^\mu =&\mu\int_{\cS}(1+\sigma_z)\omega_\mu\cdot
\nabla^{\sigma,\mu}\times C +\int_{\R^d}F^\mu\cdot \underline{C},
\end{align}
for all $C\in H^1(S)$, and with $\nabla^{\sigma,\mu}$ as in
\eqref{nablasigmamu} while
\begin{align*}
P^\mu(\Sigma)=(1+\sigma_z)\left(J^\mu_{\Sigma}\right)^{-1}\left(J^\mu_{\Sigma}\right)^{-1\,\,T},\quad\mbox{
with }\quad
\left(J^\mu_{\Sigma}\right)^{-1\,\,T}=\left(\begin{array}{ccc} 1 & 0 & \frac{-\sqrt{\mu\sigma_x}}{1+\sigma_z}\\
0 & 1 & \frac{-\sqrt{\mu}\sigma_y}{1+\sigma_z}\\ 0 & 0 & \frac{1}{1+\sigma_z},\end{array}\right)
\end{align*}
and
\begin{align*}
F^\mu&=\left(N^\mu\times\nabla^\mu \times \bU^\mu
+(N^\mu\cdot\nabla^\mu)\bU^\mu\right)_\surf\\
&=(\sqrt{\mu}\nabla\uw-\mu^\frac{3}{2}(\nabla\zeta^\perp\cdot\nabla)\uV^\perp
,-\mu\nabla\cdot \uV).
\end{align*}
The key point is that the matrix $P^\mu(\Sigma)$ is uniformly coercive
with respect to $\mu$ so that the same structure as for Proposition
\ref{horizontalregurality} can be used for the proof. In particular,
\eqref{eqI0} is replaced by
\begin{align*}
\Abs{\nabla^\mu\pa^\beta U^\mu}_2^2\leq M_N(I_1+I_2+I_3),
\end{align*}
with the following definition and upper bounds on $I_i$ , $i=1,2,3$:\\
- \emph{Upper bound for $I_1$}. Proceeding exactly as for \eqref{eqI1} we get
\begin{align*}
I_1&=\int_{\cS}\nabla^\mu \partial^\beta
U\cdot[\partial^\beta,P^\mu(\Sigma)]\nabla^\mu U\\
&\leq M_N ||\nabla^\mu \pa^\beta U^\mu||_{2}||\Lambda^{k-1}\nabla^\mu U^\mu||_{2}.
\end{align*}
- \emph{Upper bound for $I_2$}. With straightforward adaptations,
  we get as for Proposition \ref{horizontalregurality} that
\begin{align*}
I_2&=\int_{S}\Lambda^k\omega_\mu \cdot \Lambda^{-k}\nabla^{\sigma,\mu}\times\pa^{2\beta}U^\mu\\
&\leq \mu M_N||\Lambda^k\omega_\mu||_{2}||\pa^\beta\nabla^\mu U^\mu||_{2}.
\end{align*}
- \emph{Upper bound for $I_3$} We split $I_3$ into three terms
\begin{align*}
I_3=&2\mu\int_{\R^d}\pa^\beta\nabla\uw\cdot\pa^\beta\uV-
\mu^2\eps\int_{\R^d}(\nabla^\perp\zeta\cdot\nabla)\pa^\beta\uV^\perp\cdot\pa^\beta\uV\\
& -\mu^2\eps\int_{\R^d}[\pa^\beta,\nabla^\perp\zeta]\cdot\nabla\uV^\perp\cdot\pa^\beta\uV\\
=& I_{31}+I_{32}+I_{33}.
\end{align*}
Replacing the product estimate \eqref{prodP} by
$$
\forall f,g\in H^{1/2}(\R^d),\qquad \int_{\R^d} f\partial_j g \leq
\abs{\proj f}_2\babs{(1+\sqrt{\mu}\abs{D})^{1/2}g}_{2} \quad (1\leq j\leq d),
$$
and using the $\mu$-dependent version of the trace lemma furnished by
the second point of Lemma \ref{lemlemmu}, the upper bound on $I_{31}$
given in the proof of Proposition \ref{horizontalregurality} can be adapted into
$$
{I_{31}}\leq M_N\big(\mu\abs{\proj\big(\nabla \partial^\beta\psi-\eps
  \uw\nabla\partial^\beta \zeta\big)}_2+\sqrt{\mu}\Abs{\Lambda^{k-1}\nabla^\mu U^\mu}_{2}\big)\Abs{\Lambda^k
\nabla^\mu U^\mu}_{2}.
$$
For $I_{32}$, we proceed as for $I_{31}$ to get
\begin{align*}
I_{32}\leq \mu^2 M_N\Big[&\big(\abs{\proj\big(\nabla \partial^\beta\psi-\eps
  \uw\nabla\partial^\beta
  \zeta\big)}_2+\abs{\proj\partial^\beta\tpsi}_2+\abs{\proj
  \Lambda^{k-1}\underline{w}}_2\big)\\
&\times
\frac{1}{\sqrt{\mu}}\abs{(1+\sqrt{\mu}\abs{D})^{1/2}\pa^\beta\underline{ U}^\mu}_2+\abs{\partial^\beta \underline{w}}_2^2\Big];
\end{align*}
with Lemma \ref{lemlemmu} and the inequality
$\abs{\pa^\beta\underline{w}}_2^2\leq \Abs{\pa^\beta w}_2\Abs{\pa^\beta
\dz w}_2$, we therefore get
\begin{align*}
I_{32}\leq  M_N &\big(\mu\abs{\proj\big(\nabla \partial^\beta\psi-\eps
  \uw\nabla\partial^\beta
  \zeta\big)}_2+\mu\abs{\proj\partial^\beta\tpsi}_2+{\sqrt{\mu}}\Abs{
  \Lambda^{k-1}\nabla^\mu{ w}}_2\big)\Abs{\Lambda^k\nabla^\mu{U}^\mu}_2.
\end{align*}
Finally, one readily gets that
\begin{align*}
I_{33}&\leq M_N \Abs{\Lambda^{k-1}\nabla^\mu U^\mu}_2\Abs{\Lambda^k
  \nabla^\mu U^\mu}_2.
\end{align*}
Summing up the upper bounds on $I_{31}$, $I_{32}$ and $I_{33}$, we
finally get
$$
I_{3}\leq  M_N \Big(\mu\abs{\proj\big(\nabla \partial^\beta\psi-\eps
  \uw\nabla\partial^\beta
  \zeta\big)}_2+\mu\abs{\proj\partial^\beta\tpsi}_2+\Abs{
  \Lambda^{k-1}\nabla^\mu{
    U^\mu}}_2\Big)\Abs{\Lambda^k\nabla^\mu{U}^\mu}_2.
$$
\medbreak
As a result of these upper bounds on $I_1$, $I_2$ and $I_3$, we obtain
\begin{align*}
\Abs{\nabla^\mu\partial^\beta U^\mu}_2^2\leq M_N \Big(&\mu\abs{\proj\big(\nabla \partial^\beta\psi-\eps
  \uw\nabla\partial^\beta
  \zeta\big)}_2+\mu\abs{\proj\partial^\beta\tpsi}_2+\Abs{
  \Lambda^{k-1}\nabla^\mu{
    U^\mu}}_2\\
&+\mu\Abs{\Lambda^k \omega_\mu}_2\Big)\times \Abs{\Lambda^k\nabla^\mu{U}^\mu}_2.
\end{align*}
with $k=|\beta|$.
With the same induction method as in the proof of
Proposition \ref{horizontalregurality},  one then deduces
$$
\Abs{\Lambda^k\nabla^\mu U^\mu}_2\leq M_N
\big(\mu\sum_{1<\abs{\alpha}\leq k+1}\abs{\proj \psi_{(\alpha)}}_2+\mu\abs{\proj\partial^\beta\tpsi}_2 +\mu\Abs{\Lambda^k \omega_\mu}_2+\Abs{
  \nabla^\mu{
    U^\mu}}_2\big);
$$
with the estimate on
$\nabla^\mu U^\mu$ provided by Proposition \ref{prop1mu}, and the following
estimate that generalizes Lemma \ref{lemtpsimu} in the spirit of Lemma \ref{estimtpsi}
$$
\abs{\proj\partial^\beta\tpsi}\leq \abs{\Lambda^k \nabla\tpsi}_2\leq\sqrt{\mu}
C(\abs{\zeta}_{W^{1,\infty}},\frac{1}{h_{\rm
    min}})(\Abs{\Lambda^k\omega_\mu}_{2}+\frac{1}{\sqrt{\mu}}\abs{\omega_{\mu,b}\cdot N_\mu}_{H_{0}^{-1/2}}),
$$
we finally get a dimensionless version of the upper bound
of Proposition \ref{horizontalregurality}
\begin{align*}
\Abs{\Lambda^k\nabla^\mu U^\mu}_2\leq \mu M_N \Big(&\abs{\proj\psi}_{H^1}+\hspace{-3mm}\sum_{1<\abs{\alpha}\leq k+1}\hspace{-2mm}\abs{\proj
  \psi_{(\alpha)}}_2+\Abs{\Lambda^k \omega_\mu}_2+\abs{\Lambda^k\omega_{\mu,b}\cdot N_b}_{H_{0}^{-1/2}}\Big).
\end{align*}
Following the same steps as in the proof of Corollary \ref{corHkk}, we
deduce an $H^{k,l}$ estimate of $\nabla^\mu U^\mu$ from this $L^2$
estimate of $\Lambda^k \nabla^\mu U^\mu$.
\end{proof}

The adaptation to the dimensionless case of the other results
presented in \S \ref{sectproofpropshape} and \S \ref{sectAIG} are
then straightforward and we therefore omit them.
\subsubsection{A priori estimates for the vorticity}\label{sectVED}

 We prove here that the
estimates of Proposition \ref{EEvorticity} can be replaced in the
dimensionless case by
\begin{align}
\label{estvortmu}
\frac{d}{dt} \Big(\Abs{\omega_\mu}_{H^k}^2+\abs{\omega_{\mu,b}\cdot
  N_b}_{H_{0}^{-1/2}}\Big)\leq \eps  \mfN(\zeta,\psi,\omega).
\end{align}
We proceed as in the proof of Proposition
\ref{EEvorticity}, to obtain the following dimensionless version of the
vorticity equation \eqref{eqqq}
$$
 \dt\omega_\mu+\eps{\mathbb V}^{\sigma}[\eps\zeta](\psi,\omega_\mu)\cdot\nabla
 \omega+\frac{\eps}{\mu}{\mathbb a}[\eps\zeta](\psi,\omega_\mu)\dz\omega_\mu=\frac{\eps}{\mu}\omega_\mu\cdot\nabla^{\sigma,\mu} {\mathbb U}^{\sigma,\mu}[\eps\zeta](\psi,\omega_\mu)
$$
with, denoting $\tilde N^\mu=(-\sqrt{\mu}\nabla{\sigma}^T,1)^T$ (so
that $\tilde N^\mu_{\vert_{z=0}}=N^\mu$),
\begin{align*}
{\mathbb a}[\eps\zeta](\psi,\omega_\mu)&=\frac{1}{1+\dz\sigma}\Big( {\mathbb
   U}^{\sigma,\mu}[\eps\zeta](\psi,\omega_\mu)\cdot \tilde N^\mu-\mu\dt \sigma\Big)\\
&=\frac{1}{1+\dz\sigma}\Big( {\mathbb
   U}^{\sigma,\mu}[\eps\zeta](\psi,\omega_\mu)\cdot \tilde N^\mu-\frac{z+H_0}{H_0}
 \underline{\mathbb
   U}^{\sigma,\mu}[\eps\zeta](\psi,\omega_\mu)\cdot N^\mu\Big),
\end{align*}
and we are led to study the following equation instead of \eqref{vortL2}
$$
\dt \omega_\mu+\eps V\cdot \nabla  \omega +\frac{\eps}{\mu}a \dz
\omega_\mu=\eps f,
$$
(with $V={\mathbb V}^\sigma[\eps\zeta](\psi,\omega_\mu)$ and $a={\mathbb
  a}[\eps\zeta](\psi,\omega_\mu)$). The $L^2$-estimate
\eqref{estvortL2} must be refined to get a good dependence on $\mu$. Taking the
$L^2$ scalar product of this equation with
$\omega_\mu$, we get
\begin{eqnarray*}
\frac{1}{2}\dt \Abs{\omega_\mu}_{L^2}^2-\frac{\eps}{2}\int_\cS
(\nabla  \cdot V+\frac{1}{\mu}\dz a)\abs{\omega_\mu}^2
=\eps\int_\cS  f \cdot \omega_\mu,
\end{eqnarray*}
and therefore
$$
\dt \Abs{\omega_\mu}_{2}^2
\lesssim\frac{\eps}{\mu}\big(\Abs{\nabla^\mu U^\mu}_{\infty}+\sqrt{\mu}\Abs{U^\mu}_\infty\big)\Abs{\omega_\mu}_{2}^2+\eps\Abs{\omega_\mu}_2\Abs{f}_{2}.
$$
From the continuous embedding $H^{N-1}(\cS)\subset L^\infty(\cS)$ and
the fact that $\Vert U^\mu \Vert_{H^{N-1}}\leq \Vert U^\mu \Vert_2
+\frac{1}{\sqrt{\mu}}\Vert \nabla^\mu U^\mu \Vert_{H^{N-2}}$, we
deduce that
\begin{eqnarray*}
\dt \Abs{\omega_\mu}_{2}^2
&\lesssim& \frac{\eps}{\mu}\big(\Abs{\nabla^\mu
  U^\mu}_{H^{N-1}}+\sqrt{\mu}\Abs{U^\mu}_2\big)\Abs{\omega_\mu}_{2}^2+\Abs{\omega_\mu}_2\Abs{f}_{2}\\
&\lesssim & \eps \mfN(\zeta,\psi,\bom) \Abs{\omega_\mu}_{2}^2+\Abs{\omega_\mu}_2\Abs{f}_{2}.,
\end{eqnarray*}
the second inequality stemming from Propositions \ref{prop1mu} and
\ref{horizontalregularitymu}. Using this generalization of
\eqref{estvortL2}, we obtain the estimate on the $H^k$-norm of
$\omega_\mu$ of\eqref{estvortmu}. For the bottom vorticity, we have to
replace \eqref{estbotvort} by
$$
\dt (\omega_{\mu,b}\cdot N_b)+\eps\nabla\cdot (\omega_{\mu,b}\cdot N_b V_b)=0,
$$
and therefore
\begin{align*}
\dt\abs{\omega_{\mu,b}\cdot N_b}_{H^{-1/2}_0}&\leq \eps\babs{(1+\sqrt{\mu}\abs{D})^{1/2}\big((\omega_{\mu,b}\cdot
  N_b)V_b\big)}_{2}\\
&\leq \eps \mfN(\zeta,\psi,\omega_\mu),
\end{align*}
the last inequality stemming from the trace lemma, standard product
estimates, and Proposition \ref{horizontalregularitymu}. This
completes the proof of \eqref{estvortmu}.

\subsection{A priori estimates on the full equations}\label{sectAPD}

In dimensionless variable, the good unknown becomes
$$
\forall \alpha\in \N^d\backslash\{0\}, \qquad
U^\mu_{(\alpha)}=\partial^\alpha U^\mu-\partial^\alpha\sigma
\dz^\sigma U^\mu.
$$
With a straightforward adaptation,  and using the div-curl estimates derived in \S \ref{DCpar}, the quasilinear structure exhibited
in Proposition \ref{prop3} takes the following form in dimensionless variables,
\begin{eqnarray*}
(\dt +\eps \uV\cdot\nabla)
\partial^\alpha\zeta-\frac{1}{\mu}\partial_k\uU^\mu_{(\beta)}\cdot
N^\mu&=&\eps R^1_\alpha,\\
(\dt +\eps \uV\cdot \nabla) (U^\mu_{(\beta)\parallel}\cdot {\bf e}_k)+\mfa \partial^\alpha
\zeta&=&\eps R^2_\alpha,\\
(\dt^\sigma+\frac{\eps}{\mu}U^\mu\cdot\nabla^{\sigma,\mu})\partial^\beta \omega_\mu&=&\eps R^3_\beta,
\end{eqnarray*}
with $\mfa=1+(\dt+\eps \uV\cdot\nabla)\uw$ and where
$$
\abs{R^1_\alpha}_2+\abs{\proj R^2_\alpha}_2+\Abs{R^3_\beta}_2\leq \mfN(\zeta,\psi,\omega_\mu).
$$
The following dimensionless version of \eqref{esten3} can then be derived
along the same lines as in the dimensional case,
$$
\frac{1}{2}\dt
(\mfa \partial^\alpha\zeta,\partial^\alpha\zeta)+
\big((\dt+\eps\uV\cdot\nabla) (\Vpb^\mu\cdot {\bf
  e}_k),\frac{1}{\mu}\partial_k\uU^\mu_{(\beta)}\cdot N^\mu\big)\leq
\eps \mfN(\zeta,\psi,\omega_\mu),
$$
from which, proceed as for \eqref{EEmain}, we get
$$
\dt\left\lbrace
(\mfa \partial^\alpha\zeta,\partial^\alpha\zeta)+ \int_\cS(1+\dz\sigma)\abs{\partial_k U^\mu_{(\beta)}}_2^2\right\rbrace
\leq \eps \mfN(\zeta,\psi,\omega_\mu).
$$
Together with the vorticity estimate \eqref{estvortmu}, and mimicking
Steps 4 and 5 of the proof of Proposition \ref{prop3000}, we get that
for all $0\leq t\leq T/\eps$, one has
$$
\cE^N(\zeta,\psi,\bom_\mu)(t)\leq C(T,\frac{1}{a_0},\frac{1}{h_{\rm min}},\cE^N(\zeta^0,\psi^0,\omega_\mu^0)).
$$
By a classical prolongation argument, this allows one to extend the
solution provided by Theorem \ref{theo1} on a time interval
$[0,T/\eps]$, with $T$ independent of $\eps$ and $\mu$, thus
completing the proof of Theorem \ref{theomainND}.

\subsection{Justification of the shallow water equations with
  vorticity}\label{sectSWvort}

Shallow water models provide simplified models for the propagation of
water waves when $\mu\ll 1$. In the irrotational case, they are
typically stated as a set of equations coupling the evolution of the
surface elevation $\zeta$ to the vertically averaged horizontal
velocity $\ovV$,
$$
\ovV(t,X,z)=\frac{1}{h(t,X)}\int_{-1}^{\eps \zeta(t,X)} {\bf
  V}(t,X,z)dz\qquad (h=1+\eps\zeta).
$$
Various models exist, depending on the precision of the
approximation, and possible smallness assumptions on $\eps$. The
derivation and justification of these shallow water models is now well
understood; we refer to
\cite{Lannes_book} for references and a detailed description of the many shallow
water models and for their rigorous justification. In the so called
{\it shallow water, large amplitude} regime corresponding to
$$
\eps=1,\qquad \mu\ll 1,
$$
one obtains for instance at first order (i.e. up to $O(\mu)$ terms)
the well known Nonlinear Shallow Water (or Saint-Venant) equations
\begin{equation}\label{NSW}
\left\lbrace
\begin{array}{l}
\dt \zeta+\nabla\cdot (h\ovV)=0,\\
\dt \ovV+\ovV\cdot \nabla\ovV+\nabla\zeta=0.
\end{array}\right.
\end{equation}
A byproduct of the derivation and justification of this model is that
the horizontal velocity ${\bf V}$ is, at the precision of the model,
independent of the vertical variable $z$ (in the physics literature,
this is is often an assumption, called ``columnar motion''
assumption).
A consequence is that at the precision of the model, the velocity at
the surface $\uV$ is equal to the averaged velocity $\ovV$ for which
\eqref{NSW} is derived. This is of importance since direct
experimental data are more accessible for $\uV$ (using buoys for
instance) than for $\ovV$.

In the rotational case, the picture is less clear and there does not
exist any fully justified shallow water model. Even at the formal
level, there is no real consensus in the physics literature. The
assumption of ``columnar motion'' is often made to derive asymptotic
models, but, as shown below, it is in general wrong at the precision
of the model. As shown in \cite{CastroLannes} by the authors, a consequence of this fact is that even though the
Nonlinear Shallow Water model \eqref{NSW} remains the
same\footnote{This statement is not obvious and deserves a proof!}
system \eqref{NSW} in presence of
vorticity when written in $(\zeta,\ovV)$ variables, the recovery of
the velocity $\uV$ at the surface requires the resolution of one more
equation,
\begin{equation}\label{NSWcor}
\uV=\ovV-\sqrt{\mu} Q,
\quad\mbox{ with }\quad
\dt Q+\ovV\cdot \nabla Q+Q\cdot \nabla \ovV=0,
\end{equation}
The goal of this section is to show that Theorem \ref{theomainND}
provides all the necessary bounds on the solution to justify the
formal computations of \cite{CastroLannes}, and therefore to bring a
full justification of the Nonlinear Shallow Water
model\footnote{Other, more precise, models are derived in
  \cite{CastroLannes}, showing for instance the creation of horizontal
vorticity from an initially purely vertical vorticity. A good
local well-posedness theory for these models is the only thing to prove
in order to deduce a full justification along the procedure described here for
the NSW model \eqref{NSW}-\eqref{NSWcor}.} \eqref{NSW}-\eqref{NSWcor} as a model for the description of
shallow water waves in presence of vorticity.

\subsubsection{Derivation of the model}\label{sectderivSW}

For the sake of completeness, we sketch here the derivation of the NSW
model \eqref{NSW}-\eqref{NSWcor}. We refer to \cite{CastroLannes} for
details. For clarity, we also use the notation
$$
f=O(\mu^\alpha) \iff \exists k\geq 0,\forall n\geq 0, \quad \abs{f}_{H^n}\leq \mu^\alpha C({\mathcal
E}^{n+k}(\zeta,\psi,\omega_\mu))
$$
for function defined on $\R^d$, and with obvious adaptation for
functions defined on $\Omega$. We recall that the energy ${\mathcal
E}^{n+k}(\zeta,\psi,\omega_\mu)$ is controlled uniformly with respect
to $\mu$ by Theorem \ref{theomainND}.

Let us consider therefore $(\zeta,\psi,\omega_\mu)$, the solution provided by Theorem
\ref{theomainND}. We denote $\bom_\mu=\omega_\mu\circ \Sigma^{-1}$ the
corresponding vorticity in the fluid domain. One can show that the
velocity field $\bU^\mu$ has the following structure,
\begin{equation}\label{structU}
\bU^\mu
=\left(\begin{array}{c}
\sqrt{\mu}{\bf V}\\
{\bf w}
\end{array}\right)=
\left(\begin{array}{c}
\sqrt{\mu}\ovV+\mu \big(\int_z^\zeta (\bom_\mu)_h^\perp-Q\big)+O(\mu^{3/2})\\
-\mu(1+z)\Delta\psi-\mu^{3/2}\nabla\cdot \int_{-1}^z \int_{z'}^\zeta (\bom_\mu)_h^\perp -\mu^2\int_{-1}^z
\nabla\cdot V^{(1)}
\end{array}\right),
\end{equation}
with
$$
Q:=\frac{1}{h}\int_{-1}^\zeta\int_{z'}^\zeta ({\bom}_\mu)_h^\perp.
$$
Plugging this expression into the vorticity equation \eqref{eqrotND},
we obtain, for the horizontal components,
$$
\dt (\bom_\mu)_h+\ovV\cdot \nabla (\bom_\mu)_h-(1+z)\nabla\cdot
\ovV\dz (\bom_\mu)_h=(\bom_\mu)_h\cdot \nabla \ovV-(\nabla^\perp\cdot \ovV)(\bom_\mu)_h^\perp+O(\sqrt{\mu});
$$
integrating this equation then gives the following equation for $Q$,
\begin{equation}\label{eqQ}
\dt
Q+Q\cdot \nabla \ovV+\ovV\cdot\nabla Q=O(\sqrt{\mu}).
\end{equation}

Using \eqref{structU} again, we can relate $\Vp^\mu$ to the
averaged velocity $\overline{V}$ through the approximation
\begin{align*}
\Vp^\mu&=\uV+\uw \nabla\zeta\\
&=\overline{V}-\sqrt{\mu} Q +O(\mu)
\quad\mbox{ where }\quad Q:=\frac{1}{h}\int_{-1}^\zeta\int_{z'}^\zeta ({\bom}_\mu)_h^\perp.
\end{align*}
This approximation is then plugged into \eqref{eqVpND} to obtain
\begin{align}
\nonumber
\dsp \dt \ovV+\ovV\cdot \nabla \ovV+\nabla\zeta&=\sqrt{\mu}\big[\dt
Q+Q\cdot \nabla \ovV+\ovV\cdot\nabla Q\big]+O(\mu)\\
\label{eqVr}
&=O(\mu),
\end{align}
the second identity stemming from \eqref{eqQ}.

Using the exact relation $\uU^\mu\cdot N^\mu=-\mu\nabla\cdot(h\ovV)$
in the equation for the surface elevation, we get moreover
\begin{equation}\label{eqkin2}
\dt \zeta+\nabla\cdot(h\ovV)=0.
\end{equation}
The Nonlinear Shallow Water model with vorticity
\eqref{NSW}-\eqref{NSWcor} is corresponds therefore to \eqref{eqQ},
\eqref{eqVr} and \eqref{eqkin2} without all the terms of order $O(\mu)$.

\subsubsection{Justification of the model}

Let us denote by $(\zeta_{\rm SW},\ovV_{\rm
  SW},Q_{\rm SW})$  the exact solution  to \eqref{NSW}-\eqref{NSWcor}
with initial conditions
\begin{equation}\label{CINSW}
\zeta_{\rm SW}^0=\zeta^0,\quad \ovV_{\rm
  SW}^0=\frac{1}{1+\zeta^0}\int_{-1}^{\zeta^0}{\bf V}^0,\quad
Q_{\rm SW}^0=\frac{1}{1+\zeta^0}\int_{-1}^{\zeta^0}\int_{z'}^{\zeta^0} ({\bom}_\mu^0)_h^\perp.
\end{equation}
The following proposition shows that $(\zeta_{\rm SW},\ovV_{\rm
  SW},Q_{\rm SW})$ is a good approximation\footnote{In the statement
  of the theorem, the quantities $\ovV$ and $Q$ are given by
$$
\ovV=\frac{1}{1+\zeta}\int_{-1}^{\zeta} {\mathbb V}[\zeta](\psi,\bom_\mu),\quad
Q=\frac{1}{1+\zeta}\int_{-1}^{\zeta}\int_{z'}^{\zeta} ({\bom}_\mu)_h^\perp,
$$
with $\bom_\mu=\omega_\mu\circ \Sigma^{-1}$ and ${\mathbb
  V}[\zeta](\psi,\bom_\mu)$ the horizontal component of the ${\mathbb
  U}[\zeta](\psi,\bom_\mu)$, as given in Definition \ref{defimappings}.}
at order $O(\mu)$ to the
full water waves equations \eqref{ZCSgenNDS}.
\begin{proposition}
Let $N\in \N$ be large enough, $\eps=1$ and $\mu\in (0,1)$. Let
$(\zeta^0,\psi^0,\omega_\mu^0)$ be such that the assumptions of
Theorem \ref{theomainND} are satisfied.\\
There exists $T>0$ (independent of $\mu$) such that \\
{\bf i.} There exists a
unique solution
$(\zeta_{\rm SW},\ovV_{\rm SW},Q_{\rm SW}) \in C([0,T];H^N(\R^d)\times
H^N(\R^d)^2\times H^{N-1}(\R^d)^d)$ to \eqref{NSW}-\eqref{NSWcor} with
initial condition \eqref{CINSW};\\
{\bf ii.} There exists
a unique solution $(\zeta,\psi,\omega_\mu)\in E^N_T$ to
\eqref{ZCSgenNDS} with initial data $(\zeta^0,\psi^0,\omega_\mu^0)$;\\
{\bf iii.} The following error estimates hold,
$$
\abs{\zeta-\zeta_{\rm SW}}_{L^\infty([0,T]\times \R^d)}+\abs{\ovV-\ovV_{\rm
    SW}}_{L^\infty([0,T]\times \R^d)}+\sqrt{\mu}\abs{Q-Q_{\rm SW}}_{L^\infty([0,T]\times \R^d)}
\leq \mu c,
$$
with $\dsp c=C\big({\mathcal
  E}^N(\zeta^0,\psi^0,\omega_\mu^0),\frac{1}{h_{\rm min}},\frac{1}{a_0}\big)$.
\end{proposition}
\begin{proof}
The first point of the proposition is classical and stems directly
from the hyperbolic structure of \eqref{NSW}, and the second point is
a direct consequence of Theorem \ref{theomainND}. For the third point,
we have shown in \S \ref{sectderivSW} that $(\zeta,\ovV,Q)$ solves
\eqref{NSW}-\eqref{NSWcor}
up to $O(\mu)$ terms. Standard hyperbolic estimates then give a
$O(\mu)$ control of the error in $H^{2}$-norm (provided than $N$ is
chosen large enough), from which the $L^\infty$-estimate given in the
statement of the lemma follows from Sobolev embeddings.
\end{proof}

\section{A Hamiltonian formulation of the water waves equations with
  vorticity \eqref{ZCSgen}}\label{sect6}

The total energy  is given by the sum of the potential energy $E$ and
the kinetic energy $K$,
\begin{eqnarray*}
H&=&E+K\\
&=&\frac{1}{2}\int_{\R^d}g\zeta^2+\frac{1}{2}\int_{\Omega_\zeta}\abs{\bU}^2,
\end{eqnarray*}
where we always assume that $\zeta$ satisfies the nonvanishing depth
condition \eqref{hmin} and where $\Omega_\zeta$ is the fluid domain
delimited above by the graph of $\zeta$ and below by the flat bottom
$z=-H_0$ (when no confusion is possible, we simply write
$\Omega=\Omega_\zeta$ as everywhere else in this paper).
In \cite{Zakharov}, Zakharov showed that, {\it in the irrotational
  case}, the water waves equations could be formally written under a
canonical Hamiltonian formulation, namely,
\begin{equation}\label{Zakham}
\dt \left(\begin{array}{c}\zeta \\ \psi\end{array}\right)=J \mbox{\rm
  grad}_{\zeta,\psi}H,
\quad\mbox{ with }\quad J=\left( \begin{array}{cc} 0 & 1 \\ -1 & 0\end{array}\right).
\end{equation}
Several authors have proposed formulations of the water waves
equations in presence of vorticity that also have a Hamiltonian
structure (but without addressing the well-posedness of these formulations). Let us mention for instance \cite{MP,Constantin0} in a
Lagrangian framework, \cite{Wahlen} in an Eulerian framework for one
dimensional flows with constant vorticity, and the general approach of
\cite{LMM} (see also \cite{KolevSattinger}, and \cite{Kolev} for
comments on
the validity of these formulations). We investigate in this
section if our new well-posed formulation \eqref{ZCSgen} has a
structure similar to \eqref{Zakham} when the vorticity is non zero.\\
We first define admissible functionals in \S \ref{sectAF}, show how to
compute the gradients of such functionals in \S \ref{sectgradAF}, and
finally show in \S \ref{sectPoisson} that the formulation
\eqref{ZCSgen} has a formal Hamiltonian structure.

\subsection{The set of admissible functionals}\label{sectAF}

With the notations introduced in Definition \ref{defimappings}, we
can write the energy as a function of  $(\zeta,\psi,\bom)$,
\begin{eqnarray}
\nonumber
H=H(\zeta,\psi,\bom)&=&\frac{1}{2}\int_{\R^d}
g\zeta^2+\frac{1}{2}\int_{\Omega}\abs{{\mathbb
    U}(\zeta,\psi,\bom)}^2\\
\label{totalNRJ}
&=:&{\mathcal H}(\zeta,{\mathbb
    U}(\zeta,\psi,\bom)),
\end{eqnarray}
where $H$ and ${\mathcal H}$ are respectively functionals on ${\mathcal
  M}$ and ${\mathcal N}$, defined as follows.
\begin{defi}
{\bf i.} We denote
\begin{align*}
{\mathcal M}=\big\{(\zeta,\psi,\bom), &(\zeta,\psi)\in H^\infty(\R^d)^2,
\zeta \,\mbox{\rm
  satisfies }\eqref{hmin}, \\
&\bom\in H^\infty(\Omega_\zeta)^3,\, \dive
\bom=0, \, \omega_b\cdot N_b\in H_0^\infty(\R^d)\big\}.
\end{align*}
{\bf ii.} We denote
\begin{align*}
{\mathcal N}=\{(\zeta,\bU), &\zeta\in H^\infty(\R^d) \quad \mbox{\rm
  satisfies }\eqref{hmin}, \\
&\bU\in H^\infty(\Omega_\zeta)^3, \,\dive \bU=0,\, U_b\cdot N_b=0\}.
\end{align*}
\end{defi}
Denoting $\sigma_\zeta=\frac{1}{H_0}(z+H_0)\zeta$, and recalling that
$\mbox{div}^{\sigma_\zeta}$ is defined in \eqref{notstar}, we also define the
straightened version of ${\mathcal N}$ by
$$
{\mathcal N}^\sigma=\{(\zeta,U), \zeta\in H^\infty(\R^d) \quad \mbox{\rm
  satisfies }\eqref{hmin}, U\in H^\infty(\cS)^3,\, {\mbox{\rm
    div}}^{\sigma_\zeta}U=0,\, U_b\cdot N_b=0\}.
$$
Any functional ${\mathcal F}$ on ${\mathcal N}$ can be equivalently
defined as a functional ${\mathcal F}^\sigma$ on ${\mathcal N}^\sigma$
through the relation
$$
{\mathcal F}^\sigma(\zeta, U)={\mathcal F}(\zeta,U\circ \Sigma_\zeta^{-1})
\quad\mbox{ with }\quad \Sigma_\zeta(X,z)=(X,z+\sigma_\zeta(X,z));
$$
we use this observation to define the class $C^\infty({\mathcal N})$
of smooth functionals on ${\mathcal N}$.
\begin{defi}
 A functional ${\mathcal F}$ belongs to $C^\infty({\mathcal N})$ if
 and only if ${\mathcal F}^\sigma$ belongs to $C^\infty({\mathcal N}^\sigma)$.
\end{defi}
We can also use this observation to define G\^ateaux-derivatives of
functional  $C^\infty({\mathcal N})$; the following assumption is made
on these derivatives:
\begin{equation}\label{hypder}
\left\lbrace
\begin{array}{l}
\dsp \exists \,\frac{\delta {\mathcal F}}{\delta \zeta}\in
H^\infty(\R^d),\, \forall \,\delta\zeta \in H^\infty(\R^d),
\quad d_\zeta {\mathcal F}\cdot \delta \zeta=\int_{\R^d} \frac{\delta
  {\mathcal F}}{\delta \zeta} \delta\zeta,\vspace{1mm}\\
\dsp \exists \,\frac{\delta {\mathcal F}}{\delta \bU}\in
H^\infty(\Omega_\zeta)^{3},\, \dive \frac{\delta {\mathcal
    F}}{\delta \bU}=0,\,
{\frac{\delta {\mathcal F}}{\delta \bU}} {\vert_{z=-H_0}}\cdot N_b=0,\vspace{1mm}\\
\hspace{1.3cm}\forall \,\delta\bU \in H^\infty(\Omega_\zeta)^{3}\,
\mbox{\rm such that }
\dive \delta\bU=0,\,(\delta\bU)_{\vert_{z=-H_0}}\!\!\cdot N_b=0,\vspace{1mm}\\
\hspace{1.3cm} \mbox{\rm one has } d_\bU {\mathcal F}\cdot \delta \bU=\int_{\Omega_\zeta} \frac{\delta
  {\mathcal F}}{\delta \bU} \cdot \delta\bU.
\end{array}\right.
\end{equation}
We can finally define the class ${\mathcal A}$ of
admissible functionals to which the Hamiltonian $H$ belongs.
\begin{defi}\label{deffunctadm}
A functional $F$ on ${\mathcal M}$ belongs to the set ${\mathcal A}$
of \emph{admissible} functionals  if and only if there exists ${\mathcal F}\in C^\infty({\mathcal N})$
satisfying \eqref{hypder}, and such that
$$
 \forall (\zeta,\psi,\bom)\in {\mathcal M}, \quad
 F(\zeta,\psi,\bom)={\mathcal F}(\zeta,{\mathbb U}[\zeta](\psi,\bom)).
$$
\end{defi}

\subsection{Gradients of admissible functionals}\label{sectgradAF}

We give in the following proposition an expression for the gradient of
admissible functionals, as well as an expression for the cotangent
bundle
$T^*{\mathcal M}$. Note that the tangent space
$T_{\zeta,\psi,\bom}{\mathcal M}$, in which we take the variations
$(\delta\zeta,\delta\psi,\delta\bom)$, is defined\footnote{The fact
  that the variations $\delta\zeta$ are taken in $H_0^\infty(\R^d)$ is
to ensure the conservation of the volume of the fluid domain;
consequently, the variations $\delta\psi$ are taken in $\dot{H}^\infty(\R^d)$.} as
\begin{align*}
T_{\zeta,\psi,\bom}{\mathcal M}=\{(\delta\zeta,\delta\psi,\delta\bom)\in
{H}_0^\infty(\R^d)&\times \dot{H}^{\infty}(\R^d)\times
H^\infty(\Omega_\zeta)^3, \\
&\dive \delta\bom=0,\, (\delta\omega)_b\cdot N_b\in H_0^\infty(\R^d)\};
\end{align*}
we also recall that the operator $\mbox{\rm curl}^{-1}$ is
defined in Corollary \ref{invertcurl}.
\begin{proposition}\label{propgrad}
Let $F$ be an admissible functional and ${\mathcal F}$ be the
associated functional on ${\mathcal N}$.
One can write, for all variations
$(\delta\zeta,\delta\psi,\delta\bom)\in T_{\zeta,\psi,\bom}{\mathcal
  M}$,
\begin{align*}
\big\langle d_{\zeta,\psi,\bom} F, \left(\!\!\begin{array}{c}\delta \zeta\\
    \delta \psi \\
    \delta \bom \end{array}\!\!\right)\big\rangle_{T^*_{\zeta,\psi,\bom}{\mathcal
  M}-T_{\zeta,\psi,\bom}{\mathcal
  M}}&=\int_{\R^d}\frac{\delta F}{\delta
  \zeta}\delta \zeta+\int_{\R^d}\frac{\delta F}{\delta
  \psi}\delta \psi+
\int_{\Omega_\zeta}\frac{\delta F}{\delta
  \bom}\cdot \delta \bom\\
&=\big(\mbox{\rm grad}_{\zeta,\psi,\bom}\,F,\left(\!\!\begin{array}{c}\delta \zeta \\
    \delta \psi \\
    \delta \bom \end{array}\!\!\right)\big)_{L^2(\R^d)\times
  L^2(\R^d)\times L^2(\Omega_\zeta)}
\end{align*}
with the $L^2$-gradient $\mbox{\rm grad}_{\zeta,\psi,\bom}\,F$ given by
$$
\mbox{\rm grad}_{\zeta,\psi,\bom}\,F:=\left(\begin{array}{c}
\dsp \frac{\delta F}{\delta
  \zeta}\\
\dsp \frac{\delta F}{\delta
  \psi}\\
\dsp \frac{\delta F}{\delta
  \bom}\end{array}\right)=
\left(\begin{array}{c}
\dsp \frac{\delta {\mathcal F}}{\delta \zeta}-\uw \frac{\delta
  {\mathcal F}}{\delta \bU}\cdot
N-\uom_h^\perp\cdot\nabla\Delta^{-1}\frac{\delta {\mathcal F}}{\delta
  \bU}\cdot N\\
\dsp \frac{\delta {\mathcal F}}{\delta \bU}{\vert_{z=\zeta}}\cdot
N\\
\dsp \mbox{\rm curl}^{-1}\frac{\delta {\mathcal F}}{\delta
  \bU}.
\end{array}\right).
$$
Identifying the cotangent space with the set of the $L^2$-gradients of
all the admissible functionals, one has moreover
\begin{align*}
T^*_{\zeta,\psi,\bom}{\mathcal M}=\{(a,b,{\bf C})\in H^\infty(\R^d)&\times
H^\infty(\R^d)\times H^\infty(\Omega_\zeta)^3, \\
&\dive{\bC}=0,\,\nabla^\perp \cdot
C_\parallel=b,\, C_b=0\}.
\end{align*}
\end{proposition}
\begin{proof}
Let us first consider the derivative with respect to $\psi$. Since $F$ is admissible, one has
$$
F(\zeta,\psi,\bom)={\mathcal F}(\zeta,{\mathbb U}[\zeta](\psi,\bom))
$$
with ${\mathcal F}\in C^\infty({\mathcal N})$ satisfying
\eqref{hypder}, and therefore
\begin{eqnarray*}
d_\psi F\cdot \delta\psi&=& d_U{\mathcal F}\cdot \big(d_\psi {\mathbb
  U}[\zeta](\cdot,\bom)\cdot \delta \psi\big)\\
&=&\int_\Omega \frac{\delta {\mathcal F}}{\delta \bU}\cdot {\mathbb
  U}_I[\zeta]\delta \psi\\
&=&\int_{\R^d}\frac{\delta {\mathcal F}}{\delta \bU}{\vert_{z=\zeta}}\cdot N\delta \psi,
\end{eqnarray*}
where we used the definition of $ {\mathbb
  U}_I[\zeta]\delta \psi$ and the fact that $\frac{\delta {\mathcal
    F}}{\delta \bU}$ is divergence free and that its normal component
vanishes at the bottom.\\
For the derivative with respect to $\bom$, one proceeds
along the same lines as above to get
\begin{eqnarray*}
d_\bom F\cdot \delta\bom&=&\int_\Omega \frac{\delta {\mathcal F}}{\delta \bU}\cdot {\mathbb
  U}_{II}[\zeta]\delta\bom
\end{eqnarray*}
Since $\frac{\delta {\mathcal F}}{\delta
  \bU} $ is divergence free and has zero normal component at the
bottom, we can use Corollary \ref{invertcurl} and Green's identity and write
\begin{eqnarray*}
d_\bom F\cdot \delta\bom
&=&\int_\Omega \mbox{\rm curl}^{-1}\frac{\delta {\mathcal F}}{\delta
  \bU}\cdot \delta\bom
+\int_{\R^d} N\times (\mbox{\rm curl}^{-1}\frac{\delta {\mathcal
    F}}{\delta \bU}
)\cdot ({\mathbb U}_{II}[\zeta]\delta\bom)_\surf\\
&=&\int_\Omega \mbox{\rm curl}^{-1}\frac{\delta {\mathcal F}}{\delta
  \bU}\cdot \delta\bom
+\int_{\R^d}  (\mbox{\rm curl}^{-1}\frac{\delta {\mathcal
    F}}{\delta \bU}
)_\parallel^\perp\cdot ({\mathbb U}_{II}[\zeta]\delta\bom)_\parallel.
\end{eqnarray*}
The second term in the above expression vanishes because $(\mbox{\rm curl}^{-1}\frac{\delta {\mathcal
    F}}{\delta \bU}
)_\parallel^\perp=-\nabla\Delta^{-1}(\underline{C}\cdot N)$, which is
$L^2$-orthogonal to $({\mathbb
  U}_{II}[\zeta]\delta\bom)_\parallel=\nabla^\perp
\Delta^{-1}(\delta\omega_\surf\cdot N)$, and we therefore have
\begin{eqnarray*}
d_\bom F\cdot \delta\bom
&=&\int_\Omega \mbox{\rm curl}^{-1}\frac{\delta {\mathcal F}}{\delta
  \bU}\cdot \delta\bom.
\end{eqnarray*}
Finally, for the derivative with respect to $\zeta$, we get
\begin{eqnarray*}
d_\psi F\cdot \delta\psi&=& d_\zeta {\mathcal F}\cdot \delta \zeta+ d_U{\mathcal F}\cdot \big(d_\zeta {\mathbb
  U}[\cdot](\psi,\bom)\cdot \delta \zeta\big)\\
&=&\int_{\R^d} \frac{\delta {\mathcal F}}{\delta \zeta}\delta \zeta+\int_\Omega \frac{\delta {\mathcal F}}{\delta \bU}\cdot{\mathbb
  U}_I[\zeta](-\underline{w}\delta
\zeta+\frac{\nabla}{\Delta}\cdot(\uom_h^\perp\delta\zeta)),
\end{eqnarray*}
where we used\footnote{Proposition \ref{propshapeU} actually gives a
  formula for the time derivative of ${\mathbb U}[\zeta](\psi,\bom)$; the shape derivative formula used here
is obtained exactly in the same way.} Proposition \ref{propshapeU} to substitute $d_\zeta {\mathbb
  U}[\cdot](\psi,\bom)\cdot \delta \zeta={\mathbb
  U}_I[\zeta](-\underline{w}\delta
\zeta+\frac{\nabla}{\Delta}\cdot(\uom_h^\perp\delta\zeta))$ in the
second term. Since $\frac{\delta {\mathcal F}}{\delta
  \bU} $ is divergence free and has zero normal component at the
bottom, we can use Green's identity to get
\begin{eqnarray*}
d_\psi F\cdot \delta\psi
&=&\int_{\R^d} \frac{\delta {\mathcal F}}{\delta \zeta}\delta
\zeta+\int_{\R^d} \frac{\delta {\mathcal F}}{\delta \bU}\cdot N
(-\underline{w}\delta
\zeta+\frac{\nabla}{\Delta}\cdot(\uom_h^\perp\delta\zeta))\\
&=&\int_{\R^d}\big(\frac{\delta {\mathcal F}}{\delta \zeta}-\uw \frac{\delta {\mathcal F}}{\delta \bU}\cdot N-\uom_h^\perp\cdot\nabla\Delta^{-1}\frac{\delta {\mathcal F}}{\delta \bU}\cdot N\big)\delta\zeta,
\end{eqnarray*}
and the expression for $\mbox{\rm grad}_{\zeta,\psi,\bom}F$ given in
the statement of the proposition follows easily.\\
Recalling the identity $\nabla^\perp\cdot C_\parallel=(\curl{\bf
  C})_\surf\cdot N$, one immediatly deduces that
$$
T^*_{\zeta,\psi,\bom}{\mathcal M}\subset\{(a,b,{\bf C})\in H^\infty(\R^d)\times
H^\infty(\R^d)\times H^\infty(\Omega_\zeta), \, \nabla^\perp \cdot
C_\parallel=b,\, C_b=0\}.
$$
In order to prove the reverse inclusion, we need, for all
$(a,b,{\bf C})$ satisfying the condition defining the space in the
right part of the identity, to construct a functional ${\mathcal F}\in
C^\infty({\mathcal N})$ satisfying \eqref{hypder}, and such that
$$
\frac{\delta{\mathcal F}}{\delta \bU}(\zeta,\bU)=\curl{\bf C},\qquad
\frac{\delta{\mathcal F}}{\delta
  \zeta}(\zeta,\bU)=a+\underline{w}b-\omega_h^\perp\cdot \nabla\Delta^{-1}b.
$$
This is achieved by taking ${\mathcal F}$ defined as
$$
{\mathcal
  F}(\zeta',\bU')=\int_{\Omega_{\zeta'}}\curl{\bf C}\cdot
\bU'+\int_{\R^d}\big(a+\underline{w}b-\omega_h^\perp\cdot
\nabla\Delta^{-1}b-(\curl{\bf C})_{\vert_{z=\zeta}}\cdot \underline{U}\big)\zeta',
$$
for all $(\zeta',\bU')\in {\mathcal N}$.
\end{proof}
This proposition can be used to compute the gradient of the total
energy \eqref{totalNRJ}.
\begin{cor}
Let $H$ be the functional on ${\mathcal M}$ associated to the total
energy \eqref{totalNRJ}. One has
$$
\mbox{\rm grad}_{\zeta,\psi,\bom}H=\left(
\begin{array}{c}
    g\zeta+\frac{1}{2}\abs{\Vp}^2-\frac{1}{2}(1+\abs{\nabla\zeta}^2)\underline{w}^2-\underline{\omega}_h^\perp\cdot
    \nabla\Delta^{-1} \uU\cdot N\\
 \uU\cdot N \\
\mbox{\rm curl}^{-1}\bU
\end{array}\right)
$$
\end{cor}
\begin{proof}
The functional $H$ is admissible, with associated function ${\mathcal
  H}$ on ${\mathcal N}$ given by
$$
{\mathcal H}(\zeta,\bU)=\frac{1}{2}\int_{\R^d} g\zeta^2+\frac{1}{2}\int_{\Omega_\zeta}\abs{\bU}^2.
$$
The result follows from Proposition \ref{propgrad} and the observation
that
$$
\frac{\delta {\mathcal H}}{\delta
  \zeta}=g\zeta+\frac{1}{2}\abs{\uU}^2,\qquad
\frac{\delta {\mathcal H}}{\delta
  \bU}=\bU.
$$
\end{proof}

\subsection{The Poisson bracket and the Hamiltonian formulation}\label{sectPoisson}

As said above,  Zakharov showed that the irrotational water waves
equations \eqref{ZCS} can be written under the form
$$
\dt \left(\begin{array}{c}\zeta \\ \psi\end{array}\right)=J \mbox{\rm
  grad}_{\zeta,\psi}H,
\quad\mbox{ with }\quad J=\left( \begin{array}{cc} 0 & 1 \\ -1 & 0\end{array}\right);
$$
in particular, $J$ is an antisymmetric operator on
$H_0^\infty(\R^d)\times \dot{H}^\infty(\R^d)$. In order to generalize
this result to the rotational case, we need to define the notion of
antisymmetric operator on the cotangent bundle $T^*{\mathcal M}$.
\begin{defi}
A mapping $J: T^*{\mathcal M}\to
T{\mathcal M}$ is \emph{antisymmetric} if on each fiber of the
cotangent bundle, the bilinear mapping
$$
\begin{array}{lcl}
T_{\zeta,\psi,\bom}^*{\mathcal M} \times T_{\zeta,\psi,\bom}^*{\mathcal M} & \to& \R\\
(d_{\zeta,\psi,\bom}F,d_{\zeta,\psi,\bom}G) &\mapsto&  \big( \mbox{\rm
grad}_{\zeta,\psi,\bom}F , J_{\zeta,\psi,\bom} \mbox{\rm
grad}_{\zeta,\psi,\bom}G\big)_{L^2(\R^d)\times L^2(\R^d)\times L^2(\Omega_\zeta)}
\end{array}
$$
is antisymmetric.
\end{defi}
We can now state the following generalization of Zakharov's result.
\begin{thm}\label{theoham}
The water waves equations \eqref{ZCSgen} can be written
$$
\dt \left(\begin{array}{c}
\zeta\\
\psi \\
\bom
\end{array}\right)=J_{\zeta,\psi,\bom}\mbox{\rm grad}_{\zeta,\psi,\bom}H.
$$
with
$$
J_{\zeta,\psi,\bom}=\left(\begin{array}{ccc}
0 & 1 & 0\\
-1 & \big(\uom_h\!\cdot\!
\frac{\nabla^\perp}{\Delta}\bullet+\frac{\nabla^\perp}{\Delta}\!\cdot\!(\uom_h\bullet)\big)
& \frac{\nabla^\perp}{\Delta}\!\cdot\!\big(\uom_h (\curl \bullet)\!\cdot\!
N-\uom\!\cdot\! N (\curl\bullet)_h\big)\\
0 & 0 & \curl\big(\bom\times \curl \bullet\big)
\end{array}\right);
$$
the field of linear mappings
$J=(J_{\zeta,\psi,\bom})_{(\zeta,\psi,\bom)\in {\mathcal M}}:
T^*{\mathcal M}\to T{\mathcal M}$ is antisymmetric.
\end{thm}

\begin{proof}
Let us define $J^0_\bom$ by
$$
J^0_\bom=\left(\begin{array}{cc}
0 & 1 \\
-1 & \big(\uom_h\cdot
\frac{\nabla^\perp}{\Delta}\bullet+\frac{\nabla^\perp}{\Delta}\cdot(\uom_h\bullet)
\big)
\end{array}\right).
$$
We have therefore for all admissible functionals $F,G\in {\mathcal A}$
(and writing $\delta_\zeta { F}=\frac{\delta {
    F}}{\delta \zeta}$ etc.),
\begin{align}
\nonumber
\big( \mbox{\rm
grad}_{\zeta,\psi,\bom}F& , J_{\zeta,\psi,\bom} \mbox{\rm
grad}_{\zeta,\psi,\bom}G\big)_{L^2\times L^2\times L^2}
=\Big(\left(\begin{array}{c} \delta_\zeta { F} \\ \delta_\psi { F} \end{array}\right), J^0_{\zeta,\bom} \left(\begin{array}{c}
    \delta_\zeta { G} \\ \delta_\psi { G}\end{array}\right)   \Big)_{L^2\times L^2}\\
\nonumber
&+\Big(\delta_\psi {
  F},\frac{\nabla^\perp}{\Delta}\cdot\big(\uom_h (\curl
\delta_{\bom}{ G})_\surf\cdot
N-\uom\cdot N (\curl \delta_{\bom}{ G})_{\surf,h}\big)\Big)_{L^2(\R^d)}\\
\label{west}
&+\Big(\delta_{\bom}{ F},\curl(\bom\times \curl \delta_{\bom}{ G})\Big)_{L^2(\Omega)}.
\end{align}
Focusing our attention on the last two terms of the right hand side,
we remark first that
\begin{align*}
\Big(\delta_{\psi}{ F},\frac{\nabla^\perp}{\Delta}&\cdot\big(\uom_h (\curl \delta_{\bom}{ G})_\surf\cdot
N-\uom\cdot N (\curl \delta_{\bom}{ G})_{\surf,h}\big)\Big)\\
=&-\Big(\frac{\nabla^\perp}{\Delta} \delta_{\psi}{ F},\uom_h (\curl \delta_{\bom}{ G})_\surf\cdot
N-\uom\cdot N (\curl \delta_{\bom}{ G})_{\surf,h}\Big)\\
=&\Big({\curl \delta_{\bom}{ G}}_\surf,\uom\times \left(\begin{array}{c}
\frac{\nabla}{\Delta}\delta_{\psi}{ F} \\
\frac{\nabla}{\Delta}\delta_{\psi}{ F}\cdot \nabla\zeta
\end{array}\right)\Big).
\end{align*}
For the last term of \eqref{west}, we use Green's identity to get
\begin{align*}
\Big(&\delta_{\bom}{ F},\curl(\bom\times \curl \delta_{\bom}{ G})\Big)_{L^2(\Omega)}\\
&=
\Big(\curl \delta_{\bom}{ F},\bom\times \curl \delta_{\bom}{ G}\Big)_{L^2(\Omega)}\!\!\!+\big(({\curl \delta_{\bom}{ G}})_\surf,\uom\times(N\times (\delta_{\bom}{ F}))_\surf\big)_{L^2(\R^d)}\\
&=
\Big(\curl \delta_{\bom}{ F},\bom\times \curl \delta_{\bom}{ G}\Big)_{L^2(\Omega)}\!\!\!+\big({(\curl \delta_{\bom}{ G}})_\surf,\uom\times
 \left(\!\!\begin{array}{c}
(\delta_{\bom}{ F})^\perp_{\parallel}\\
(\delta_{\bom}{ F})^\perp_{\parallel}\cdot \nabla\zeta
\end{array}\!\!\right)
\big)_{L^2(\R^d)}.
\end{align*}
We therefore get from \eqref{west} that
\begin{align}
\nonumber
\Big( \left(\begin{array}{c} \delta_\zeta{ F} \\ \delta_\psi{ F}  \\ \delta_\bom{ F} \end{array}\right),& J_{\zeta,\psi,\bom} \left(\begin{array}{c}
    \delta_\zeta{ G}  \\ \delta_\psi{ G}  \\ \delta_\bom{ G} \end{array}\right)   \Big)_{L^2\times
  L^2\times L^2}\\
\nonumber
=&\Big(\left(\begin{array}{c} \delta_\zeta{ F}  \\ \delta_\psi{ F}  \end{array}\right), J^0_{\zeta,\bom} \left(\begin{array}{c}
    \delta_\zeta{ G}  \\ \delta_\psi{ G} \end{array}\right)   \Big)_{L^2\times L^2}
+
\Big(\curl \delta_\bom{ F} ,\bom\times \curl \delta_\bom{ G} \Big)_{L^2(\Omega)}\\
\nonumber
&+\Big((\curl \delta_\bom{ G})_\surf,\uom\times
 \left(\begin{array}{c}
(\delta_\bom{ F} )^\perp_{\parallel}+\frac{\nabla}{\Delta}\delta_\psi{ F} \\
((\delta_\bom{ F} )^\perp_{\parallel}+\frac{\nabla}{\Delta}\delta_\psi{ F} )\cdot \nabla\zeta
\end{array}\right)
\Big)_{L^2(\R^d)}.
\end{align}
Now, the assumption that ${ F}\in { A}$ implies by
Proposition \ref{propgrad} that the last term vanishes, so that
\begin{align}
\nonumber
\Big( \left(\begin{array}{c} \delta_\zeta{ F} \\ \delta_\psi{ F}  \\ \delta_\bom{ F} \end{array}\right),& J_{\zeta,\psi,\bom} \left(\begin{array}{c}
    \delta_\zeta{ G}  \\ \delta_\psi{ G}  \\ \delta_\bom{ G} \end{array}\right)   \Big)_{L^2\times
  L^2\times L^2}\\
\label{west2}
=&\Big(\left(\begin{array}{c} \delta_\zeta{ F}  \\ \delta_\psi{ F}  \end{array}\right), J^0_{\zeta,\bom} \left(\begin{array}{c}
    \delta_\zeta{ G}  \\ \delta_\psi{ G} \end{array}\right)   \Big)_{L^2\times L^2}
+
\Big(\curl \delta_\bom{ F} ,\bom\times \curl
\delta_\bom{ G} \Big)_{L^2(\Omega)}.
\end{align}
Since moreover $J^0_{\zeta,\bom}$ is obviously skew-symmetric for the $L^2(\R^d)\times L^2(\R^d)$
scalar product, the result follows directly.
\end{proof}
We can now deduce the following corollary that shows that the water
waves equations with vorticity can be formally written in Hamiltonian
form.
\begin{cor}
The water waves equations \eqref{ZCSgen} are equivalent to the
Hamiltonian equation
$$
\forall F\in {\mathcal A}, \qquad \dot{F}=\{F,H\},
$$
where $H$ is the Hamiltonian \eqref{totalNRJ}, while the {\it Poisson
  bracket} $\{\cdot,\cdot\}$ is defined as
\begin{align*}
\{F,G\}=&\int_{\R^d}
\frac{\delta F}{\delta \zeta} \frac{\delta
      G}{\delta \psi} -
\frac{\delta F}{\delta \psi} \frac{\delta
      G}{\delta \zeta}-\int_{\R^d} \uom_h\cdot \big[ \frac{\delta F}{\delta
  \psi}\frac{\nabla^\perp}{\Delta}\frac{\delta G}{\delta \psi}-
\frac{\delta G}{\delta \psi}\frac{\nabla^\perp}{\Delta}\frac{\delta F}{\delta \psi}\big]\\
&+\int_{\Omega} (\curl \frac{\delta F}{\delta \bom})\cdot (\bom \times
\curl \frac{\delta G}{\delta \bom}),
\end{align*}
for all $F,G\in {\mathcal A}$.
\end{cor}
\begin{remark}
As said above, the Hamiltonian formulation derived above is only
formal. In order to obtain a {\it valid} Hamiltonian structure \cite{Kolev}, one
must also prove that the Poisson bracket satisfies Jacobi's identity,
and that it is closed (i.e. that for all $F,G\in {\mathcal A}$,
${F,G}$ is also an admissible functional). Checking these points is
left for future work. Note that it is proved in \cite{Kolev} that the
Poisson brackets derived in \cite{LMM} are not valid; actually, even
in the irrotational case, it does not seem to be known whether
Zakharov's formulation \eqref{Zakham} provides a {\it valid}
Hamiltonian structure.
\end{remark}
\begin{proof}
The fact that if $(\zeta,\psi,\bom)$ solves \eqref{ZCSgen} implies
that the Hamiltonian equation is satisfied follows directly from
Theorem \ref{theoham} after remarking that \eqref{west2} corresponds exactly
to the Poisson bracket.\\
Conversely, if the Hamiltonian equation is satisfied for all admissible
functional $F$, one deduces from Theorem \ref{theoham} that
$$
\forall F\in {\mathcal A}, \qquad \big(\mbox{\rm
  grad}_{\zeta,\psi,\bom} F,\left(\begin{array}{c} \dt \zeta\\ \dt \psi
    \\ \dt \bom \end{array}\right)-J_{\zeta,\psi,\bom}\mbox{\rm
  grad}_{\zeta,\psi,\bom} H \big)_{L^2\times L^2\times L^2}=0.
$$
Using the last point of Proposition \ref{propgrad}, one readily
deduces that
$$
\left(\begin{array}{c} \dt \zeta\\ \dt \psi
    \\ \dt \bom \end{array}\right)-J_{\zeta,\psi,\bom}\mbox{\rm
  grad}_{\zeta,\psi,\bom} H=0,
$$
and the result is proved.
\end{proof}

\bigbreak
\noindent
{\bf Acknowledgment.} The authors warmly thank B. Kolev for his help
on the various hamiltonian structures encountered in fluid mechanics. A.C is support by the grant MTM2011-266696 (Spain), ICMAT Severo Ochoa project SEV-2011-0087 and ERC grant 307179-GFTIPFD. D. L. acknowledges support from the ANR-13-BS01-0003-01
  DYFICOLTI, the ANR BOND, and the INSU-CNRS LEFE-MANU project Soli.

\end{document}